\documentclass[12pt, leqno]{amsart}
\usepackage{latexsym}
\usepackage{amssymb}
\usepackage{amsmath}
\usepackage[english]{babel}
\usepackage[colorlinks=true, pdfstartview=FitV, linkcolor=red, citecolor=red, urlcolor=blue, backref=page]{hyperref} 

\usepackage[top=1.25in, bottom=1.25in, left=1.25in, right=1.25in]{geometry}

\newtheorem{theorem}{Theorem}

\newtheorem{corollary}[theorem]{Corollary}

\newtheorem{definition}[theorem]{Definition}

\newtheorem{lemma}[theorem]{Lemma}

\newtheorem{proposition}[theorem]{Proposition}
\newtheorem{remark}[theorem]{Remark}

\newtheorem{mytheorem}{Theorem}

\newcommand{\rn}[1]{\mathbb{R}^{#1}}

\newcommand{\re}{ \mathbb{R}}

\newcommand{\beq}{\begin{equation}}
\newcommand{\bea}[1]{\begin{array}{#1} }
\newcommand{\eeq}{ \end{equation}}
\newcommand{\ea}{ \end{array}}
\newcommand{\ep}{\epsilon}
\newcommand{\es}{\emptyset}

\newcommand{\al}{\alpha}
\newcommand{\ga}{\gamma}

\newcommand{\de}{\delta}
\newcommand{\ds}{\displaystyle}
\newcommand{\ts}{\textstyle}

\newcommand{\ran}{\rangle}
\newcommand{\lan}{\langle}
\newcommand{\Ga}{\Gamma}
\newcommand{\la}{\lambda}
\newcommand{\La}{\Lambda}
\newcommand{\ar}{\partial}
\newcommand{\si}{\sigma}

\newcommand{\om}{\omega}
\newcommand{\Om}{\Omega}
\newcommand{\be}{\beta}
\newcommand{\ch}{\chi}

\newcommand{\ph}{\phi}
\newcommand{\he}{\theta}

\newcommand{\Ph}{\Phi}
\newcommand{\hs}[1]{\mbox{$ \hspace{#1}$}}

\newcommand{\sem}{\setminus}
\newcommand{\ze}{\zeta}

\newcommand{\ti}{\tilde}
\newcommand{\noi}{\noindent}

\renewcommand*{\backref}[1]{}
\renewcommand*{\backrefalt}[4]{%
 \ifcase #1 (Not cited.)%
   \or        (Cited on page~#2.)%
    \else      (Cited on pages~#2.)%
    \fi}  

\begin{document} 

\title[Failure of Fatou type theorems for Solutions to PDE of  $p$-Laplace Type]{Failure of Fatou type theorems for Solutions to PDE of  $p$-Laplace Type  in Domains with Flat Boundaries} 


\author[M. Akman]{Murat Akman}
\address{{\bf Murat Akman}\\
Department of Mathematical Sciences, University of Essex\\
Wivenhoe Park, Colchester, Essex CO4 3SQ, UK
} \email{murat.akman@essex.ac.uk}

\author[J. Lewis]{John Lewis}
\address{{\bf John Lewis} \\ Department of Mathematics \\ University of Kentucky \\ Lexington, Kentucky, 40506}
\email{johnl@uky.edu}

\author[A. Vogel]{Andrew Vogel}
\address{{\bf Andrew Vogel}\\ Department of Mathematics, Syracuse University \\  Syracuse, New York 13244}
\email{alvogel@syracuse.edu}


\begin{abstract}
  Let   $  \rn{n} $ 
   denote Euclidean $ n $ space    and  given  $k$   a positive integer let  $ \La_k  
  \subset \rn{n} $,   $ 1 \leq k < n - 1,  n \geq 3, $  be a $k$-dimensional plane with $ 0 \in \La_k.$    
  If     $p>n-k$, we  first  study  the    Martin boundary  problem  for solutions to  the   $p$-Laplace equation (called $p$-harmonic functions) in
   $ \rn{n} \sem \La_k $ relative to $ \{0\}. $  We then  use the results from our study to 
     extend  the work of Wolff in \cite{W} (on the failure of Fatou type theorems  for   $p$-harmonic functions in $ \rn{2}_+ $)  to $p$-harmonic functions  in $ \rn{n} \sem \La_k $   when $ n-k < p <\infty$.  Finally, we  discuss generalizations of our  work to  solutions of  $ p $-Laplace type   PDE  (called $ \mathcal{A}$-harmonic  functions).  
 \end{abstract}

\maketitle
\setcounter{tocdepth}{2}
\tableofcontents

\section{Introduction}  
\label{sec1}    
In 1984  Wolff brilliantly  used ideas  from  harmonic analysis and  PDE  to  prove that 
the Fatou theorem fails for $p$-harmonic functions when  $2<p<\infty$. He proved 
\begin{theorem}[{\cite[Theorem 1]{W}}] 
\label{thm1.1}  
If  $2<p<\infty$   then there exist   bounded  weak  solutions  $ \hat u $   of  the $p$-Laplace equation: 
\begin{align}
  \label{1.1}  
  \mathcal{L}_p \hat u  := \nabla \cdot \left( |\nabla \hat u|^{p-2} \nabla \hat u \right) =   0 
\end{align}
  in $\rn{2}_+ = \{ x = (x_1, x_2) : x_2 > 0 \},  $   for which  $\{ t \in \re  : {\ds  \lim_{ y \to 0 } \hat  u ( t + i  y ) }$ exists\}   
  has  Lebesgue measure  zero.  Also    there exist positive  bounded weak solutions of $ \mathcal{L}_p \hat v = 0$  such that  $ 
\{ t \in \re  :  {\ds \limsup_{y \to 0 } \,   \hat v   (t + i  y)  > 0 \}}  $  has Lebesgue measure $0$.    
\end{theorem}    
 
The  key to his proof  and  the only obstacle in extending  Theorem \ref{thm1.1} to $ 2 < p\neq 2 < \infty$    
was the  validity of  the  following  theorem,   stated  as  Lemma 1 in  \cite{W}.  
\begin{theorem}[{\cite[Lemma 1]{W}}]
\label{thm1.2}   
 If  $2<p<\infty$  there exists a  bounded  Lipschitz function $ \Ph $  on  
the closure of $  \rn{2}_+ $  with   $  \Ph ( z + 1 )  =  \Ph ( z )$ for $z \in \rn{2}_+$,   $\mathcal{L}_p  \Phi  = 0   $ 
weakly  on $  \rn{2}_+$, $   \int_{ (0,1) \times (0, \infty) }  | \nabla  \Ph |^p   \, dx dy   <  \infty,  $  
and    
\begin{align}
\label{1.2} 
  {\ds   \lim_{y \to \infty} \Phi ( t + i  y ) = 0}\, \, \mbox{uniformly for}\, \, t \in \re,   \quad  \mbox{but}\quad     \int_0^1  \,  \Ph ( s )  ds   \neq 0.   
\end{align}
\end{theorem}  
 Theorem \ref{thm1.2} was  later  proved for $ 1 < p < 2, $    by the second author of this article
  in   \cite{L1}  (so  Theorem \ref{thm1.1}  is  valid   for  $  1 < p\neq 2 <  \infty). $  Wolff notes that  
  Theorems \ref{thm1.1} and \ref{thm1.2}, generalize to  $ \rn{n}_+ = \{x = (x_1, \ldots, x_n) : x_n > 0 \}$  
  simply by defining $ \hat u, \hat v, \Ph, $  to be constant in the  additional  coordinate directions.       
  Wolff remarks above the statement of   his Lemma 1,  that Theorem \ref{thm1.1}     
  ``should generalize to  other domains  but  the  arguments are easiest in a 
  half space since  $ \mathcal{ L}_p $   behaves  nicely under  Euclidean operations''.
 In fact Wolff made extensive use  in his  argument  of    the fact that $ \Ph ( N  z + z_0 )$, $z = x + i y  \in \rn{2}_+$,  $ N $  a  
 positive integer,  $ z_0   \in \rn{2}_+,   $ is  $p$-harmonic in $ \rn{2}_+ ,     $ 
 and  $ 1/N $  periodic in $x$,  with  Lipschitz norm  $ \approx N $  on $ \re  = \ar \rn {2}_+$.  Also he  used    functional analysis-PDE   arguments,  involving the  
Fredholm alternative and perturbation of  certain $p$-harmonic functions   to  get $\Phi $ satisfying \eqref{1.2}  when $2<p<\infty$.     
  
Building on a work of  Varpanen in \cite{V},  we  managed  to obtain analogues of  Theorem \ref{thm1.1}  and  Theorem  \ref{thm1.2}  for $ 1 < p < \infty, $ in the unit disk of $ \rn{2}$ in \cite{ALV}.  In fact we 
   gave two proofs  of these theorems when $ p > 2. $  One proof used  the exact values   of exponents in the Martin boundary problem for  $p$-harmonic functions, for $p>2$, in $ \rn{2}_+ $  relative to $ \{0\},$  as well as, boundary Harnack inequalities  for certain $p$-harmonic functions.   This proof    seemed conceptually simpler and more straight forward  to us  than the other proof, so we dubbed it   `a  hands on proof'.  As a warm up for this proof  
    we first gave,  in   Lemma  3.1 of  \cite{ALV},   a  `hands on example' of  a  $ \Ph $  for which  
 Theorem \ref{thm1.2} is  valid.    
  In this  paper we    use a    similar    argument to prove    an  analogue of 
 Theorems \ref{thm1.1} and \ref{thm1.2} for $p$-harmonic  functions  in  domains  whose complements in $ \rn{n},$ are $k$-dimensional planes where $1 \leq k  <  n - 1$.  
 To be more specific  we need some  definitions  and notations.  
 \subsection{Definitions and Notations} 
 \noindent Let $n\geq 2$ and denote points in Euclidean $ n
$-space $ \rn{n} $  by $ y = ( y_1,
 \ldots,  y_n) $.   Let $ \mathbb S^{n-1} $ denote the unit sphere in $\mathbb R^n$. 
 We write  $ e_m,  1 \leq m \leq n,  $  for the point in  $  \rn{n} $  with  1 in the $m$-th coordinate and  $0$  elsewhere. Let  $ \bar E,
\ar E$,  and $\mbox{diam}(E)$ be the closure,
 boundary, and diameter of the set $ E \subset
\mathbb R^{n} $ respectively. We define $ d ( y, E ) $   to be the distance from
 $  y \in \mathbb R^{n} $ to $ E$.  
 Let  $   \lan \cdot,  \cdot  \ran $  denote  the standard inner
product on $ \mathbb R^{n} $ and  let  $  | y | = \lan y, y \ran^{1/2} $ be
the  Euclidean norm of $ y. $   For  $z\in \mathbb R^{n}$ and $r>0$, put
\[
B (z, r ) = \{ y \in \mathbb R^{n} : | z  -  y | < r \}. 
\]
Let  $dy$ denote  the $n$-dimensional Lebesgue measure on    $ \mathbb R^{n} $ 
and let  $\mathcal{H}^{\la},  0 <  \la  \leq n, $  denote the  $\la$-dimensional  \textit{Hausdorff measure} on $ \rn{n}$ defined by 
\[
\mathcal{H}^{\la }(E)=\lim_{\delta\to 0} \inf\left\{\sum_{j} r_j^{\la}; \, \, E\subset\bigcup\limits_{j} B(x_j, r_j), \, \, r_j\leq \delta\right\}
\] 
where the  infimum is taken over all possible  $  \de$-covering  $\{B(x_j, r_j)\} $ of  $E$.  
If $ O  \subset \mathbb R^{n} $ is open and $ 1  \leq  q  \leq  \infty, $ then by   $
W^{1 ,q} ( O ) $ we denote the space of equivalence classes of functions
$ h $ with distributional gradient $ \nabla h= ( h_{y_1},
 \ldots, h_{y_n} ), $ both of which are $q$-th power integrable on $ O. $  Let  
 \[
 \| h \|_{1,q} = \| h \|_q +  \| \, | \nabla h | \, \|_{q}
 \]
be the  norm in $ W^{1,q} ( O ) $ where $ \| \cdot \|_q $ is
the usual  Lebesgue $ q $ norm  of functions in the Lebesgue space $ L^q(O).$ 
Let $ C^\infty_0 (O )$ be
the set of infinitely differentiable functions with compact support in 
$O $ and let  $ W^{1,q}_0 ( O ) $ be the closure of $ C^\infty_0 ( O ) $
in the norm of $ W^{1,q} ( O  ). $  
\begin{definition}  
\label{def1.3a} 
For fixed $p$ with $ 1 < p < \infty, $, given a compact set $ E$ and open set $O$ with $E \subset  O $,   define   the  $p$-capacity of  $ E $  relative to $ O $  by   
\[ 
\mathcal{C}_p ( E, O) :=  \inf\left\{ \int_O | \nabla h |^p   dx  : h \in W^{1,p}_0 (O) \mbox{ with } h  \geq 1 \mbox{ on } E \right\}.  
\]  
\end{definition}

\begin{definition} 
\label{def1.3b} 
If $p$ is fixed, $ 1 < p < \infty, $ then 
{\em $ \hat u $   is said to be $p$-harmonic in  an  open set 
 $ O  $}  provided $ \hat u \in W^ {1,p} ( G ) $ for each open $ G $ with  $ \bar G \subset O $ and
	\begin{align}
	\label{1.3}
		\int \lan    |\nabla \hat u|^{p-2}  \nabla \hat u (y),  \nabla  \he ( y ) \ran \, dy = 0 \quad \mbox{whenever} \, \,\he \in W^{1, p}_0 ( G ).
			\end{align}
	 We say that $  \hat u  $ is a  {\em  $p$  subsolution}  ($p$ {\em supersolution}) in $O$ 
	 if   $  \hat u \in W^{1,p} (G) $ whenever $ G $ is as above and  \eqref{1.3} holds with
$=$ replaced by $\leq$ ($\geq$) whenever $  \theta  \in W^{1,p}_{0} (G )$ with $\theta \geq 0$.  
 Here $ \nabla \cdot $  denotes the divergence operator. 
 \end{definition}  
   \begin{definition} 
   \label{def1.1}  
   Given $ 1 \leq  k  \leq   n - 2, n \geq 3, $  let  $ \La_k \subset \rn{n} $ be a $k$-dimensional plane.   
   If  $ p $   is fixed, $  n - k  < p < \infty,   $   and  $ z \in  \La_k, $ then $ u $  is said to be a $p$-Martin function for   $ \La_k, $  relative to 
   $ \{z\}, $  provided  $  u > 0 $ is    $p$-harmonic  in  $ \rn{n} \sem \La_k  $  and  $ u (x) \to 0 $ as $ x \to \infty$, $x \in  \rn{n} \sem \La_k.$   Also  $u$ is  continuous  in $ \rn{n} \sem \{z\}$ with $  u  \equiv 0 $ on $ \La_k \sem \{z\}. $  A $p$-Martin function is defined similarly  when $ k = n - 1$, $z \in \La_k$ only relative to a component of  $ \rn{n} \sem \La_k. $  
   \end{definition} 
      Existence  of   $ u $  for $ 1 \leq k  \leq n - 2 , p > n - k ,$  is shown in Lemma 8.2  of \cite{LN}. For existence  of $u$  when $ k = n - 1$ and $p > 2, $  see subsection 5.1 in  \cite{LLN}.     Also (see \eqref{4.1}  in section \ref{sec4}), 
   \beq \bea{l} 
   \label{1.4} 
   (a) \hs{.2in}  u \mbox{ is unique up to constant multiples,} \\ \\
   (b) \hs{.2in} \mbox{There exists $ \si =  \si ( p, n, k)  > 0 $ such that  \,  $  u ( z + t  x ) =  t^{- \si} \hat u ( z +  x ) $}  \\ 
   \hs{.46in}  \mbox{ whenever $  t > 0.$}  
   \ea \eeq   
To make our `hands on` argument work,  when  $ 1 \leq k \leq n - 2$ and $p > n - k, $ we need to show in   \eqref{1.4} $(b)$   that  $ \si < k $   
when $ n - k  < p. $  In estimating $ \si $  and in  statement of our theorems, we assume that 
\[    
z=0 \mbox{ and }  \La_k = \{ (  x_{1}, \ldots, x_k, 0, \ldots, 0)\in\mathbb{R}^n  :  x_i \in \re,  1  \leq i  \leq k  \}. 
\]  
This assumption is permissible  since $p$-harmonic functions are invariant under translation and rotation.   
Moreover with a slight abuse of notation we write $ \rn{k} $  for  $ \La_k$ and $\rn{n} = \rn{k} \times \rn{n - k}.  $        
We now state our first result. 
\begin{mytheorem} 
\label{thmA}    
Let $ k, n $ be  fixed positive integers with   $1\leq  k \leq n - 2$ and  $p>n-k$.      
Let   $ x' =  ( x_1, x_2,...,  x_{k}),  x''  =  ( x_{k+1}, \ldots, x_n ), $ and $ x = 
 ( x', x'' ). $     Put 
\begin{align}
\label{1.5}  
\hat u ( x ) =   |x''|^{\be}  |x|^{-\ga}  =  |x''|^{\be} r^{-\ga}  \mbox{ where }  \be   =  {\ts \frac{ p +  k - n }{p-1}}, \, \, r = |x|, \mbox{ and } \ga >\be >  0.
\end{align}
Let $ \la  =  \ga - \be. $  If    
\begin{align}
\label{1.6}  
\la     > \max   \left(   \frac{  (p+k-n) ( k  +  p  - 2)}{ (p - 1) ( 2  p - n  + k - 2) },   \frac{ k}{p-1} \right)  =: \chi=\chi (p,n,k),  
\end{align}                     
then $ \hat u $ is a  $p$-subsolution on  $ \rn{n} \sem \rn{k}  $  and $ \chi  < k, $ while   if    
\begin{align}
 \label{1.7} 
 \la <   \min \left(   \frac{  (p+k-n) ( k  +  p  - 2)}{ (p - 1) ( 2  p - n  + k - 2) },   \frac{ k}{p-1} \right) =: \breve \chi=\breve \chi (p,n,k),  
\end{align}
 then $ \hat u $ is a $p$-supersolution on $ \rn{n} \sem \rn{k}$.     
\end{mytheorem}   

\begin{remark} 
\label{rmk1.6}  
We  note  that  Llorente,  Manfredi, Troy, and Wu in \cite{LMTW}   proved   \eqref{1.6},  \eqref{1.7},  when  $ k = n - 1$ and $2  < p <  \infty. $       
We shall use this result   throughout  section \ref{sec6}.    
\end{remark} 
In order to state our second result, we need to introduce some notations.  Given  $ \tau  > 0$ and $y'  \in  \La_k  = \rn{k},  $ let 
 \begin{align}
 \label{1.8a}    
 Q_{\tau} ( y' )  = Q^{(k)}_{\tau} (y' ) :=  \{ z' \in \rn{k} : |z'_i - y'_i | < \tau/2 ,  \mbox{  when }   1 \leq i \leq k\}.  
\end{align}
Armed with Theorem \ref{thmA} and Remark \ref{rmk1.6},  our second result generalizes the work of Wollf \cite{W} and our earlier work in \cite{ALV} when the boundary is a low dimensional plane. 
\begin{mytheorem} 
\label{thmB}
 Let  $ k, n $   be positive integers  with  either  $ (i) $  $   1  \leq k  \leq  n - 2$ and $p >  n - k, $ or $(ii) $  $ k = n - 1$ and $p > 2. $   
 In  case $(i)$   there exists a  $p$-harmonic  function $ \Psi $  on  $ \rn{n}\sem \rn{k},$  that is continuous  on  $ \rn{n}, $ with
 \beq  
 \bea{l} 
 \label{1.8} 
  (a) \hs{.2in} { \ds \Psi \mbox{ Lipschitz   on $ \rn{k} $  and $ \int_{ Q_{1/2}(0) \times \rn{n - k} }    | \nabla  \Psi |^p   \, dx    <  \infty, $ } } 
  \\ \\
  (b) \hs{.2in}   \Psi ( x + e_i  )  =  \Psi ( x  ) \mbox{   for  $   1 \leq i \leq k, $    whenever  $ x \in \rn{n},$ } 
    \\ \\   
    (c) \hs{.2in}  {\ds  \lim_{x'' \in \rn{n-k}  \to \infty} \Psi ( x' ,  x'' ) = 0\,  \mbox{uniformly for}\,  \, x'  \in \bar Q_{1/2} (0),}  \, \,  
   \\ \\ 
   (d) \hs{.2in}    {\ds  \int_{Q_{1/2}(0)} \,  \Psi (x',0)  \, \,  d \mathcal{H}^k x'    \neq 0. }  
\ea    
\eeq
 In case $(ii) $ there exists a $p$-harmonic function $ \Psi $ on $ \rn{n}_+ $ that  is  continuous on  the closure of $ \rn{n}_+$, 
 satisfying  \eqref{1.8} when  $ k = n-1 $ with $ x'' = x_n > 0$.    
 \end{mytheorem} 
 Theorem \ref{thmA} and the technique in proving Theorem  \ref{thmB}  are  also easily seen to imply the following corollary.          
\begin{corollary}  
\label{cor1.7}
 Let $\chi$ and $\breve \chi$ be as in Theorem \ref{thmA}. Let  $ k, n $   be positive integers  with  $   1  \leq k  \leq  n - 2,   $ and $ p $ fixed, $ p >  n - k. $  Let 
$ 0 \leq \om_p ( B(0,r) \cap  \rn{k} , \cdot )   \leq 1, $ denote the unique bounded $p$-harmonic function on $ \rn{n} \sem \rn{k} $ 
which is  1 on $ B (0,r) \cap \rn{k}$  and 0  on  $ \rn{k}  \sem \bar B (0,r) .$  
There exists  $ c = c (p,n,k) \geq 1 $ so that if $ 0 < r < 1/2, $ then 
\[  
c^{-1} \, r^{ \chi} \leq      \om_p ( B (0,r) \cap \rn{k}, e_n )   \leq  c  r^{\breve \chi}.    
\]
\end{corollary}   
Corollary \ref{cor1.7} was proved for $ p > 2$ and $k = n - 1$ in  \cite{LMTW} (see also \cite{DB}) using the analogue of  Theorem \ref{thmA}  (see Remark \ref{rmk1.6}).  

We  can   use  the gist of  Wolff's argument and Theorem  \ref{thmB}  to show the 
failure of a Fatou's theorem  for $p$-harmonic functions vanishing on low dimensional planes.
     \begin{mytheorem}   
     \label{thmC} 
     Let  $ k, n $   be positive integers  with either $(i)$  $   1  \leq k  \leq  n - 2$ with $p > n - k,  $  or  $ (ii)$  $ k = n - 1$ with $p > 2. $  
     In case $ (i) $  there exists  bounded  $p$-harmonic   functions   $ \hat u, \hat v $ in $ \rn{n} \sem \rn{k} $  with  the following properties.  
      Suppose $ \ze :   \rn{n-k} \sem \{0\}  \to  \rn{n-k} \sem \{0\}  $ is  continuous with       $ {\ds \lim_{x'' \to 0}  \ze (x'') = (0,..,0)}. $       
    Then           
\[
\{ x'  \in  \rn{k} : {\ds \lim_{  x''  \to 0 }}      \hat  u ( x',  \ze (x'') )\, \,  \mbox{exists} \} \subset D_1 \quad \mbox{where}\quad  \mathcal{H}^k (D_1) = 0.
\]        
Also   
\[
\{x' \in \rn{k}  :  {\ds \limsup_{x''  \to 0 } } \,   \hat v   (x' ,  \ze (x''))  > 0  \} \subset D_2\quad  \mbox{where}\quad \mathcal{H}^k (D_2) = 0.
\]
 The Borel sets  $ D_1$ and $D_2 $  are  independent of the choice of   $ \ze. $    
   In case $(ii) $  there exists bounded  $ p$-harmonic   functions   $ \hat u, \hat v $ in $ \rn{n}_+ $  with  the above properties when $ k = n - 1$ and $ x'' = x_n > 0$.   
   \end{mytheorem}    
   
   \begin{remark}  
   \label{rmk1.8} To get  Theorem  \ref{thmC} for $ \hat v,$  it would have sufficed to just prove existence of limit 0  for 
   $ \mathcal{H}^k $ almost every  $ x' \in \rn{k} $    when  
 $ \ze (x'') = (0,...,0,x_n), $  thanks to  Harnack's  inequality for  positive $p$-harmonic functions (see  \eqref{2.2} $(c))$.   
  Also in case $(i), $  one could just prove this theorem for $ k = 1,  p > n - 1, $  since the general case would  
  then follow from extending these functions to  $ \rn{n+k - 1 } \sem{\rn{k}}$ for $1 < k $  
  by defining them to be constant in the other  added $ k - 1 $ coordinate directions.  
  However  our approach  yields  a   larger and arguably more interesting variety of  examples. 

We also note that  when $1<p  \leq n - k $
\begin{align} \label{1.9}  
  \mathcal{C}_p (   \rn{k}   \cap \bar B (0, R),  B(0, 2R)) =  0
\end{align}
for every $R>0$ (see \cite[Page 43]{HKM})    and consequentially  neither  Theorem   \ref{thmB} or  Theorem  \ref{thmC} has  an  analogue  in case $(i)$ when   $1<p \leq n - k. $   
    For the reader's convenience  a proof  of this statement is given in  Remark \ref{rmk5.6} after the proof of  
   Theorem \ref{thmC}  in section \ref{sec5}.   
   \end{remark} 


Finally  in   section  \ref{sec6}    we  consider  partial   analogues    of  Theorem  \ref{1.1},  Theorem \ref{1.2},  Theorems  \ref{thmA}-\ref{thmC},  for solutions  to  a  more general class of  PDE's   modelled on the $p$-Laplacian,   which are  called  $ \mathcal{A}$-harmonic functions (see 
Definition \ref{def6.2}    in section \ref{sec6}).  $\mathcal{A}$-harmonic  functions  share with $p$-harmonic functions the properties  used in the proof of  Theorems  \ref{thmA}-\ref{thmC}, so originally  we hoped to prove   these  theorems  with  
$p$-harmonic replaced by  $ \mathcal{A}$-harmonic.   However preliminary investigations using  maple and hand calculations,   indicated  that  this   class  would not in general   yield the  necessary estimates on  exponents of an  $ \mathcal{A}$-harmonic  Martin function for  $p$ in the required  ranges.   For this reason  we  relegated our discussion of  
$ \mathcal{A}$-harmonic  functions to  section \ref{sec6} and made this discussion more or less  self contained. Subsections of section \ref{sec6}  include:

\begin{enumerate}

\item A definition of  $\mathcal{A}$-harmonic functions and   listing of  their basic properties.  \\

\item Statement  of   two    Propositions   concerning validity of  Theorems  
\ref{thmB} -  \ref{thmC}  for   $ \mathcal{A}$-harmonic operators sufficiently near the $p$-Laplace operator.  \\

\item Estimates of   $\mathcal{A}$-harmonic Martin  exponents  when  $ 1 \leq k \leq n - 2$ and $p  > n - k$ for $n  >  3 , $ in  $ \rn{n} \sem \rn{k} $  
(see subsection 6.3)   and when $ k  = n - 1$ and $p > 2$ for $n \geq 2 $  in  $ \rn{n}_+ $  (see subsection 6.4).   
These estimates        give partial analogues of  Theorems \ref{1.1} and \ref{1.2}   for a  subclass  of  $  \mathcal{A}$-harmonic operators,  
only slightly more general than the  $p$-Laplace operator.   Still  this  was   an  interesting  subclass for us to
highlight  computational  difficulties  in   showing  $ \si < k $  for the exponent in \eqref{1.4}  of an  $\mathcal{A}$-harmonic  Martin function.  
Moreover,  in the baseline $ n = 2$ and $p > 2$  case  we obtained a  rather  surprising  result  (see  subsection  6.5  for more details).
\end{enumerate}
   
As for the plan of this paper in  section \ref{sec2} we introduce and state some lemmas  listing  basic estimates for  $p$-harmonic functions. 
Also   specific estimates  for  $ \La_k$ when  $ 1 \leq k \leq n - 1$  are given.    
Statements and references for proofs of  these lemmas  are made  so that  we can essentially say `ditto'  in our discussion of  
$ \mathcal{A}$-harmonic  functions.    In section \ref{sec3}  we prove  Theorem \ref{thmA}.   
In section \ref{sec4}  we  use Theorem  \ref{thmA}  to prove Theorem \ref{thmB}. In section \ref{sec5} 
we indicate the changes in   Wolff's main  lemmas for applications  and  prove  Theorem \ref{thmC}.    
In  section 6  we   introduce   $\mathcal{A}$-harmonic functions and proceed as outlined above.

\section{Definition and Basic Estimates  for $p$-harmonic functions}
\label{sec2}  
\setcounter{equation}{0} 
 \setcounter{theorem}{0}
 
 In this section we  first introduce some more notation  and  then state  some fundamental estimates for 
  $p$-harmonic  functions.  Concerning constants, unless otherwise stated, in  sections \ref{sec2}-\ref{sec6}, $c$  will  denote  a  positive constant  $ \geq 1$, not
necessarily the same at each occurrence, depending only on $p, n, k.$ In general throughout this paper,
$ c ( a_1, a_2,  \ldots,  a_m) $  denotes  a  positive constant  $ \geq 1$, not
necessarily the same at each occurrence, depending only on $  a_1, \ldots, a_m. $  Also   $ A  \approx  B $  means  $  A/B $ is bounded above and below 
 by  positive constants  whose dependence will be stated.  
We also let $ { \ds \max_{E}   \hat v, \,  \min_{E} \hat v } $  denote  the
essential supremum and infimum  of $  \hat v  $  (with respect to  Lebesgue $n$-measure)  whenever $ E \subset  \rn{n} $ and  $ \hat v $ is defined on $ E$.

Next we state  some  basic  lemmas  for   $p$-harmonic functions.  
  \begin{lemma}  
\label{lem2.2}  
For fixed  $ p,   1<p<\infty$,  suppose   $  \hat v $ is  a  $p$-subsolution and  $ \hat h $ is  a  $p$-supersolution in  the open set $ O $ with 
 $ \max ( \hat v -  \hat h, 0 )  \in  W_0^{1, p} (O) . $ 
  Then   $  { \ds \max_O ( \hat v -  \hat h)}  \leq 0 . $ 
\end{lemma} 
  \begin{proof}
A proof of this lemma can be found in \cite[Lemma 3.18]{HKM}. 
  \end{proof}
\begin{lemma}
\label{lem2.3}
For fixed $ p, 1 < p < \infty,   $   let
 $ \hat v $  be  $p$-harmonic in $B (z_0 ,4\rho)$ for some $\rho>0$ and $z_0  \in \rn{n}$. Then  
there exists $  c = c (p,  n)  $   with    
\begin{align}
 	\label{2.2} 
\begin{split}
&(a)  \hs{.2in}  {\ds  \max_{B (z_0,  \rho/2)}  \, \hat v  -     \min_{B ( z_0, \rho/2 ) }  \hat v     \leq  c \left( \rho^{ p - n} \,\int_{B ( z_0, \rho)} \, | \nabla \hat v |^{ p } \, dx \right)^{1/p}   \, \leq \, c^2  \, (\max_{B ( z_0,
 2 \rho )} \hat  v   -   \min_{B ( z_0, 2 \rho ) }  \hat v  ).} \\
&	\mbox{ Furthermore, there exists $\ti \be  = \ti \be (p, n)  \in (0,1)$ such that if $ s \leq  \rho,$  then} \\ 			
& (b) \hs{.2in}   {\ds  \max_{B(z_0,s) }  \hat v  - \min_{B ( z_0, s) } \hat v  \leq  c \left( \frac{s}{\rho} \right)^{\ti \be} \,  \left(\max_{B (z_0, 2 \rho )}  \, \hat v  -     \min_{B ( z_0, 2\rho ) }  \hat v \right) .} \\  
&(c) \hs{.2in}  \mbox{ If  $ \hat v \geq 0 $ in $ B ( z_0, 4 \rho), $  then }      {\ds  \max_{B(z_0, 2\rho ) }  \hat v  \leq   c  \min_{B ( z_0, 2 \rho ) } \hat v }. 
\end{split}
\end{align}
 \end{lemma}  
\begin{proof}
     Lemma \ref{lem2.3} is well known.   A   proof of this  lemma,   using Moser iteration of  positive solutions to PDE of 
      $p$-Laplace type,    can be found in \cite{S} or Chapter 6 in \cite{HKM}. \eqref{2.2}  $(c)$ is called  Harnack's inequality.  
\end{proof}    
\begin{lemma} 
\label{lem2.4} 
Let  $ k $  be a positive integer $ 1 \leq k \leq n-1, $   $p$  fixed with $  n - k  < p < \infty$  and $ \rho > 0. $   Let $ \rn{n}_+  = \{ x \in \rn{n} : x_n > 0 \},  z_0   \in  \rn{k},   $   and  put $ \Om = (\rn{n} \sem \rn{k}) \cap B ( z_0, 4 \rho)$ 
when $ 1 \leq k \leq n - 2 $  while $ \Om =  \rn{n}_+ \cap B ( z_0, 4\rho) $  when $ k = n - 1. $   Let 
$ \ze  \in C_0^\infty ( B ( z_0, 4 \rho )) $   with $ \ze \equiv 1 $ on $ B ( z_0, 3 \rho) $ and $ |\nabla \ze | \leq c (n) \rho^{-1}  . $ 
Suppose  $  \hat v $ is  ${p}$-harmonic  in   $  \Om,  $  $ \hat h  \in  W^{1,p}  ( \Om),  $   and   $ ( \hat v  - \hat h ) \ze  \in  W^{1,p}_0  ( \Om). $   

  If    $  \hat h $ is continuous on $  \ar \Om \cap B(z_0, 4 \rho), $  
  then $ \hat v $ has a continuous extension to  $ \bar \Om \cap B ( z_0, 4 \rho), $ 
   also denoted  $ \hat v, $  with  $ \hat v \equiv \hat  h $ on  $  \ar  \Om \cap B ( z_0, 4 \rho). $    If  
\[  
|\hat h ( z )  -  \hat h (w) | \leq    M'   | z - w |^{\hat \si} \quad \mbox{whenever}\quad    z, w \in    \ar \Om \cap B (z_0, 3 \rho) ,  
\]  
for some $  \hat \si \in (0, 1], $  and $  1  \leq  M'  <  \infty, $ then there exists $  \hat \si_1 \in  (0, 1],  c \geq 1, $  depending only on $ \hat \si,   n, $ and $p $,  such that   
\begin{align}
\label{2.3}   
 | \hat v ( z ) -  \hat v ( w) |  \leq  8 M' \rho^{ \hat \si}  +   \,    \,  (| z - w |/ 2 \rho)^{ \hat \si_1}     \, \,    \max_{\Om \cap  \bar B ( z_0, 2 \rho)} \, | \hat v|
\end{align}
whenever $z, w   \in \Om \cap B ( z_0,  \rho   )$. 
 
If  $   \hat h \equiv 0  \mbox{ on } \ar  \Om $,   $\hat v   \geq 0 $ in  $  B ( z_0, 4 \rho )$,    $\hat c  \geq  1,  $  and  $ z_1  \in   \Om  \cap    B ( z_0, 3 \rho )$, with $ \hat c \,  d ( z_1, \ar \Om ) \geq    \rho, $  then there exists  $  \ti c,   $  depending only on $ \hat c,   n, $ and  $p$,    such that  
\begin{align}
 \label{2.4} 
\hs{.85in} (+) \hs{.2in}   \max_{B (z_0, 2 \rho )} \hat v  \leq  \ti c \left( \rho^{ p - n} \,\int_{B ( z_0, 3 \rho)} \, | \nabla \hat v |^{ p } \, dx  \right)^{1/p}   \leq   (\ti c)^2  \, \hat  v ( z_1 ).
\end{align} 
Furthermore, using \eqref{2.3}, it follows for $ z, w  \in  \bar \Om \cap B ( z_0, 2 \rho )$ that
\[
(++)   \hs{.2in}       | \hat v ( z  ) - \hat  v ( w ) |   \leq  c\,  \hat v (z_1)   \left( \frac{ | z  -  w |}{\rho } \right )^{ \hat \si_1}  . 
\]
\end{lemma}

 \begin{proof}  
 Continuity of  $  \hat v $   given  continuity of  $ \hat h $  in  $ \bar \Om  $  follows from Corollary 6.36  in  \cite{HKM}.      
   This  Corollary  and the H\"{o}lder  continuity  estimate on  $ h $ above,   are then used in   Theorem 6.44  of  \cite{HKM}  
   to  prove  an  inequality analogous to  \eqref{2.4}.   Proofs  involve   Wiener  type estimates (in terms of  $p$-capacity) 
   for  $p$-harmonic subsolutions   that vanish on  
   $  \ar  \Om  \cap  B ( z_0,  3 \rho) $.         
\end{proof} 
\begin{lemma}   
\label{lem2.5}	
Let $p,  \hat v, z_0,\rho,$ be as in Lemma \ref{lem2.3}.
Then	 $ \hat  v $ has a  representative locally in  $ W^{1, p} (B(z_0, 4\rho))$,    with H\"{o}lder
continuous partial derivatives in $B(z_0,4\rho) $  (also denoted $\hat v $), and there exist $ \hat  \ga \in (0,1]$ and $c \geq 1 $,
depending only on $ p, n, $ such that if $ z,  w  \in B (  z_0,  \rho/2 ), $     then  
\begin{align}
	\label{2.5} 	
\begin{split}	
&	(\hat a) \hs{.1in} \, \, 
 c^{ - 1} \, | \nabla \hat v  ( z ) - \nabla \hat v ( w ) | \, \leq \,
 ( | z -  w  |/ \rho)^{\hat  \ga }\, {\ds \max_{B (  z_0 ,\rho )}} \, | \nabla \hat  v | \leq \, c \,  \rho^{ - 1} \, ( | z -  w |/ \rho )^{\hat  \ga  }\, {\ds \max_{B (z_0, 2 \rho )}}  | \hat v |.  \\
&   \mbox{Also $ \hat v $  has distributional  second partials with } \\
&   (\hat b)  \, \,  \, {\ds \int_{B(z_0, \rho) \cap  \{ \nabla \hat v \not = 0 \} }}  \,   |\nabla \hat v  |^{p-2} \, \left( 
\sum_{i,j=1}^n \,  |\hat v_{x_i x_j}|^2 (z)   \right) dx  
\leq   c \, \rho^{n-p - 2} \, {\ds  \max_{B (z_0, 2 \rho )}  | \hat v |} .\\
\end{split}
\end{align}
 \end{lemma}

 \begin{proof}
  For  a  proof of  \eqref{2.5}  $(\hat a), (\hat b), $    see  Theorem 1  and Proposition 1 in  \cite{T}. 
 \end{proof} 
     In  the proof of  Theorems \ref{thmB} and \ref{thmC},  we need the following boundary Harnack inequalities.   
\begin{lemma} 
\label{lem2.6}  
Let   $k, n, p,  z_0, \Om,  $ be as in Lemma \ref{lem2.4}.   Suppose  $ \hat u,$  $\hat v,$ are non-negative  
$p$-harmonic functions  in   $ \Om $  with   $ \hat u  = \hat v \equiv 0 $  on  $ \ar \Om  \cap  B ( z_0, 4 \rho).$      
There exists  $ c = c (p, n, k) $  and $  \be = \be ( p, n, k) $  such that  if $ y, z \in  \Om \cap B (z_0, \rho ), $ then 
\begin{align}
\label{2.6}  
 \left|  \frac{\hat u (z)}{\hat v (z)}  -  \frac{\hat u (y)}{\hat v (y)} \right|   \leq     c  
 \frac{\hat u (z)}{\hat v (z)}      \left( \frac{ | y - z |}{ \rho}  \right)^{\be}. 
\end{align}
 \end{lemma}  
 
\begin{proof}  
For a proof  of Lemma \ref{lem2.6}  when $ k = n - 1, p > 1, $    see  Theorem 1  in \cite{LLN} and when  $ 1 \leq k \leq n - 2 $ see  Theorems 1.9 and 1.10   in \cite{LN}.  
\end{proof}        
    
    \begin{lemma} 
\label{lem2.7}  
Let   $k, n, p,  z_0,  \rho,  $ be as in Lemma \ref{lem2.4}.   Let  $  G =  \rn{n} \sem ( \rn{k} \cup \bar B ( 0, \rho ) ) $ 
when  $  k < n - 1 $    and   $ G  = \rn{n}_+ \sem \bar B ( 0, \rho ) $  when $ k  = n - 1.$  Suppose  $ \hat u,$  $\hat v,$ are non-negative  
$p$-harmonic functions  in   $ G $  with  continuous boundary values  and  $ \hat u  = \hat v \equiv 0 $  on  $ \ar G  \cap  \rn{k} $  when $ k \leq  n - 1$.  Moreover  
$    \hat u ( x) + \hat v (x) \to  0  \mbox{ uniformly for }  x \in G \mbox{ as }  |x| \to  \infty. $ 
     There exists  $ c = c (p, n, k) $  and $  \be' = \be' ( p, n, k) $  such that  if $ y, z \in   G \sem B ( 0, 2 \rho), $  then 
\begin{align}
\label{2.7}  
 \left|  \frac{\hat u (z)}{\hat v (z)}  -  \frac{\hat u (y)}{\hat v (y)} \right|   \leq     c  
 \frac{\hat u (z)}{\hat v (z)}      \left(\frac{ \rho}{ \min (|y|, |z|)}  \right)^{\be'}   . 
\end{align}
 \end{lemma}  
 \begin{proof}  
For the proof  of Lemma \ref{lem2.7}  when $ k = n - 1, p > 1, $    see Theorem 2 in  \cite{LLN}.   
For a proof of Lemma  \ref{lem2.7} when  $ 1 \leq k \leq n - 2, $ see  Theorem 1.13  in    \cite{LN}.  In both cases the proof of \eqref{2.7}  is given only  when  $ G = \Om \sem \bar B (0, \rho) $  and  $ \Om $ is a   bounded domain.  However the proof is essentially the same  in either case for $ G  $ as above.   \end{proof}   
     
  \noi  For  fixed $ k$ with $1\leq k  \leq n - 1$ for  $n \geq 2, $ and  $ y'  \in  \rn{k}, $  let  $  Q_{\tau} (y') =  Q_{\tau}^{(k)} (y'),  $  be as in \eqref{1.8a}.   Let    
\[
S  (y',  \tau )  =   S^{(k)} (y', \tau) :=   \{  (x', x'') \in  Q_{\tau}^{(k)} (y')  \times ( \rn{n-k}\sem\{0\} )  \}
\]
 when $ 1 \leq k \leq n - 2 $  and 
 \[  
 S ( y', \tau ) := \{ (x', x_n) :  x' \in Q^{n-1}_{\tau} (y'),  x_n > 0 \} 
 \]  
 when $ k = n - 1. $  For short we write $ S (\tau)  $  when  $ y' = 0 \in \rn{k}. $   If   $ 1 \leq k  \leq n - 2$ with $p > n - k, $ let   $  R^{1, p } (S (\tau)) $  
denote the Riesz space of  equivalence classes of functions $ F $ with distributional derivatives  
on $ \rn{n} \sem \rn{k} $  and   $ F ( z + \tau e_i  )  = F ( z ) ,  1 \leq i \leq k,  $  when $ z \in \rn{n}\sem \rn{k}.  $  Also 
\begin{align} 
\label{2.8} 
\| F \|_* =   \| F \|_{*, p}   =    \left( \int_{S ( \tau )}   | \nabla  F |^p  \, dx \right)^{1/p}   <  \infty. 
\end{align} 
If  $ k  = n - 1 $  and $ p > 2, $  define  $ R^{1,p}(S(\tau))  $ similarly,   only  with $ \rn{n}\sem \rn{k} $ 
replaced  by   $ \rn{n}_+.$   
Next let $ R^{1,p}_0 ( S ( \tau )) $    denote functions in  $ R^{1,p} ( S ( \tau ) ) $ which can be approximated 
arbitrarily closely   in the norm of  $ R^{1,p} ( S ( \tau ))  $  by   functions in  this space which are  
infinitely differentiable and vanish in    an   open    neighbourhood  of $ \rn{k}. $      
Using  a  variational argument  as  in   \cite{E} it can be shown   that given  $ F  \in  R^{1, p} (S (\tau)), $  there exists a  unique 
$p$-harmonic function  $ \hat  v $   on  $ \rn{n} \sem \rn{k}  $  when $  1  \leq k \leq n - 2, p > n - k ,$   and on 
$ \rn{n}_+ $  when $ k  =  n - 1,  p > 2, $  with   $ \hat  v ( z +  \tau e_i  ) = \hat  v ( z ),   1 \leq i \leq k, $ for $z \in  
\rn{n}\sem\rn{k}  $ or $\rn{n}_+.$   Moreover  $ \hat  v   - F  \in R_0^{1,p} (S(\tau)).   $  In fact the usual minimization argument 
yields that $ \|\hat v \|_{*,p} $ has minimum norm among all functions $ h $  in  $ R^{1,p} ( S ( \tau ) ) $ with $ h - F 
\in R_0^{1,p} (S (\tau)).$    Uniqueness of $ \hat v  $ is  a  consequence of  the maximum  principle  in  Lemma  
\ref{lem2.2}.    
Next we state 
\begin{lemma}   
\label{lem2.8} 
Let  $ p, n, k, \tau, F, \hat v,     $  be as above.     Given  $ t > 0, $  let   
$ Z (t) =  \{ ( x', x'') \in \rn{k} \times \rn{n-k} : |x''| = t  \}   $  when $ 1  \leq k  \leq n - 2, $  
and $ Z (t) = \{ ( x', x_n') \in \rn{k} \times \rn{n-k} : x_n = t  \}   $  when $  k = n - 1. $  
There exists $  \de  = \de  (p, n, k ) \in (0, 1)  $ and $ \xi  \in \re $  such that  
\begin{align}
\label{2.9}   
| \hat v  ( z ) -  \xi |  \leq   \liminf_{t \to 0} \left(  \max_{Z(t)}  \hat v  -   \min_{Z(t)}   \hat  v  \right)   \,  \left( \frac{\tau}{|z''|} \right)^{\de}   
\end{align}
 whenever either  $ z = (z', z'')   \in  \rn{n}\sem \rn{k} $ or $ z \in \rn{n}_+ $ with  $ z_n =  |z''|. $    
 \end{lemma} 
 \begin{proof}  
 Fix $ t > 0 $  and first suppose that  $ 1 \leq k \leq n - 2, p > 2. $    
 We note  that   $ \hat v $   restricted  to   $ D(t) =  \{ (x', x'' )  \in  \rn{k} \times \rn{n-k} :   |x'' | > t \} $    is   a  $p$-harmonic  solution 
 to a certain  calculus of  variations  minimization problem.   Moreover   
 $ \ti v (x)  =   \min (\hat v (x),  \max_{Z(t)} v )$ for $x \in D(t),   $    belongs to  the class of  possible  minimizers   
and   
\begin{align} \notag
\int_{D(t)} | \nabla \ti v |^p  dx  \leq   \int_{D(t)} | \nabla \hat v |^p  dx  
\end{align}
so  by uniqueness of the minimizer,    $ \hat v   \leq \max_{Z(t)} v $ on $ D(t).$   Similarly  $ \hat v  \geq \min_{Z(t)} v $ on $ D(t). $  
Thus the term in \eqref{2.9} involving $ \hat v $   is decreasing as a function of $t.$    \eqref{2.9}   follows from   this fact,  
and  an argument using  $  \tau e_i, 1\leq i \leq k, $  periodicity  of $ \hat v $  together with Harnack's inequality (see Lemma 1.3  in \cite{W} for a 
 similar argument).    Note that  
 \[  
 \xi =  \lim_{t\to  \infty}  \max_{Z(t)}  \hat v  =   \lim_{t \to  \infty}  \min_{Z(t)}   \hat  v  .   
 \]   
 A similar argument applies if  $ k = n - 1$ and $p > 2. $ We omit the details. 
\end{proof}   

\begin{remark}   
If $ k = n  - 1$ and $p > 2,   $   Harnack's   inequality actually yields the stronger decay rate:  
\[    
| \hat v  ( z ) -  \xi |  \leq   \liminf_{t \to 0} \left(  \max_{Z(t)}  \hat v  -   \min_{Z(t)}   \hat  v  \right)   \, e^{ -  \de z_n /\tau}.  
\]   
	  Moreover  this decay rate can also be obtained when  $ 1 \leq k  \leq n - 2$ for $p > k, $ using rotational invariance of  $ \hat v (x) = \hat v (x', x'' ) $ in the $ x'' $ variable.   
	  However  we can only prove \eqref{2.9}  for the  larger class of  $ \mathcal{A} =  \nabla f-$harmonic functions  discussed in section \ref{sec6}  
	  and this inequality is all we need in  the proof of  Theorem \ref{thmC}.   
\end{remark}

\setcounter{equation}{0} 
 \setcounter{theorem}{0}
 \section{Proof of Theorem \ref{thmA}} 
 \label{sec3}   
 In  this section we prove Theorem \ref{thmA}.  To do so  let    
$x=(x', x'')\in \mathbb{R}^k \times \mathbb{R}^{n-k}$ for $1\leq k \leq n-2$ and set 
$t=|x'|$, $s=|x''|$.   We begin by  deriving  the $p$-Laplace equation  for a smooth rotationally invariant function, $u(x)=u(s,t),$  in $ \rn{n}\sem\rn{k}.$    

\subsection{The $p$-Laplace Equation in $s, t$}  
We compute 
\begin{equation}
\label{gradu}
\nabla u = ( \frac{u_t}{t} x', \frac{u_s}{s}x'') \quad \text{and} \quad |\nabla u| = \sqrt{ u_t^2 + u_s^2}.
\end{equation}
Now we show that the  $p$-Laplacian of $u (x) $, is the $(s,t)$ $p$-Laplace plus another term:
\begin{equation}
\label{3.2}
\nabla \cdot (|\nabla u |^{p-2} \nabla u)  =
\nabla_{s,t} \cdot (|\nabla_{s,t} u(s,t)|^{p-2} \nabla_{s,t} u(s,t) )+ 
 \sqrt{u_t^2+u_s^2}^{p-2} ((k-1) \frac{u_t}{t} +
(n-k-1) \frac{u_s}{s}).
\end{equation}
This is the calculation
\begin{multline*} 
\nabla \cdot \sqrt{u_t^2+u_s^2}^{p-2} ( \frac{u_t}{t} x', \frac{u_s}{s}x'') =  \\
(p-2) \sqrt{u_t^2+u_s^2}^{p-4} 
\left \lan\,  \left( (u_t u_{tt}+u_su_{st})\frac{x'}{t}, ( u_t u_{ts} + u_su_{ss})\frac{x''}{s})\right), 
\ ( \frac{u_t}{t} x', \frac{u_s}{s}x'' )  \, \right\ran \\
+ \sqrt{u_t^2+u_s^2}^{p-2} \nabla \cdot ( \frac{u_t}{t} x', \frac{u_s}{s}x'') \\
=(p-2) \sqrt{u_t^2+u_s^2}^{p-4}  (u_t^2u_{tt}+2 u_t u_s u_{ts}+u_s^2u_{ss}) + 
 \sqrt{u_t^2+u_s^2}^{p-2} (u_{tt}+u_{ss} -\frac{u_t}{t}-\frac{u_s}{s} +k \frac{u_t}{t} +
(n-k) \frac{u_s}{s}) \\
=(p-2) \sqrt{u_t^2+u_s^2}^{p-4}  (u_t^2u_{tt}+2 u_t u_s u_{ts}+u_s^2u_{ss}) + 
 \sqrt{u_t^2+u_s^2}^{p-2} (u_{tt}+u_{ss} +(k-1) \frac{u_t}{t} +
(n-k-1) \frac{u_s}{s}) \\
= \sqrt{u_t^2+u_s^2}^{p-4}  ((p-2)  (u_t^2u_{tt}+2 u_t u_s u_{ts}+u_s^2u_{ss}) 
+ (u_t^2 + u_s^2)(u_{tt}+u_{ss} +(k-1) \frac{u_t}{t} +
(n-k-1) \frac{u_s}{s}) ). 
\end{multline*}   
To find solutions,  subsolutions,  supersolutions in terms of $(s,t)$ we need only study when the next display is 
$0$, positive, negative
\begin{equation}
\label{steqn}
(p-2)  (u_t^2u_{tt}+2 u_t u_s u_{ts}+u_s^2u_{ss}) 
+ (u_t^2 + u_s^2)(u_{tt}+u_{ss} +(k-1) \frac{u_t}{t} +
(n-k-1) \frac{u_s}{s}).
\end{equation} 
\subsection{Particular Forms for $u(s,t)$ and Proof of Theorem \ref{thmA}}
\begin{proof}[Proof of Theorem \ref{thmA}]
Let $r=\sqrt{s^2+t^2}$ and  for Martin $p$-harmonic  sub or super solutions we consider  functions with the form 
\begin{equation} 
\label{3.4} 
u(x) = u(s,t)= s^\beta r^{-(\lambda + \beta)}, \quad \mbox{for} \,\, x \in \rn{n} \sem \rn{k}, 
\end{equation}   
where  $  \be =  \frac{p-n +k }{p-1} $   
and $ \la >  0. $ Our  choice of  $ \be $ and the form of $u$  is motivated by \eqref{1.4}  and  
by the boundary Harnack inequality  in Lemma \ref{lem2.6}. Indeed    
$ |x''|^{\be}, $  is a solution to the $p$-Laplace equation in $ \rn{n} \sem \rn{k}$
 with continuous  boundary value 0 on  $ \rn{n} \sem \rn{k} $.  
 So    the Martin  $p$-harmonic function relative to  $0$  for  $ \rn{n} \sem \rn{k} $ is homogeneous  in $ r $   with negative exponent  
 and by  Lemma \ref{lem2.6} this function is  bounded above and below at $ x \in \ar B(0,1) \sem \rn{k} $   by $ |x''|^{\be}.$ 
 Ratio  constants depend only on $ p,n,k$ and  the  value of the Martin function at  $e_n$.       

Now for the derivatives of $u$ we get 
  \begin{align} 
  \label{3.5} 
  u_t & = -(\lambda + \beta)  \frac{t}{r^2} u, \\ 
u_s & \notag = u\frac{\beta}{s}  -(\lambda + \beta) u \frac{s}{r^2} =  \frac{-\lambda s^2 + \beta t^2}{s r^2} u.
\end{align} 
Consequently 
\begin{equation} 
\label{3.6}
u_s^2+u_t^2 = u^2 
\left( \frac{ (\lambda+\beta)^2 t^2 s^2}{s^2 r^4}+ \frac{(-\lambda s^2 + \beta t^2)^2}{s^2 r^4}\right).
\end{equation}
Now the cross term $2 \lambda \beta s^2 t^2$ in the numerator cancels, leaving 
$(\lambda^2+\beta^2)s^2t^2 + \lambda^2s^4+\beta^2 t^4$ which factors into
$(\lambda^2 s^2+\beta^2 t^2)(s^2+t^2)= (\lambda^2 s^2+\beta^2 t^2)r^2$.
Altogether this gives 
\begin{equation}
\label{3.7} 
|\nabla u(s,t)|^2 = \frac{u^2}{s^2r^2}( \lambda^2 s^2 +\beta^2 t^2).
\end{equation}
Writing the second derivatives in terms of $u$ we have 
\begin{align} 
\label{3.8}
\begin{split}
u_{tt} & = \frac{u}{r^4} (\lambda + \beta) ( (\lambda + \beta+1)t^2 -s^2), \\ 
u_{ts} &  =  \frac{ut}{s r^4} (\lambda + \beta) ( (2+\lambda) s^2 -\beta t^2), \\  
u_{ss} & = \frac{u}{s^2 r^4} ( (\lambda^2+\lambda) s^4 -(\lambda +3 \beta +2 \lambda \beta) s^2 t^2 + (\beta^2 -\beta ) t^4).
\end{split}
\end{align}
Looking at equation \eqref{steqn} we see that we can factor out $\frac{u^3}{s^4 r^8}$ in the first term, the rest of the first term factors
\begin{equation}
\label{p1}
(p-2)\frac{u^3}{s^4 r^8} ( \lambda^3 (\lambda +1) s^4 +\beta^2 (2\lambda^2 -\beta +\lambda)s^2 t^2+\beta^3(\beta -1)t^4) r^4.
\end{equation}
In the second term we can factor out $\frac{u^3}{s^4 r^6}$ and the rest of it factors
\begin{equation}
\label{p2}
\frac{u^3}{s^4 r^6}( \lambda^2 s^2 +\beta^2 t^2) ( (\lambda^2 +(2-n)\lambda -\beta k)s^2 + \beta(\beta+n-k-2)t^2) r^2.
\end{equation}
Cancelling the $r$'s and adding these two terms we get 
\begin{equation}
\label{abceqn}
\frac{u^3}{s^4 r^4} ( A \lambda^2 s^4 + B \beta s^2 t^2 + C \beta^3 t^4)
\end{equation}
where 
\begin{align} 
\label{3.13}  
\begin{split}
A & = (p-2)(\lambda^2+\lambda)+\lambda^2 +(2-n)\lambda -\beta k  = (p-1)\lambda^2 +(p-n)\lambda -\beta k, \\ 
B & = \beta(2\lambda^2-\beta+\lambda)(p-2)+\lambda^2(\beta+n-k-2)+\beta(\lambda^2+(2-n)\lambda -\beta k)  \\ 
 &= 	(2\beta(p-1)+n-k-2)\lambda^2 + \beta (p-n)\lambda -\beta^2(p-2+k), \\ 
C& = (p-2)(\beta-1)+\beta+n-k-2 
 = (p-1)\beta - (p-n+k).
\end{split}
\end{align}
We are using $\beta = \frac{p-n+k}{p-1}$ so that $C=0$, and in equation \eqref{abceqn} we can factor out an $s^2$.   Thus to determine solutions, subsolutions, 
supersolutions we just need to know when the following equation is 0, positive, negative for all $(s,t)$.
\begin{equation}\label{abeqn}
A \lambda^2 s^2 + B \beta  t^2.
\end{equation}
For this we need $A,B=0$, $A,B >0$, $A,B<0$ respectively.   Now $A,B$ are quadratics in $\lambda$ with positive leading coefficient.  Using the quadratic formula, the discriminant is a perfect square, $A$ has roots  $-\beta$ and $\frac{k}{p-1}$ while $B$ has roots $-\beta$ and $\beta \frac{p-2+k}{2p-n+k-2}$. In view of these facts and \eqref{abeqn},  we conclude \eqref{1.6},  \eqref{1.7} of Theorem \ref{thmA}.   To  show  $  \chi < k,  \chi  $ as in \eqref{1.6},   it suffices to  show  
\begin{align}
\label{3.14}   
k >   \frac{  (p+k-n) ( k  +  p  - 2)}{ (p - 1) ( 2  p - n  + k - 2) } = \be  \frac{  k  +  p  - 2}{2  p - n  + k - 2 }.
\end{align} 
Since   $ \be  <  1,  $   it is easily seen that   \eqref{3.14} is true for  $ p  \geq n. $   
 To  prove that  this inequality holds  for  $ 2 < p < n, $  we   gather terms in  $ k $  to get that  \eqref{3.14} is valid if  
\begin{align}  
\label{3.15} 
I  =  (p  - 2 )[  k^2  + k ( 2p - n - 2)  + n - p ]   > 0. 
\end{align}
Since $ n > p  > n - k $ and  $ p > 2, $  it follows that   
\begin{align}
 \label{3.16} 
 \begin{split}
 I   & >  ( p - 2 )  [   k^2 +  k ( p + n - k  - n - 2)   + n - p ]\\
 & =  (p-2)  [ k ( p - 2) + n - p ] > 0.
 \end{split}
 \end{align}
 Thus   $ \chi  < k. $   
\end{proof} 
\begin{remark} 
\label{rmk3.1}
When $p=n$ the quadratics  $A, B$ have the common roots $\lambda = \pm \beta$.  This checks, when $\lambda = -\beta$ as we already know  the solution $u=s^\beta$  in $ \rn{n}_+ $  and that the $n$-Laplacian is invariant under an inversion.  
\end{remark}



 \section{Proof of Theorem \ref{thmB}}  
 \label{sec4} 
\setcounter{equation}{0} 
 \setcounter{theorem}{0}
   \begin{proof}  
   We   prove  Theorem \ref{thmB}  only when  $ 1 \leq k \leq n - 2, $  since  the  proof when $ k = n - 1$ and $p > 2 $ is essentially the same.   
   Recall  from section \ref{sec1}  that   if   $ p  >  n - k$ with $1 \leq k \leq n-2, $    then    $   u  $ is said to be a  $p$-harmonic Martin function  for  $ \rn{n} \sem \rn{k} $   relative to  $0$  provided  $ u >0 $  is $p$-harmonic in $ \rn{n} \sem \rn{k}, u(x) \to  0,  $ as $ x \to  \infty, x \in \rn{n} \sem \rn{k}, $  and $u$ has  continuous  boundary value $0$  on  $ \rn{k}\sem \{0\}. $  To briefly  outline the proof of  \eqref{1.4}, suppose $ v $ is another  $p$-harmonic Martin function relative  to  $0$.   Applying  \eqref{2.7} of Lemma  \ref{lem2.7}  to $ u, v $ and letting $ \rho \to  0 $ it  follows that 
 $ u/v  \equiv $ a constant in  $ \rn{n}\sem\rn{k}. $   Since $p$-harmonic functions are invariant under dilation we deduce that if $ t > 0, $  and $ u (e_n) = 1, $   then  
\begin{align}
 \label{4.1}  
 u  (t x ) =  u  ( t e_n )  u ( x ) \quad \mbox{whenever}\, \, x \in \rn{n} \sem \rn{k} .  
 \end{align}
Differentiating  \eqref{4.1} with respect to $ t $ (permissible by  \eqref{2.5} $(\hat a)$ of   Lemma  \ref{lem2.5})   and evaluating at $ t =  1 $  we see that   
      \[ 
        \lan x,  \nabla u ( x ) \ran =   \lan e_n,  \nabla u ( e_n ) \ran  u ( x )\quad \mbox{whenever} \,\, x  \in \rn{n} \sem \rn{k} . 
        \]    
If we put $ \rho =  | x |,   x / | x | = \om \in    \mathbb{S}^{n-1}, $   in this identity  we obtain that  
\[  \rho \,   (u)_{\rho}  ( \rho \om ) =  \lan e_n,  \nabla u ( e_n )\ran u ( \rho  \om ).  
\]
Dividing this equality by  $ \rho  u  ( \rho  \om ) , $  integrating with respect to $ \rho, $  and  exponentiating,   we find  for $r>0$ and $\om \in \mathbb{S}^{n-1}$ that   
\begin{align}
 \label{4.1a}  
 u ( r \om )  =  r^{-\si } u ( \om )  \quad \mbox{where}\quad \si   =  - \lan e_1,  \nabla u  ( e_1 ) \ran. 
\end{align}
    For  fixed  positive integer $ 1 \leq k  \leq n - 2$ and  $p  >  n - k$,  let   
$ \hat  u $  be the  $p$-subsolution defined in  \eqref{1.5} where  $ \la $ is 
chosen so that  $ \ch < \la < k$ where $\ch$ is as in Theorem \ref{thmA}.    From   Lemma  \ref{lem2.6} and  the fact (mentioned earlier)  
that $ | x'' |^{\be} $  is  $p$-harmonic in  $  \rn{n} \sem \rn{k}  $   we  see that  
$ u/v  \approx 1 $  on $  \ar B (0, 1 ) \cap ( \rn{n} \sem \rn{k} ) .$  
Comparing boundary values of $  \hat u, u $ and using the fact that  $  \hat u ( x )  +  u ( x ) $  $ \to  0 $ as $ x \to  \infty, x \in \rn{n} \sem 
\rn{k}, $  we deduce from the boundary maximum principle in  Lemma  \ref{lem2.2}  that 
$   \hat u   \leq   c u $  in  $  \rn{n} \sem (\rn{k} \cup B (0, 1)) $  where $  c =  c (p, n. k). $  Letting  $ x  \to  \infty $  we get  
\begin{align}
 \label{4.1b}   
 \si \leq \la < k.  
\end{align}
Next   given  $  0 < t < 10^{-10 n},  $ let     $ a(\cdot)$  be  a   $ C^\infty $  smooth function on  $ \re $
 with  compact support in  $ ( - t, t )$, $0 \leq a  \leq 1,$  $a  \equiv  1 $  on  $ ( - t/2, t/2 ), $  
 and   $ | \nabla a | 
\leq   10^5 /t .$   Let     $ f ( x )   = \prod_{i=1}^n  a(x_i),  x \in \rn{n}, $    and  
 for fixed $ p >  n - k $,  let  $ \ti v $  be  the  unique   $p$-harmonic 
 function on $  \rn{n}\sem\rn{k} $   with  $ 0 \leq \ti v \leq 1$ satisfying 
\begin{align}
 \label{4.2}    
  \int_{\rn{n}\sem \rn{k}}   |\nabla \ti v |^p   dx    \leq  \int_{\rn{n}\sem\rn{k}}   |\nabla  f |^p   dx    \leq   
	  c  t^{n-p} 
\end{align}  
where $ c = c (p, n, k)$ and $( \ti v  - f ) \ze   \in   W_0^{1, p} (B ( 0, \rho )  \sem \rn{k}) $  
whenever  $  0 < \rho < \infty$. Here $\ze $ is as in  Lemma \ref{lem2.4} (for a fixed  $\rho)$. 
Once again existence and uniqueness of  $ \ti v $  
  follows with  slight modification  from  the usual calculus of  variations argument  for bounded domains (see \cite{E}).   
       We  claim     that  
    there exist $  \be_* = \be_* (p,n,k)  \in  ( 0, 1]$ and $c = c(p, n, k, t )  \geq 1$    such that if    $  x, y \in B (0, \rho) \cap (\rn{n} \sem \rn{k}), $ then  
\begin{align}
 \label{4.3}  
  |  \ti v ( x )  -  \ti v (y) |  \leq   c   \left( \frac{|x-y| }{ \rho  } \right)^{\be_*} \quad \mbox{and}\quad  \ti v (x) \leq  \left(\frac{4nt}{|x|}\right)^{\be_*}. 
  \end{align} 
The left hand inequality in  \eqref{4.3}  follows  from  Lemma \ref{lem2.4}.   
To  prove the  right hand 
inequality in   \eqref{4.3}  let  $  \ti v_j, j = 4n + 1,  4n +2, \ldots,  $  be  the $p$-harmonic function  in  $  B ( 0, j) \sem \rn{k} $ with 
continuous boundary values  $ \ti v_j  =  \ti v$ on $B ( 0, j) \cap  \rn{k} $  and $ \ti v_j  = 0 $  on $ \ar B ( 0, j). $ 
Observe from  the  boundary maximum principle in Lemma \ref{lem2.2} 
 and $  0 \leq \ti v_j   \leq 1, $  that  ${\ds \max_{\ar B(0, r)} \ti v_j  }$  is  nonincreasing for     $ r  \in (2nt,  j). $ 
Also using uniqueness of $ \ti v $ in the  calculus of  variations minimizing argument it follows that $ \ti v_j  \to  \ti v $  
uniformly on compact subsets of  $  \rn{n}. $  Thus  ${\ds \max_{\ar B(0, r)} \ti v  }$  is also non increasing as  
a  function of $r.$  Using this  fact and   Harnack's inequality in  Lemma \ref{lem2.3} $(c)$   applied to $ {\ds  \max_{\ar B(0, r)} \ti v    - \ti v } $, 
and  \eqref{2.4} $(++)$  we  deduce the existence of  $  \he \in (0, 1) $ with 
\begin{align}
 \label{4.4}   
\max_{\ar B(0, 2r)} \ti v  \leq   \he   \max_{\ar B(0, r)} \ti v  \quad \mbox{whenever}\quad r > 2n t. 
\end{align}
 Iterating this inequality we  get the right hand inequality in  \eqref{4.3}.

 Next we  show  that    
\begin{align}
 \label{4.4a} 
 \ti  v ( e_n )   \approx     t^{\si}  
  \end{align} 
where   $ \si $ is as in   \eqref{4.1a} and the proportionality constants depend only on $ p, n, k$.  To prove   \eqref{4.4a},  
  put $ \ti u  =  t^{\si} u (x ), x \in  \rn{n}.  $   
Then from  Harnack's inequality and   
\eqref{2.4}  (++)    of   Lemma  \ref{lem2.4} with $  \hat v  = 1 - \ti v $,   we find that  
$ \ti v (  t e_n ) \approx 1. $   In  view of the boundary values of  $  \ti  v$, $ \ti u,  $  and  $ \ti  v (te_n) \approx  \ti u (te_n) = 1,$   
as well as  Harnack's inequality in \eqref{2.2} $(c)$,   we  see that  first  Lemma  \ref{lem2.6}  can be applied to get    
\begin{align} 
\label{4.5}   
 \ti  u/ \ti v   \approx 1   
 \end{align} 
     on  $     \ar B ( 0, 2nt ) \sem  \rn{k}    $  where ratio constants depend only on  $ k, n, p.$     
     Second  from      \eqref{4.3} for $  \ti v, $   and  $ \si > 0 $     we find  that  
$ \ti u(z), \ti v(z) \to 0 $  as  $ z  \to  \infty $ in      $  \rn{n}  \sem \rn{k}  $    and thereupon from 
Lemma \ref{lem2.2}  that    \eqref{4.5}  holds in  $ (\rn{n} \sem \rn{k}) \sem B (0,2nt). $    

Since $ \ti u (e_n)  = t^{\si} $   we conclude from   \eqref{4.5} that \eqref{4.4a} is true.   

Let  $ \ti a $ denote the one periodic  extension of $ a|_{[-1/2, 1/2]}$   to  $ \re. $ 
 That is  $ \ti a  ( r + 1 ) = \ti a ( r )$ for $r \in \re$ and  $ \ti a = a $ on $ [-1/2, 1/2]. $ 
 Also  let   $  \ti \Psi   $   be  the  $p$-harmonic function on 
$ \rn{n}\sem \rn{k} $  with continuous boundary values   on $ \rn{k} $ and 
\begin{align}
 \label{4.6}
 \begin{split}  
& (a) \hs{.2in}   \ti  \Psi ( x + e_i  )  = \ti \Psi ( x ) \, \, \mbox{for}\,\,  1 \leq i \leq k,   \quad \mbox{whenever} \, \, x  \in \rn{n} \sem \rn{k}, \\
& (b)  \hs{.2in}  \ti \Psi (x)  - \prod_{i=1}^n  \ti a(x_i)   \in  R^{1,p}_0 ( S(1))   \quad \mbox{and} \quad 0 \leq \ti \Psi \leq 1  \mbox{ in } \rn{n}\sem\rn{k},   \\
& (c)  \hs{.2in}   { \ds  \int_{S(1)}  | \nabla \ti \Psi|^p  dx dy   \leq  c \, t^{n-p}  <  \infty,   } \mbox{ where }  c = c (p,n,k),  \\
& (d)  \hs{.2in}    { \ds  \lim_{  x'' \to  \infty, x'' \in \rn{n-k}} \ti \Psi ( x' , x'' )  = \xi \,\, \mbox{a constant,  uniformly for}\, \, x' \in \rn{k}. }  
\end{split}
\end{align}
Existence of  $ \ti \Psi $ satisfying $ (a)-(d)$ of \eqref{4.6}  follows  from the discussions after  \eqref{2.8} and \eqref{2.9}.        
Comparing  boundary values  of  
 $ \ti v, $ $\ti \Psi, $   we see  that  $  \ti v  \leq  \ti \Psi $  on   $ \rn{k}. $  
 Using  this fact and Lemma  \ref{lem2.2}  we find in  view of    \eqref{4.3} that 
\begin{align}
 \label{4.7} 
  \ti v    \leq  \ti \Psi \quad   \mbox{in}\quad    \rn{n}\sem  \rn{k}. 
 \end{align}
 Let  $ \hat e $   be that point in  $ \rn{n-k}$ with  $  \hat e = ( \hat e_1,  \ldots, \hat e_{n-k}) $ 
 where $ \hat e_i = 0, 1 \leq i \leq n - k - 1, $ and  $ \hat e_{n-k}  = 1. $     
 From   \eqref{4.7}, \eqref{4.5},   and  Harnack's  inequality for  $  \ti v,  $   we have   
\begin{align}  
\label{4.8} 
\int_{Q_{1/2}^{(k)}(0)}   \ti \Psi (x', \hat e) \, d \mathcal{H}^k x'  \geq  \int_{Q_{1/2}^{(k)}(0)}   \ti v (x', \hat e) \, d \mathcal{H}^k x'\approx t^{\si }. 
\end{align}
   Also from \eqref{4.6}(b),  the definition of $ \ti a,  $  and continuity of  $ \ti \Psi $  in $ \rn{n}$    we obtain     
\begin{align}
 \label{4.9} 
  \int_{Q_{1/2}^{(k)}(0)}   \ti \Psi (x', s \hat  e) \, d \mathcal{H}^k x'  \leq 4^k t^k \,   ,
 \end{align} 
for small $ s > 0.$  Thus there exists $ c = c(p,n,k) \geq 1 $ with 
\begin{align} 
\label{4.10}   
\int_{Q_{1/2}^{(k)}(0)}   \ti \Psi (x', s \hat e ) \, d \mathcal{H}^k x'    \leq  c t^{k -\si } \,   \int_{Q_{1/2}^{(k)}(0)}   \ti \Psi (x',  \hat  e) \, d \mathcal{H}^k x'  \, .  
\end{align}  
We conclude for $ t > 0, $ small enough that Theorem \ref{thmB} is valid  with  
$ \Psi (x)  =  \Psi (x', x'' ) $   one of the functions, $  \ti \Psi (x', x'' + \hat e )  -   \xi $  or  $  \ti \Psi (x', x'' + s \hat e )  -   \xi $ when $ x \in  \rn{n} $.  
\end{proof}

The proof  of  Theorem \ref{thmB}  in case $ (ii), $ i.e.,  when $ k = n - 1,  p > 2,  $ and  in $ \rn{n}_+ $   
is  essentially the same as in case  $(ii), $  only in this case one  uses Remark \ref{rmk1.6}.  Thus  we omit the details.   

\begin{remark} 
\label{rmk4.1}   
We remark that Corollary \ref{cor1.7}  follows from \eqref{4.5}        and  Lemma  \ref{lem2.2}.     
\end{remark} 
 
\section{Proof of Theorem \ref{thmC}} 
\label{sec5} 
 \setcounter{theorem}{0}
\setcounter{equation}{0}
 In this section  we indicate the  changes  in  Wolff's argument that are necessary to show that  Theorem \ref{thmB} implies  Theorem \ref{thmC}.  
 Unless otherwise stated,  we let $ c \geq 1 $  denote a  positive constant which may depend on $ p, n, k, $ and 
 the Lipschitz norm of  $ \Psi |_{\rn{k}}. $   
 Recall the definition of  $ f \in  R^{1,p} (\mathcal{S} (1)) $ and $ \| f \|_{*, p} $  in  \eqref{2.8}  
 when $ k, n $  are fixed positive integers  with  $ 1 \leq k \leq n - 2 $ or $ k = n - 1. $  
 Given   $  h \in  R^{1, p } (\mathcal{S} (1))$  we note that  $ h $ has  a  trace on $ \rn{k}$ 
 which is well defined  $ \mathcal{H}^k $ almost everywhere. Let  $ h|_{\rn{k}} $  denote this trace and extend  $ h $  to $ \rn{k} $    by setting   
  $ h ( x', 0)  =  h|_{\rn{k}} (x', 0)$ for $x' \in \rn{k} $. Also  let $ \| h|_{\rn{k}} \|_{\infty}$  and   
  $ \|  h|_{\rn{k}} \breve \|  =   \| |\nabla h|_{\rn{k}}| \|_{\infty},  $    denote respectively the $ \infty$ and 
  Lipschitz norms of $ h |_{\rn{k}}. $    Let   $ \hat h \in R^{1, p } ( \mathcal{S} (1)),$   be  the $p$-harmonic function  
  on  either $(i)$  $ \rn{n} \sem \rn{k} $  (when $1 \leq k \leq n - 2, p> n - k ) $  or   $(ii)$   $ \rn{n}_+ $ (when $ k = n - 1$ and $p > 2$),     
  with  $ \hat h - h \in R_0^{1, p } ( \mathcal{S} (1)).$  Throughout this section   $ Q_{r}^{(k)} ( x' )  =  Q_{r} (x') $ when $ x' \in \rn{k}$ and $r  > 0. $  
  Also proofs  of  Lemmas  will only be given in case $(i),$  as the proof in case $ (ii) $  is essentially the same.   
 
We first state an analogue  of  Lemma 1.4 in  \cite{W}.  
 \begin{lemma}  
 \label{lem5.1}      
 Suppose    $\bar u, \bar v $   are    $p$-harmonic in either  $(i)$  $ \rn{n} \sem  \rn{k}  $  when  $ 1 \leq k \leq n - 2$ and $p > n - k,  $  or $ (ii)$  $ \rn{n}_+ $ when 
 $ k = n - 1$ and $p > 2,   $ with continuous boundary values.   Also assume that  $ \bar u, \bar v   \in  R^{1, p} ( \mathcal{ S } ( 1 ))  $ with 
   $\| \bar u|_{\rn{k}}  \|_{\infty}  + \| \bar v|_{\rn{k}} \|_{\infty} < \infty  $  and  for some $ z' \in \rn{k},  0 < \ga \leq 1/4, $  that   $  \bar u|_{\rn{k}} \leq  \bar v|_{\rn{k}}$   on  
   $  Q_{2 \ga} ( z' ). $    Let $ 0 < t \leq 1/2$  and if  $ (i) $  holds  put 
   \[ 
   E (t)   =  \{ x  = (x', x'') \mbox{ with }   \bar u (x) - \bar v  (x) > 0  \quad \mbox{and}\quad  x \in  Q_{\ga} (z')   \times \{ x'' :  |x''|   \leq t \} \}  
   \]   
   while if $ (ii) $ holds  replace  $ x'' $  in the above display by  $ x_n$ where $x_n  > 0. $ Then     there exists $c=c (p, n, k, \ga)$ such that
\begin{align} 
\label{5.1}   
\int_{E (t) }  | \nabla ( \bar u  - \bar v  )^{+}   | \,  dx    \leq   c\,  t^{(n-k)(p-1)/p} ( \|  \bar u  \|_{p, *}   +  \|  \bar  v  \|_{p, *} )^{\al } \, \,   [  \max_{ S(1) }  (\bar u  - \bar v )^{+} ]^{1- \al}
\end{align}
    where   $  \al =  1 - 2/p$   and $  a^+  =  \max (a, 0) $.     
\end{lemma}  
   
    \begin{proof}  
    We note that since  $ p > 2$, then
\begin{align}
\label{5.2}   
 ( | \nabla \bar u |^{p-2} \nabla \bar u  - | \nabla \bar v |^{p-2} \nabla \bar v) \cdot  \nabla ( \bar u -   \bar v)  \geq c^{-1} ( |\nabla \bar u |  + |\nabla \bar v | )^{p-2} | \nabla  \bar u  -  \nabla \bar v   |^2  \end{align}
 and
\begin{align} 
 \label{5.3}   
  |  | \nabla \bar u |^{p-2} \nabla \bar u  - | \nabla \bar v |^{p-2} \nabla \bar v  |  \leq c ( | \nabla \bar u | + | \nabla \bar v |)^{p-2} | \nabla \bar u  -  \nabla \bar v |  
\end{align}  
      on $ \mathcal{S}$ (1). 
         Let  $ 0 \leq  \he \leq 1  \in  C_0^\infty ( Q_{2\ga} (z') \times \{ x'' :   |x''|  \leq 1  \}) $  
         with  $ \he \equiv 1 $ on   $ Q_{\ga} (z') \times \{ x'' : | x''| \leq t \}, $  and  $ | \nabla \he | \leq c \ga^{-1}.$     
                  If   $ a^+  = \max(a,0), $  then $ \he^2  ( \bar u - \bar v )^{+} $ can be used 
                  as a test function in the definition of $p$-harmonicity for $ \bar u,  \bar v. $  
                  Doing this, using \eqref{5.2}, \eqref{5.3},  and a standard  Caccioppoli type argument we obtain     
\begin{align}
\label{5.4} 
\begin{split}
{\ds  \int_{E(t)}  | \nabla \bar u - \nabla \bar v |^p  dx } &\leq    {\ds  \int \he^2   | \nabla (\bar u -  \bar v)^+ |^p  dx}   \\
& \leq  c \ga^{-2} {\ds  \int_{S(1) \cap B (z', 2n) } [(\bar u- \bar v)^+]^2    ( |\nabla \bar u | + |\nabla \bar v | )^{p-2} dx}  \\
& \leq c \ga^{-2} \,    (  \| \bar u \|_{*, p}   +  \| \bar v \|_{*, p} )^{p-2}  {\ds \, \, [\max_{S(1) }  (\bar u - \bar v)^+ ]^2   } 
\end{split}
\end{align}
  where $ c = c(p,n,k) $  in \eqref{5.1}-\eqref{5.4}.  
\end{proof}        
Next we state an analogue of Lemma 1.6  in  \cite{W} which the authors view as  Wolff's main lemma for applications.

\begin{lemma} 
\label{lem5.2}    Let  $  k, n  $ be  positive integers  with either  $(i) $  $ 1\leq k \leq n -   2 $ and  $ p > n - k  $   or  $(ii)$ $ k = n - 1 $ and $ p > 2$  fixed.  
Let $ \ep \in (0, 1) $ and  $1 \leq M < \infty.$  Then there are constants $ A = A (p, n,  k, \ep, M ) > 0  $
and $ \nu_0 = \nu_0 ( p, n, k, \ep, M) < \infty, $  such that if  $  \nu  >   \nu_0  > 100  $ is  a positive  integer, 
$ f$, $g \in R^{1,p} (\mathcal{S} ( 1 )),$ $q \in R^{1,p} (\mathcal{S}(\nu^{-1})),$   and if      
\begin{align}
 \label{5.5}   
 \max ( \| f|_{\rn{k}} \|_{\infty}, \| g|_{\rn{k}} \|_{\infty}, \| q|_{\rn{k}} \|_{\infty},  \| f|_{\rn{k}} \breve \|, \| g|_{\rn{k}}  \breve \|,  \nu^{-1} \| q|_{\rn{k}} \breve \| )   \leq M, 
\end{align} 
   then  for  $ x = (x', x'')  \in  \mathcal{S}$(1),  $ \al = 1 - 2/p, $  and $ 1 \leq  k \leq n - 2, p > n - k, $      
\begin{align}  
\label{5.6}   
|  \widehat {f q +  g} (x) -  f (x', 0) \,  \hat q (x ) - g (x',0 ) |   <   \ep  \quad \mbox{if}\, \,    | x'' |  \leq   A \nu^{ - \al}.     
\end{align}
     If, in addition,  $   \hat q (y)  \to   0, $  uniformly as $ y \to  \infty,  y \in \mathcal{S}$(1),   then 
   \begin{align}
\label{5.7} 
|  \widehat{ q f + g } \,  ( x',  x'' )   -   g ( x' , 0 ) | < 2  \ep \quad \mbox{if}\, \, |x ''| = A \nu^{-\al} ,      
\end{align}    
and     
\begin{align} 
\label{5.8} 
| \widehat{q f + g} (x)  -   \hat g (x) | < 3 \ep    \quad \mbox{if } \, \,  | x'' |  \geq   A  \nu ^{-\al} . 
\end{align}   
If   $ k = n - 1, p > 2, $  replace $ x''  $ by $ x_n, x_n > 0 $ in  \eqref{5.6}-\eqref{5.8}.      
\end{lemma}  
 
 \begin{proof}   
 We note that our proof scheme is similar to   Wolff's  but details are somewhat different.     
 The   first step in the proof  of  \eqref{5.6} is to show for given $  \be  \in (0, 10^{-4}),$  that  \eqref{5.6} holds  for some 
  $ A = A (p, n,  k, \ep, M, \be  ) > 0  $   with   $  \be \nu^{-1} \leq  |x''|  \leq  A \nu^{-\al} $ provided  $ \nu \geq \nu_0 (p,n,k,\ep,M).$  
      To do this let 
\[
 J (x)  = \widehat{f q +  g} (x) -  f (x', 0) \,  \hat q (x ) - g (x',0 ),  \quad \mbox{for} \, \, x \in \mathcal{S}(1). 
\]      
Now suppose that
\begin{align}  
\label{5.9} 
|J (z)|    >  \ep 
\end{align}
 where   $z=(z', z'')$ with  $z'  \in  Q_1 (0)$,   $ | z''| = t, $    and   $  \be \nu^{-1}  \leq t < 1/4. $  
   Let  
   \[
   H =  \widehat{fq+g}\quad \mbox{and}\quad K  = f(z',0) \hat q -  g(z',0). 
   \]
Then  $ H$ and $K$  are both  $p$-harmonic in  $ \rn{n} \sem \rn{k}  $ when  $ 1 \leq k \leq  n - 2$
and in $ \rn{n}_+ $  when $ k = n - 1$  with  continuous  boundary values. 
Also from  \eqref{5.5}  and   $ H,  K  \in  R^{1,p} (\mathcal{S}(1)), $
  we deduce  that
\begin{align}
 \label{5.10} 
 |  H ( x',0) - K (x',0) | \leq  ( M^2 + M ) | x' - z' |   \quad \mbox{for} \, \, x' \in \rn{k}. 
\end{align} 
   From   translation invariance of $p$-harmonic functions  and  the maximum principle for $p$-harmonic functions
    in  Lemma \ref{lem2.2} we see for $x \in \rn{n}$ and  $1\leq j \leq k$ that  
\begin{align}  
\label{5.11}  
 \hat q  ( x  + i e_j /\nu )  = \hat q ( x ) \quad  \mbox{for every integer $i$}.  
\end{align}
From  \eqref{5.11}, \eqref{5.10},  Lemma \ref{lem2.2}, and translation invariance of $p$-harmonic functions it  follows that if  
\[  
\La =  \left\{\sum_{j=1}^k   i_j e_j/\nu, \quad  \mbox{where}\, \, i_j, 1 \leq j \leq k,\, \, \mbox{are integers} \right\} 
\] 
   then for  $ \hat \ze \in \La, $     
\begin{align}
 \label{5.12} 
 \begin{split}
 | H  ( z' + \hat \ze , z'' ) - H  ( z', z'' ) | +  | K   ( z' + \hat \ze , z'' ) &- K ( z', z'' ) |  \\
  &\leq  {\ds\max_{\rn{k}}}\,  | H  ( x' + \hat \ze , 0 ) - H ( x' ,  0 ) |\\
  & \leq   (M^2 + M ) |\hat \ze|. 
  \end{split}
  \end{align}
Next we note  from  Lemma \ref{lem2.3} $(b)$  and $ H, K   \in  R^{1,p} (\mathcal{S}(1)), $   that there exists  $ \rho = \rho( p,n,k,\ep,M) \in (0, 1/2),  $  such that    
if $ x  \in   B (  z +   \hat \ze ,  \rho t )$, then   
\begin{align}
     \label{5.13} 
     | H  ( x )   - H (  z + \hat \ze   ) |    +    | K ( x   )  - K ( z + \hat \ze  ) |<  \ep/1000
\end{align}
 whenever   $\hat \ze  \in \La$.   Let     $ \ga  =  \min (  \frac{\ep}{ 10^3 n ( M^2 + M )}, 1/4). $      
 Then  from  \eqref{5.10}  we observe that  
\begin{align} 
 \label{5.14}  
  |H - K | \leq \ep/100  \quad \mbox{on} \, \,  Q_{2\ga} (z'). 
  \end{align}
 Also without loss of generality  assume that 
\begin{align}
\label{5.15}    J ( z )  =  (H - K) (z) > \ep. 
\end{align}   
   Let $  \ti \La =  \{  \hat \ze \in \La  :  z'  +  \hat \ze  \in Q_{\ga} (z') \}$   and  suppose    $  \ga \geq  10^3  \nu^{-1} .$
    Put    
    \[   
    W ( t )  := Q_{\ga} ( z')  \cap \{  x' :  \min_{\hat \ze \in \ti \La} | x' - (z' +  \hat \ze  )| < \rho t / (10 n^2)  \}. 
    \]    
    Then either   $ W (t) =  Q_{\ga} ( z' )  $ or  
\begin{align}  
\label{5.16}    
Q_{\rho t } ( z' + \hat \ze  ) \cap    Q_{\ga} ( z' ) \not = \es   \mbox{  for   }   \approx   (\ga \nu) ^k \mbox{ points }  \hat \ze  \in \ti \La,  
\end{align}  
where  proportionality constants depend only on $ k.$  From \eqref{5.16} and   $  \be/\nu \leq t \leq 1/4, $   we conclude in either case that    
\begin{align} 
\label{5.17}  
1   \leq  c ( p,n,k, \ep, M, \be) \, \mathcal{H}^k (W(t)).  
\end{align}
 Also  if  $ y'  \in W(t) $  and  $ k \leq n - 2, $   we see from  \eqref{5.12}-\eqref{5.15},  the definition of 
 $ \ga, $  $ \nu^{-1} \leq 10^{-3} \ga ,$ that   there is  a Borel set  $ F ( t, y') \subset  \{ (y', x'')  \in \rn{n} :  | x'' |  =  t   \} $  satisfying  
\begin{align} 
\label{5.18}    t^{ n - k - 1}  \leq    c ( p, n,  k, \ep, M)   \mathcal{H}^{n - k - 1} ( F ( t, y' ) ),  
\end{align}
  and  the  property  that     if  $ ( y', w'')   \in F (t, y'), $ then  $  ( y', w'') \in   B ( z +   \hat \ze  , \rho t ) $ for  some 
  $ \hat \ze  \in \ti \La,$   with
\begin{align}  
\label{5.19}     
\begin{split}
( H  -  K)    (y', w'')      &\geq   \bar   (H - K )  ( z +  \hat \ze  /\nu, z''  )  -  \ep/ 100  \\
& \geq      (H - K ) (z)  - 2 \ep/100  \\
& \geq  98 \ep/100.  
\end{split}
\end{align}
   Let  $ G(t,y')  =  \{ \om \in \mathbb{S}^{n-k-1}  :  t \, \om \in F(t, y') \} $  and $ \bar u = H$,  $\bar v  = K + \ep/4. $        
   From  \eqref{5.14} we see that  $ \bar u  - \bar v  < 0 $  on $ Q_{2\ga} (z'). $   
   Using  this observation,   \eqref{5.18},  \eqref{5.19},  and  continuity of    $ H, K  $ in $ \rn{n} $  
   it follows  that  for some $ \ti c  = \ti c ( p, n, k, \ep, M ) \geq 1,  $         
\begin{align} 
\label{5.20}   
\begin{split}
1    &\leq \ti c \,  \, \left[ \int_{G(t, y') }  (\bar u  - \bar v )^+ ( t \om )   d \mathcal{H}^{n - k - 1} \om \right]^p   \\ 
&\leq   \ti c  \, \,  \left[\int_{G(t,y')}  \int_{0}^t   | \nabla ( \bar u - \bar v )^{+}  ( r \om ) | dr d \mathcal{H}^{n - k - 1} \om  \right]^p. 
\end{split}
\end{align}
Let   $ \hat E (t, y') $ be the  set  of  points consisting of  line segments with one endpoint $y'$ 
 and the other endpoint in  $ G ( t, y').$  Switching to polar coordinates  
 we see   from  \eqref{5.20}  and H\"{o}lder's inequality  applied to   $  |\nabla  (\bar  u  - \bar v )^{+}  (r \om) | r^{ (n -k-1)/p}$ and $r^{( - n +k+1)/p}, $  that 
  \begin{align} 
  \label{5.21}    
  \left[\int_{G(t,y')}  \int_{0}^t   | \nabla ( \bar u - \bar v )^{+} ( r \om ) | dr d \mathcal{H}^{n - k - 1} \om  \right]^p    \leq  c  \, t^{p + k - n } \,  \int_{\hat E (t,y')}  |\nabla ( \bar u -  \bar v )^{+} |^p d \mathcal{H}^{n - k}
 \end{align}
 where $c=c (p,n,k,\ep,M)$. Using  \eqref{5.21}, \eqref{5.20}, \eqref{5.17},  and integrating over $y' \in  W (t) $  we get     
\begin{align}
\label{5.22}  
1    \leq    \bar  c (p,n,k,\ep,M,\be)  \, t^{p + k - n } \,  \int_{E (t)}  |\nabla ( \bar u -  \bar v )^{+} |^p dx  
\end{align}
where $ E (t) $ is as in Lemma \ref{lem5.1}.  Applying  Lemma \ref{lem5.1} and using Lemma \ref{lem2.2},  \eqref{5.5},   we arrive at  
\begin{align}
 \label{5.23}  
 1  \leq    \,  c ( p, n, k, \ep, M, \be)   t^{p + k - n } \,     ( \| H \|_{*,p}   +  \| K\|_{*,p}  \|)^{ p - 2}  
\end{align}
 Let  
 \[
 h ( x )  =  [  f (x',0) q (x', 0 ) + g ( x', 0)  ] \, \ph ( |x''| ),
 \]
 when $ x \in \rn{n}$  where $ \ph \in  C_0^{\infty} ( - 2/\nu, 2/\nu ) $     with  $  \ph  \equiv 1 $ on  $(-1/\nu, 1/\nu)$  
 and  $  |\ph'| \leq  1000 \, \nu$. Then  $ h - H  \in W^{1,p}_0  ( \mathcal {S} (1) ), $  so   
\begin{align}  
\label{5.24}    
 \| H \|^{p - 2 }_{*,p} \leq  \left(\int_{\mathcal {S} (1)}    | \nabla h |^p  dx\right)^{\al}  \leq    c (p,n,k)  (M^2 + M)^k   \nu^{\al (p + k - n)}. 
\end{align} 
 Likewise one gets the same estimate for $ \| K \|^{p-2}_{*,p}$  as for   $ \| H \|^{p-2}_{*,p}$ 
in  \eqref{5.24}.    Using  these  estimates in  \eqref{5.23}  we   conclude  that     $  c ( p,n,k,\ep, M, \be) \,   t   >   \nu^{-\al} $   provided  $ \nu \geq \nu_0  (p,n,k,\ep,M,\be) $.

  To complete the proof of \eqref{5.6} it remains to  fix $  \be = \be (p,n,k,\ep,M) $  and show \eqref{5.6}  holds for 
 $ 0 < t  \leq \be/\nu. $  
  To do this  we apply   \eqref{2.3} of    
Lemma  \ref{lem2.4}    with   $ \hat v = \widehat{qf+g}$,  $\hat  q,  $  and 
with $ \rho =  \be^{1/2} \nu^{-1}$, $\hat \si = 1$,  $M'  =   (M^2 + M)  \nu$,   to get   
for $  |x''|   <  \be \nu^{-1},$       
\begin{align} 
\label{5.25}  
\begin{split}
  | J ( x ) | = | J (x)   -  J ( x', 0)  |  &\leq    c ( M ) \left(  \nu \,  ( \be^{1/2} \nu^{-1}    )    +     ( \frac{ \be \nu^{-1}  }{\be^{1/2}  \nu^{-1}}   )^{\hat \si_1}\right)   \\
  &\leq  c' ( M ) \,   \be^{\hat \si_1/2} .  
\end{split} 
  \end{align} 
 Choosing $ \be = \be (p,n,k,\ep,M ) > 0 $ small enough and then fixing  $ \be $ we obtain   \eqref{5.6}  from \eqref{5.25} for     
$ t  <  \be \nu^{-1}.$       

To prove  \eqref{5.7} we note from \eqref{5.5} and  \eqref{2.9} of  Lemma  \ref{lem2.8} that   
\begin{align} 
\label{5.26}   
| \hat q ( x )  |  \leq    2 M    (   | x''| \nu  )^{- \de}   
\end{align}    
since  $\hat q (x)  \to  0 $ as  $|x''|  \to  \infty$.  Choosing  $ |x''| =  A \nu^{-\al}$  and  $ \nu_0,
 $  still larger if necessary  we get  \eqref{5.7} from  \eqref{5.26}.   To  prove  \eqref{5.8} 
observe  from   \eqref{2.9} of Lemma \ref{lem2.8}  with  $ \hat v = \hat g$ and  $\rho =
 A \nu^{- \al/2}   $  that    
 \begin{align}  
 \label{5.27} | \hat g (x) -  g(x', 0) | < \ep   \quad \mbox{when} \,\, |x''| = A \nu^{-\al}
 \end{align}
  for $ \nu_0 = \nu_0 ( p,n,k,\ep, M )$ large enough. Now \eqref{5.8}  follows from \eqref{5.27}  and  Lemma \ref{lem2.2}.   This finishes the  proof of Lemma \ref{lem5.2} when $ 1 \leq k \leq n - 2$ and $ p  > n - k.$   The same conclusion holds if $ k = n - 1. $  However in this case the proof is somewhat  simpler since 
$G(t, y') $ is a point  and $  F (t, y') $ is  a line segment so  one can write a version of   \eqref{5.21}  with a single integral. 
 \end{proof} 

\subsection{Lemmas   on  Gap Series}  
Throughout the rest of this section we write $ dx' $  for  $ d \mathcal{H}^k x' $  when $x' \in \rn{k}. $      The  examples  in  Theorem \ref{thmC} will be  constructed when  either $(i)$ $ 1 \leq k  \leq n  - 2$ and $p > n - k, $  or $ (ii)$ $ k = n -1$ and $p > 2, $  using   Theorem    \ref{thmB}, 
as   the  uniform limit on compact subsets of  $  \rn{n} \sem \rn{k}  $ or $ \rn{n}_+ $   of a sequence of  $p$-harmonic  functions
    whose boundary values are  partial sums     of       $ a_j  L_j (x')  \Psi ( T_j x', 0), \,  $  where $ \Psi $ is as  in  Theorem \ref{thmB}, $ (L_j)$ is to be defined,  
    and  $ (T_j) $  is a sequence of positive integers satisfying
\begin{align}
\label{5.28} 
\begin{split}
  & T_1 = 1, \\
  & T_{j+1}\, \, \mbox{is an integer multiple of $T_j$ and}\, \,   T_{j+1}    \geq   4 T_j   [ \log ( 2 + T_j ) ]^3 \quad \mbox{for}\, \, j = 1,2,\ldots. 
    \end{split}
\end{align}    
 Lemma   \ref{lem5.2}  will be used  to make estimates on  this sequence. 
 Throughout the rest of this section we  also assume, as we may,  that if $ \Psi $  is as  in  Theorem \ref{thmB}, then  
 \begin{align}
 \label{5.28a}   \| \Psi (x', 0) \|_\infty +  \|  \Psi (x', 0) \breve \|  \leq  1/2 .
\end{align}
          Also, let    $   \bar b   =  \int_{Q_{1/2}  (1)}  \Psi (x', 0)  dx' $  and  let  $(a_j) $ be a sequence of real numbers with 
 \begin{align}  
 \label{5.29}  
 \hat \chi^2 := \sum_{j=1}^{\infty}  a_j^2    <  \infty. 
 \end{align}
     
  \begin{lemma} 
\label{lem5.3} 
Let  $ \psi (x') = \Psi (x',0) - \bar b$  for $x' \in \rn{k}  $  and put  $ \psi_j (x')  =  \psi ( T_j x')$  for $j = 1, 2,  \ldots. $   
If        
  \[ 
  s^* (  x' ) : =  \sup\limits_{l}   \left|   \sum_{ j = 1}^l   a_j  \psi_j ( x' )  \right|, \quad \mbox{for}\, \, x' \in \rn{k}, 
  \]
 then  for $ \la > 0, $     
\begin{align}  
\label{5.30}  
\la^2 \,   \mathcal{H}^k (  \{ x' \in Q_{1/2} (0) :  s^* ( x'  ) >  \la  \} )    \leq  c \,   \,   \hat \chi^2   
\end{align}
  where $ c = c(k)  \geq 1.  $     Consequently,    
\begin{align}
 \label{5.31}  
 s ( x' )  : = 	   \lim_{l  \to  \infty}  \sum_{j=1}^l     
	  a_j   \psi_j ( x' )  \mbox{ exists for $ \mathcal{ H}^k $   almost every }   x'  \in   \rn{k}.  
\end{align}
\end{lemma}   
	   \begin{proof}   
	   For  $ m = 1, 2, \ldots, $ let  $ G_m  $    denote the set of all open  `cubes` $ Q $  in $ \rn{k} $ 
	   with side length  $r_m =  1/ T_m $ and  center  at 	   $ r_m \tau  = (r_m \tau_1, r_m \tau_2, \ldots r_m \tau_k ) $  
	   where $  \tau_i, 1 \leq i \leq k, $  are integers.  If   $     Q  \in   G_m  $  and  $ j > m $ then from  \eqref{5.28}     
	   we see that $ \mathcal{H}^k $ almost everywhere  
\begin{align}
 \label{5.32} 
 Q =  \bigcup \{ Q' \in  G_j :  Q' \subset Q \}.  
 \end{align}  
	  Also from \eqref{5.28},   and the definition  of    $  (\psi_j ),  $ we  see that if $ m < j $ and $ Q'  \in G_j,  $ then   
	  \[
	  \int_{Q'} \psi_j dx' = 0.
	  \]
Therefore, if $ y' \in Q',$ then
\begin{align}  
\begin{split}  
\label{5.33} 
\left| \int_{Q'} \psi_m \psi_j dx' \right|  &=  \left| \int_{Q'} ( \psi_m - \psi_m (y') ) \psi_j dx'  \right| \\ 
&\leq  k^{1/2}  \,  (T_m/T_j)  \, \mathcal{H}^k (Q') . 
\end{split}
\end{align}
Summing over  $ Q'  \subset Q, $  and then  over  $ Q \in G_m $  with  $  Q \subset Q_{1/2}(0), $   it follows from  \eqref{5.33}  that    
\begin{align}
 \label{5.34}  
 |\int_{Q_{1/2} (0)} \psi_m \psi_j dx' | \leq  k^{1/2} \, T_m/T_j. 
\end{align} 
   Now from \eqref{5.28}   we observe  that  
\begin{align}
 \label{5.35}   
 T_m/T_j    \leq  ( j - 1)^{-3}  4^{m - j} \quad \mbox{for $ m = 1, 2, \ldots $ and $ j = m + 1 , \ldots $. } 
 \end{align}
	 Let 
	 \[  s_l (x')   =     \sum_{j=1}^l  a_j  \psi_j (x') ,   \, \mbox{ for } l = 1, 2,  \ldots, \mbox{ and }    x'  \in \rn{k}. 
	 \]  In view  of  \eqref{5.35}, \eqref{5.34}   we  get  
\begin{align}	 
	  \label{5.36}
\begin{split}	  
	  \sup_{l }  \int_{Q_{1/2} (0)}  s_l^2  dx'   &	      	 	      =  \sup_{l}    \sum_{m, j = 1}^l    a_m a_j \int_{ Q_{1/2} (0) } \psi_j  \psi_m dx'  \\ 
	      	 	        & \leq \hat \chi^2   +   2 \sum_{m=1}^{\infty}   |a_m| \, \left( \sum_{j = m+1}^l     | a_{j}| \left| \int_{Q_{1/2} (0) } \psi_j \psi_m dx' \right|\right) \\
		      	 	      & \leq   \hat \chi^2  \left[ 1   +  2  \sum_{m=1}^{\infty}  \, \left( \sum_{j = m+1}^{\infty} k^{1/2}  (j-1)^{-3} 4^{m-j} \right) \right]  \\
		      	 	      & \leq   3  k^{1/2}   \hat \chi^2.
\end{split}
\end{align}		      	 	      
	Using \eqref{5.36} with $ s_l $ replaced by  $  s_l    - s_i,  i  < l,  $  and  letting $l, i \to  \infty  $ through certain sequences we  see from   \eqref{5.29}   that   $ \lim_{j \to  \infty}  s_j  =   s  $  in   the $ \| \cdot \|_{2} $ norm of   $ Q_{1/2} (0). $ 
	Moreover  \eqref{5.36} is valid with $ s_l $ replaced by $  s. $   
	
	Next  let $q$ be a large positive integer  and   set             
	\[
	 \breve s (x')  :=   \sup_{1\leq l  \leq q }   |s_l (x')|  \quad \mbox{for} \, \, x' \in Q_{1/2} (0). 
	 \]
	We note from Cauchy's  inequality that     
		$   s_l^2  \leq   \hat \chi^2  \,  l. $    Thus  if  $ i $ is a positive integer and  $   s_l (x')   >    i \hat  \chi ,   $    then $  i^2 <  l. $  
	 With $ i $ fixed, $  i  $  a positive integer, 
	let    	 $  K_i $   be   those cubes  $ Q' \in  \bigcup_{j=1}^{q} G_j $   with $ Q' \in  G_l,   $    if   and only if    
	$  l $ is  the smallest  positive integer  satisfying  $ |s_l (y')| > 8  k^{1/2}  i  \hat \chi $ for  some $ y'  \in Q' .$ Clearly the cubes in $ K_i $ are disjoint.   
	Note   that   if $ Q' \in K_i, $ then $\{ x' : \breve s (x') >   8  k^{1/2} i \hat \chi \} \cap Q'  \not = \es  $  and $ ( \cup_{m=1}^{64} G_m ) \cap K_i  = \es. $     
	  	Also  if   $  Q'  \in K_i \cap  G_l $,  it   follows from     \eqref{5.28},   \eqref{5.28a},  as in the last line of  \eqref{5.36},   that for $ x' \in Q',   $  
\begin{align}
\label{5.37} 
\begin{split}
(8  k^{1/2}i + 1)    \hat \chi     &\geq  |s_l (x')|   \geq    8 k^{1/2}  \hat \chi i   -  |s_{l-1}  (x') -  s_{l-1} (y')|  -    k^{1/2}  \hat \chi   \\
&	\geq    7 k^{1/2}   \hat \chi i   -  k^{1/2}   \hat  \chi  \sum_{j=1}^{l-1}  T_j /T_l  \\
& \geq  6  k^{1/2}   \hat \chi i.  
\end{split}
\end{align}
Next  using  \eqref{5.34},  \eqref{5.35},  we obtain      
\begin{align} 
  \label{5.38}     
\begin{split}  
   \int_{Q_{1/2} (0)}  s_q^2   dx'   & \geq  \sum_{l=1}^{q}  \sum_{Q' \in K_i \cap G_l }  \int_{ Q' \cap Q_{1/2} (0) }  s_q^2   dx'   \\
&	 \geq \sum_{l=1}^q   \sum_{ Q' \in K_i \cap G_l }  \int_{Q'  \cap Q_{1/2} (0) } [ s_l^2 + 2 ( s_q - s_l ) s_l ]  dx'  \\ 
	 & =         \sum_{l=1}^{q}  \sum_{Q' \in K_i \cap G_l} \left(   \int_{Q' \cap Q_{1/2} (0) } s_l^2 dx'  -   	  \sum_{j=l+1}^q   \sum_{m=1}^l 2 a_j a_m    \int_{ Q' \cap Q_{1/2} (0) }  \psi_j        \psi_m   dx' \right)  \\ 
&     \geq   \sum_{l=1}^q   \sum_{ Q' \in K_i \cap G_l }  \int_{Q'  \cap Q_{1/2} (0) }  s_l^2  dx'  -   4k^{1/2}   \hat \chi^2    \sum_{ Q' \in K_i  } \mathcal{H}^k (Q' \cap Q_{1/2} (0)) 
\end{split}
\end{align}
 where the last inequality  is  proved once again  using the same argument as in  the last line of  \eqref{5.36}.   From \eqref{5.36}, \eqref{5.37},  and \eqref{5.38}  we conclude that   
\begin{align} 
\label{5.39}  
\begin{split}
64  k   \hat \chi^2  i^2  \mathcal{H}^k  (  \{ x' \in Q_{1/2} (0) :   \breve s ( x')  >  8 k^{1/2}  \hat \chi i      \})  &\leq  64  k   \hat \chi^2  i^2  \,  {\ds   \sum_{Q' \in K_i}} \mathcal{H}^k ( \bar Q' \cap Q_{1/2} (0) ) \\
& \leq   2    \sum_{l=1}^q   \sum_{ Q' \in K_i \cap G_l }  \int_{Q'  \cap Q_{1/2} (0) }  s_l^2  dx'   \\
& \leq 2  \int_{Q_{1/2} (0)}   s_q^2  dx'   +   8 k^{1/2}    \hat \chi^2   \\
& \leq  14 k^{1/2}   \hat \chi^2.  
\end{split}
\end{align}
   Letting  $  q \to  \infty $ in \eqref{5.39}  and  using the  definition of  $ \sup $    we  find   after some elementary algebra  that  \eqref{5.30} 
	     is true.   To prove  \eqref{5.31} we can now use   \eqref{5.30}  and a standard argument.   Indeed  
	     observe  that  if $ r > 0,  $   then  
\[     
V   =   \{  x'   :   \limsup_{l\to  \infty}  s_l (x')  -    \liminf_{l \to  \infty} s_l (x') > r  \}  \subset       \{x' : s^* (x') >   r/2\}.    
\]    
Also  $  V $     is  unchanged  if   we  put any finite number of 
	     $ a_j  = 0. $  Using  these observations  and  \eqref{5.30}  we get  $ \mathcal{H}^k ( V ) = 0. $ Since $ r > 0 $ is arbitrary we  conclude from our earlier work that 
	     $  s_l (x')  \to  s(x')  $  as $ l \to  \infty $    for   $ \mathcal{H}^k $ almost every $ x'  \in \rn{k}. $  
 \end{proof}    
\subsection{Construction of Examples}  
 Let  $(T_j) $ be as in \eqref{5.28}, $ (a_j), \hat \chi  $ as in  \eqref{5.29}   and  $ \bar b, \psi,  (\psi_j),  (s_j),s, s^*,  $ as in  Lemma  \ref{lem5.3}.   
Let   $  \ti \psi_j  =  \psi_j  + \bar b$ for $j = 1, 2, \ldots$. 
 For our first example  we choose   $ (a_j)_1^{\infty}$   in addition to  \eqref{5.29}  so that if 
 \[
 d_m = \sum_{j=1}^m a_j,  \quad \mbox{for} \, \, m = 1, 2, \ldots,
 \]
 then
\begin{align} 
\label{5.40}   
|d_m| \leq C < \infty \quad \mbox{and} \quad {\ds \lim_{m\to  \infty}}  d_m \, \,  \mbox{does not exist}.  
\end{align}
   Also put 
\begin{align} 
\label{5.40a}   
\ti s_l = \sum_{j=1}^l  a_j \ti \psi_j =    s_l  +  \bar b  \sum_{j=1}^l a_j  \quad \mbox{when} \, \, x'  \in \rn{k}. 
\end{align}
 From \eqref{5.31}, \eqref{5.40},  and   $ \bar b \neq 0, $ (thanks to Theorem \ref{thmB}),  we see that  
\begin{align}
 \label{5.41}  
 \sum_{j=1}^{\infty} a_j  \ti  \psi_j \mbox{ diverges $ \mathcal{H}^k $  almost everywhere. } 
 \end{align} 
Also  we construct  a sequence of  functions  $ (L_j) $   satisfying     
\begin{align} 
\label{5.42}  
\begin{split}
& (a) \hs{.2in}  L_1 \equiv 1 \quad \mbox{and}\quad  L_j ( x'  + e_l ) = L_j (x')  \quad \mbox{for $ x' \in \rn{k}$},\, \, 1\leq l \leq k, \, \, \mbox{and}\,\,  j  = 1, 2, \ldots \\
& (b) \hs{.2in}  1/ 2  \leq  L_{j+1}/ L_j  \leq 1 \quad \mbox{and}\quad  \| L_{j+1} \breve \| \leq  c^* (k)  T_{j}, \quad  \mbox{for} \, \, j = 1, 2, \ldots.
\end{split}
\end{align}

Moreover we shall make the construction so that 
\begin{lemma} 
\label{lem5.4} 
If    $ \si_m = {\ds  \sum_{j=1}^m   a_j L_j \ti \psi_j ,  \,  m = 1, 2, \ldots,} $    then for some $ 0   <  C' = C' (  k, \bar b,  \hat \chi ) < \infty, $  we have 
 \begin{align}
 \label{5.43} 
  \sup_m |   \si_m  (x' ) |  <  C'  \mbox{ for all } x' \in \rn{k} \quad \mbox{and} \quad  \sum_{j=1}^{\infty} a_j L_j  \ti \psi_j \mbox{ diverges $ \mathcal{H}^k $ almost everywhere.} 
\end{align}  
    \end{lemma} \begin{proof} 
  To construct  $ (L_j)$  we proceed by induction:  $ L_1 \equiv 1 $ and  if   $ L_j $   
  has been defined  so that  \eqref{5.42} is true for $ j $  let  $ K_i $ be as  defined  above   \eqref{5.37} with $ q  =  \infty $  and $ G_j $   
as defined  after   \eqref{5.31} with $ m = j. $ 
   If   $ Q' \in G_{j} \cap (\bigcup_i  K_i) , $  put   $ L_{j+1} \equiv  (1/2) L_j $  on   $ \bar Q' . $   Let   
  \begin{align*}
  \breve E  &=   \bigcup_{i} \{ \bar Q' :  Q'  \in K_i \cap G_{j} \},  \\   
  F_1 &=  \{ x' \in \rn{k} :  d (x', \breve E) \geq  \frac{1}{4T_j} \}, \\ 
   F_2 &= \{ x' \in \rn{k} :  d (x', \breve E) \geq  \frac{1}{8T_j} \}.
\end{align*}      
   If   $ \breve E  \not =  \es, $  let  $ 0 \leq  \he \in  C_0^{\infty} ( Q_{1/2} (0) ) $ with  $  | \nabla \he |  \leq c'(k) $   
   and  $\int_{Q_{1/2} (0) }  \he (y' )  dy' \equiv 1.$   Let $  \chi_{F_2} $  denote the characteristic function of $F_2$ and let $\ep =   \frac{1}{16 k^{1/2} T_j}$. Set  
   \[ 
   \ze_j (x' )  =   \ep^{-k} \int_{\rn{k}}  \he (  \ts \frac{x' - y'}{\ep} )   \,  \chi_{F_2} (y' )  dy'  \quad \mbox{when}\, \, x'  \in \rn{k} .
   \]       
   Then one easily verifies that     
\begin{align}
\label{5.44}  
 \ze_j \equiv 0 \quad \mbox{on}\,   \breve  E, \quad \ze_j  \equiv 1  \quad  \mbox{on} \, F_1,  \quad  0 \leq \ze_j  \leq 1, \quad \| \ze_j \breve \|  \leq  c' (k) T_j.  
\end{align}
Put    
\begin{align}
\label{5.45}  
L_{j+1} = (1/2) (  \ze_j +  1)  L_j .  
\end{align}
 If  $   \breve E  = \rn{k} $,  let $ L_{j+1} = L_j . $    From   \eqref{5.44}, \eqref{5.45},  the induction hypothesis, and the definition of  $ (T_j) $  in  \eqref{5.28}   we see that  
\begin{align}   
\label{5.46}   
| \nabla  L_{j+1} (x') | \leq (1/2)  c'(k)  T_j  + c^* (k) T_{j-1}  \leq  c^{*} (k) T_j  \quad \mbox{for}\, \, c^* (k) =  c'(k).
\end{align}
 The rest of the induction hypothesis is also easily checked  using \eqref{5.44},  \eqref{5.45}.  
 Thus by induction we have defined   $ (L_j )_1^{\infty}$  satisfying   \eqref{5.42}.

  To begin the proof   of  Lemma \ref{lem5.4}  we  note that    if  $ x' \in  Q'  \in K_i \cap G_j $ 
  then $ |L^{1/2}_{j+1}\,  \ti s_j| $  is   uniformly bounded.   
   Indeed  from \eqref{5.28a},  \eqref{5.29},  we have  $ | s_{l+1} -  s_l | \leq    \hat \chi $ 
   so  there exist indices  $  l_1 < l_2 < ... < l_{8i}  \leq j $  with   $  L_{l_m +1} (x')  = (1/2) L_{l_m} (x'). $  Thus  
\begin{align} 
\label{5.47}  
| L^{1/2}_{j+1} (x')  \ti s_{j} (x') |  \leq     2^{- 4 i} \, [ (8i k^{1/2} +1)  \hat \chi  + C  |\bar b| ] = \ti C. 
\end{align}
 
To prove  \eqref{5.43}  we use  \eqref{5.47}   and  following Wolff (see proof of Lemma 2.12 in \cite{W}) integrate by parts to obtain 
\begin{align}  
\label{5.48}  
\begin{split}
| \si_m (x')  |   & \leq {\ds  |  \sum_{l=1}^m  (  L_{l}  (x') - L_{l+1} (x') )   \ti s_l (x') |  +  c(k)  ( | \bar b | C  +    \hat \chi) }    \\   
& \leq  {\ds    \sum_{l=1}^{\infty}  (  L_{l}  (x')^{1/2}  - L_{l+1} (x')^{1/2}) (  L_{l}  (x')^{1/2}  + L_{l+1} (x')^{1/2} ) | \ti s_l (x') |  + c(k)  ( | \bar b | C  +   \hat \chi) } \\
 &  \leq   (\sqrt{2} + 1)  \ti C  {\ds    \sum_{l=1}^{\infty}  (  L_{l}  (x')^{1/2} - L_{l+1} (x')^{1/2} )  + c(k) ( | \bar b | C +  \hat \chi )  \leq C'   }  
\end{split}
\end{align} 
    where $ C' $  is as in \eqref{5.43}. To prove the last statement in  Lemma  \ref{lem5.4}, 
    given $ r > 0 $ and a  cube $ Q \subset \rn{k}, $ let $ r Q $ denote the cube $ \subset \rn{k} $  
    with the same center as $ Q $ and side length  $= r$ times the side length of $ Q. $  If  $ x' \in  Q_{1/2} (0), $   
    we see from the definition of  $ (L_j ), (\ze_j), $  that either   
   $ L_{j} (x') =  L_m (x') $  for some  integer $ m $  and $ j \geq m, $ or at least one of $ (a), (b)$ is true where
\begin{align}  
\label{5.49} 
\begin{split}    
&(a)  \hs{.2in} \mbox{  For arbitrary large $m$   there is  a  $ Q' \in  \bigcup_{i = m}^{\infty}   K_i  $  with $ x' \in \frac{5}{4} Q' $  }  \\ 
&    (b) \hs{.2in}   \mbox{ For some $ i $  and  arbitrary large $ m $   there is  a  $ Q' \in K_i $  with }   \\ 
&  \hs{.46in}   \mbox{sidelength  $  \leq  1/T_m $ and  $ x' \in  \frac{5}{4} Q'. $ } 
\end{split}
\end{align}          
Now  from \eqref{5.39} and  basic measure theory it follows that the set of  all $ x' $ in $ Q_{1/2} (0) $  
for which  either \eqref{5.49} $(a)$ or $(b)$ holds  has $ \mathcal{H}^k$ measure 0.   Moreover  if $  L_j (x') $ is eventually  constant then  from    
  \eqref{5.41}  we  deduce  that     $  \sum_{j=1}^{\infty} (a_j L_j  \ti \psi_j)(x')  $  diverges $ \mathcal{H}^k $ almost everywhere.  
  \end{proof}  

Lemma \ref{lem5.4} will be used  to construct $ \hat u$ and for this  we  need  the next   lemma. 
    
\begin{lemma} 
\label{lem5.5}  
For $j =1, 2, \ldots$, let     $   a_j     =    -  \frac{1}{4j}$ and define  $ T_j, \psi_j,  \ti \psi_j, G_j,  $ as in 
Lemma \ref{lem5.4} for $ \,   j =1,2, \ldots  $  Also define $ s_m,  \ti s_m $ relative to $(\psi_j), (\ti \psi_j),$  and the current $ (a_j)  $ as in 
\eqref{5.40a}.  There is  a  choice of   $ ( L_j ) $ satisfying     \eqref{5.42}   such that  if   $    \ti  \si_m  =   1 +   {\ds \sum_{j=1}^m } a_j  L_j \ti \psi_j, $  then  
\begin{align} 
\label{5.50}
\begin{split}
&(a')  \hs{.2in}  \ti \si_m   > 0 \quad \mbox{for}\, \,  m = 1, 2, \ldots  \\ 
&(b') \hs{.2in}  \sup  \ti \si_m (x')  <  c(k)  < \infty  \quad \mbox{for all}\, \,  x' \in  \rn{k},  \\
&(c') \hs{.2in} \si ( x') ={ \ds \lim_{m \to \infty} \ti  \si_m }  ( x' ) = 0   \quad  \mbox{for $ \mathcal{H}^k $  almost every }\, \,  x' \in \rn{k}. 
\end{split}
\end{align}
 \end{lemma} 
\begin{proof}  
Lemma \ref{lem5.5}     is    essentially  a $k$-dimensional  version of  Lemma
 2.13   in \cite{W}.  To   begin the  proof  let  $ i, m $ be positive integers and  for fixed $ m, $ let  
$  \mathcal{K}_{im}$ for $i = 1, 2, \ldots, $  be the set of    $ Q  \in G_m$ for which  $ \max_{\bar Q} \ti s_m  > i $ 
and $ Q$ is not contained in  a $ Q' \in  \mathcal{K}_{im'} $  for some $ m' < m.$    
Set $ L_1 \equiv 1 $  so that   $ \ti \si_1  =  1+ a_1 \ti \psi_1. $  Next  define  $ (L_m) ,  (\ti \si_m),$  
and sets of cubes,     $  \mathcal{F}_{i m},$ \, $\mathcal{H}_{im},$ by induction   as follows:       
 Suppose  $ L_m,   \si_m $   have  been defined for  $ m  \leq  j. $  
   Assume also that  
   $ \mathcal{F}_{im}, \mathcal{H}_{im}  \subset  G_m $ have been 
   defined  for  nonnegative integers $ m < j $  and all positive integers $i$ 
  with $ \mathcal{F}_{i0} = \es  =   \mathcal{H}_{i0}.$      If  $ i $ is  a 
  positive integer and  $ Q  \in G_j,  $  we  put  $ Q  \in \mathcal{F}_{ij} $  
  if  $  \min_{\bar{Q}} \ti \si_j  <  2^{-i} $  and  this cube is not contained in  any cube in $  \mathcal{F}_{im} $   
  for some $  m < j.$    Moreover we put  $ Q  \in  \mathcal{H}_{ij} $ 
  if     $  \min\limits_{\bar Q} \ti s_j  <  - \frac{ 2^{i}}{i+1}$  and 
  $ \max\limits_{\bar Q} L_j > 2^{-i}. $    We then  define  $ \ze_j,   L_{j+1}$ as in \eqref{5.44}, \eqref{5.45},  only now     
\begin{align}
\label{5.51}  
\breve E  =  \left\{ \bar Q' :  Q'  \in  \bigcup_i (  \mathcal{F}_{ij}\cup   \mathcal{H}_{ij} \cup  \mathcal{K}_{ij} ) \right\}. 
\end{align}
  Arguing as in  \eqref{5.46} we then  get   \eqref{5.42}.    
With  $ L_{j+1} $  defined  we  put $\ti  \si_{j+1} =  \sum_{m=1}^{j+1}  a_m L_m \ti \psi_m$  and after that define $ \mathcal{F}_{i(j+1)},   \mathcal{H}_{i(j+1)}, $ for all positive integers $i.$   
By induction we conclude the definitions of $ (L_j), (\ti \si_j), ( \mathcal{F}_{ij}), ( \mathcal{H}_{ij}). $     

  From the definition  of $\ze_j, L_{j+1}, $ in   \eqref{5.44}, \eqref{5.45},    as regards   $ \breve E $  in \eqref{5.51}, 
  and the same argument as in the proof of  \eqref{5.47}  in   Lemma \ref{lem5.4} we deduce  that  
  \[   
  L_j (x')^{1/2}  \,   \max (\ti s_j (x'),0)    \mbox{  is uniformly bounded  for all $ x' $  in $ \rn{k}$.}  
  \]   
  Using this fact  and  arguing as in \eqref{5.48}     we get   \eqref{5.50} $(b').$    
  
  Next  if  
  $ 2^{-(m+1) }   \leq \min_{Q} \ti \si_j  <  2^{-m} $  and  $ L_{j+1} \leq 2^{-m} $  on $ Q \in G_j, $ then from  \eqref{5.28a} and our choice of $(a_l), $  we see that  
  \[  
  \ti \si_{ j + 1}  \geq \ti \si_j   -  2^{-(m+2) }  \geq 2^{-(m+2)} \quad \mbox{on}\, \,  Q.
   \]    
   Using   this   observation and  the definition  of $\ze_j, L_{j+1}, $ in   \eqref{5.44}, \eqref{5.45},    
   as regards   $ \breve E $  in \eqref{5.51},   one can show by induction on $m$ that  for a  positive integers $l,$     
\begin{align}  
   \label{5.52}  
   \mbox{  if  $ Q  \in G_l $   and $ \min\limits_Q \,  \ti \si_l < 2^{ - m}$  then $ L_{l+1} < 2^{-m} $ on $ Q$}.  
\end{align}   
  \eqref{5.52}   implies  \eqref{5.50} $(a') $  
 since  $\ti  \si_1 > 0 $  and  if    $   2^{- ( m + 1) }  \leq  \ti \si_{l}  <   2^{-m} $  
 on $ Q  \in G_l,$ then  from \eqref{5.52},  \eqref{5.28a},  and the definition of $ (a_j), $       
 \[
 \ti \si_{ l + 1}  >   \ti \si_l   -  2^{ - ( m + 2)} > 0   \quad  \mbox{on}\, \, Q.
 \]    
Thus   \eqref{5.50} $(a') $  holds.  
   
 It remains to prove  \eqref{5.50} $(c').$ 
To do so we first claim that  for $l=1,2,\ldots$,
\begin{align}
\label{5.53} 
c(k)^{-1}  \, L_{l+1} (x' ) \leq  L_{l+1} (y' )  \leq  c (k) \, L_{l+1} (x') ) 
\end{align}
whenever $x, y \in  \frac{5}{4}Q  \in G_l$.  Indeed from  \eqref{5.44}, \eqref{5.45}, we have  
\[ 
\frac{L_{l+1} (y' )}{ L_{l+1} (x' )} \leq  \frac{\prod_{j=1}^l  ( 1/2)(1 + \ze_j (y') )}{ \prod_{j=1}^l 1/2}   \leq   \frac{\prod_{j=1}^l  ( 1/2)(1 +  c'(k) T_j/T_l  )}{ \prod_{j=1}^l 1/2}  \leq c(k). \] 
The lower estimate is proved similarly.

       To prove  \eqref{5.50} $(c')$    let  $ E_m $  denote the  set of  all  $  x' \in \rn{k} $  
        for which there exist  $ l_1$ and $l_2  $  positive integers with 
        $ l_1  < l_2 $  satisfying
\[
\ti s_{l_2} (x')   >  -  \frac{ 2^m}{2(m+1)} \quad      \mbox{while} \quad  \ti s_{l_1}(x')  < -  \frac{ 2^m}{m+1}  \, . 
\]
       Since  $ a_j  =  -  (1/4) j^{-1}$ for $j = 1, 2, \ldots $  it follows that    
\begin{align}
 \label{5.54}  
 \begin{split}
 \max [|  s_{l_1} (x')|,|  s_{l_2} (x')|] &= \max \left[ | \ti s_{l_1} ( x')  + (\bar b/4)  \sum_{j=1}^{l_1} j^{-1}  |\, , \, |  \ti s_{l_2}  ( x' ) + (\bar b/4)  \sum_{j=1}^{l_2}  j^{-1}  |\right ]  \\  
 & \geq   \frac{2^m \bar b}{ 8 (m+1)}
\end{split}
\end{align}
 for $m\geq 100$. If we let 
  \[    
   \Ga  :=   \{ x' \in \rn{k}  :  x'  \in E_{m} \mbox{ for infinitely many $ m $}  \} \bigcup 
  \{ x'  \in \rn{k}   :  \limsup_{j\to \infty}  \ti s_j (x' )  > -  \infty \}
  \]
	  then using   \eqref{5.54} and  \eqref{5.30} of  Lemma \ref{lem5.3} we arrive at  
\begin{align} 
\label{5.55}     
|  \Ga  |  = 0.
\end{align}
Next from induction on  $ m  $   and   the definitions of  $ \mathcal{H}_{lm},  L_{l+1}, $  it follows    that
    if   $  \ti s_l ( y' )  <  -  \frac{2^m}{m+1}, $  $ y' \in \bar Q \in G_l, $   
    then   $ L_{l+1 } ( y' )  \leq 2^{-m}. $  Therefore if  $y'  \not \in \Ga, $  
    then 
\begin{align}  
\label{5.55a}    
    \lim_{l \to \infty} \ti s_l \,  (y' ) = - \infty  \quad \mbox{and} \quad \lim_{l \to \infty} (\ti s_l \,  L_{l +1} ) (y' ) = 0. 
\end{align}   
 Now     \eqref{5.55a},     \eqref{5.50} $(a'),  $    and  
\begin{align} 
\label{5.56}   
0  <  \ti  \si_j   ( y')   = 1 +  \sum_{l \leq  j }  ( L_l  - L_{l+1} ) \ti  s_l  ( y' )   +   s_{j} L_{j+1} ( y')
\end{align}
 imply  that if    $ y'  \not \in \Ga, $  then  it must be true that $ \lim_{j \to \infty }\ti  \si_j (y' )$ exists and is non-negative.

 Suppose this limit is positive.  In this case  we observe from the definition of $(a_l), $  
 \eqref{5.28},  \eqref{5.28a},  \eqref{5.42},    that if    $\be_l =   T_l^{-1} ( \| \ti \si_l \breve \| + \| \ti s_l   \breve \| ), l =1,2, \ldots$ then
\begin{align}
 \label{5.57}  
 \sup_l \be_l \leq c(k) \quad \mbox{and} \quad \be_l  \to  0 \, \, \mbox{as}\, \, l \to  \infty.
 \end{align} 
It follows from \eqref{5.57}, \eqref{5.50} $(a')$,  and the facts $ \lim_{j \to  \infty}\ti s_j (y') = - \infty, $   
    $ \ti \si (y') > 0,$    that 
\begin{align}  
\label{5.58}  
\sup\limits_{j} \{  \max_{ \frac{5}{4} \bar Q }  \ti s_j  : y' \in Q  \in G_j \}   < \infty \quad \mbox{and} \quad
   \inf\limits_{j} \{   \min_{\frac{5}{4} \bar Q}\ti \si_j  : y'\in  G_j \}   >  0 . 
\end{align}   
      So   $  y' $ belongs to at most a finite number of  $ \frac{5}{4} Q $  with   
$ Q \in  \mathcal{K}_{lm}   \cup \mathcal{F}_{lm}  $  for  $ l, m = 1, 2, \ldots $. 
Now if  $ y' \in \frac{5}{4} \bar Q,  \, Q  \in  \mathcal{H}_{i(m+1)}, $  
then from     \eqref{5.53}   and  $ y' \not \in \Ga, $   \eqref{5.57},  we find  for $ i \geq  i_0 (y')$ and $m \geq m_0 (y'),  $ sufficiently large that 
\begin{align}
  \label{5.59}   c(k)  L_{i} (y')  \geq  2^{-m}    \quad \mbox{and}\quad  \ti s_{i'} (y')  <  - {\ts \frac{2^{m-1}}{m}}  \quad \mbox{for}\, \, i' \geq i.
  \end{align}     
 Using   \eqref{5.55a}, \eqref{5.59},  we  deduce the existence of  an increasing  
sequence $ (i_l)   $ for $ l \geq l_0 $  so that 
\[  
L_{i_l} ( y')  = 2^{ - l }  \quad \mbox{and}  \quad c(k)  \ti s_{i_{l+1}}  (y')  \leq    { - \ts \frac{ 2^{-l}}{l+1}}.  
\]       
It then follows   from  \eqref{5.56} that  
$ \ti\si ( y' ) = - \infty $  which contradicts  $ \ti \si (y') > 0. $ 
Thus $\ti \si ( y')=0$  for $ \mathcal{H}^k $  almost every $ y'  \in \rn{k}$  and the  proof of Lemma \ref{lem5.5} is complete. 
 \end{proof} 
  
\subsection{Final Proof of Theorem \ref{thmC}}
       To  finish the  proof  of  Theorem \ref{thmC}  we again follow   Wolff in \cite{W}   and use  
        Lemmas  \ref{lem5.2}, \ref{lem5.4}, and  \ref{lem5.5} to construct examples.    
        Let  $ T_1 = 1$ and  by  induction suppose  $ T_2,  \ldots, T_l $ have been chosen,    
        as in  \eqref{5.28}.   Let $ (a_j),  (\psi_j), (\ti \psi_j)  $ be as in Lemmas   \ref{lem5.4}, \ref{lem5.5}.  
        First we  define   $  \si_j,  \ti \si_j ,  1\leq j  \leq l,  $   relative to these sequences, and after that  $ L_{j+1},  \si_{j+1}, \ti \si_{j+1}. $   
              Next we  define    $ T_{l+1} $  satisfying several conditions:  First suppose    \eqref{5.28} is valid for $ j =  l $.   
               Let   $ g = \si_{l},$ or  $ \ti \si_{l} $, $f = a_{l+1} L_{l+1}$, and   define  $ q   = \ti \psi_{l+1}   $    relative to 
               $ T_{l+1}. $    Also  suppose that      
\begin{align} 
\label{5.60}   
\max ( \| f \|_{\infty}, \| g \|_{\infty}, \| q \|_{\infty},  \| f \breve \|, \| g \breve \|, T_{l+1}^{-1} \| q \breve \| )   \leq M 
\end{align}
  where $ M = M ( T_1, \ldots, T_{l}) $ is a constant.     Next   apply  Lemma \ref{lem5.2}  with  $ M $ as in 
    \eqref{5.5},  $  3 \ep = 2^{-(l+1)}, $ obtaining $ A = A_l $  and  $  \nu_0 =  \nu_0 (p,n,k,l,M) $ 
    so that \eqref{5.6}-\eqref{5.8} are valid.    Choosing  $ T_{l+1} $ still larger if necessary  we may assume that   $ T_{l+1} > \nu_0 $ 
    and  $ A_{l} T_{l+1}^{- \al }  <  \frac{1}{2} A_{l-1} T_l^{- \al}  $      
        where $  \al $   is as in  Lemma  \ref{lem5.2}.   
        By induction we now get  $ (\si_l)$  or  $(\ti \si_l)$ as in 
        Lemma \ref{lem5.4} or  Lemma \ref{lem5.5}.  Moreover  if   $  \si'_j \in  \{ \si_j, \ti \si_j \},  j = 1, 2, \ldots, $ then  in case $(i)$  of Lemma \ref{lem5.2} we have                      
\begin{align} 
\label{5.61}   
| \hat \si'_{j+1}  ( x  )  -  \hat \si'_j ( x ) | < 2^{ - (j+1) }  
      \quad  \mbox{ when }  |x''|  >   A_{j} T_{j+1}^{-\al},
\end{align}  
and     \begin{align}
\label{5.62}   
| \hat \si'_{j+1}  ( x )  -  \si'_j (  x' )  | < 2^{ - (j + 1) }  + |a_{j+1}| \quad   \mbox{ when }  |x''|  <    A_{j} T_{j+1}^{-\al} \,  .     
\end{align}
  From   \eqref{5.61}   we see that $ (\hat \si'_{j+1}) $  converges  uniformly  on compact 
  subsets of $ \rn{n} \sem \rn{k} $   to  a  $p$-harmonic  function  $  \hat \si'  $   satisfying  
\begin{align}
 \label{5.63}   
 | \hat \si'  ( x  )  - \hat \si'_l ( x ) | < 2^{ - l }      \quad    \mbox{ when }  |x''|  >  A_{l} T_{l+1}^{-\al}.  
\end{align}
Using  \eqref{5.62}- \eqref{5.63},   and the triangle inequality we also have  
for $   A_{l+1}  T_{l+2}^{-\al} < |x''|  <      A_{l}  T_{l+1}^{-\al}    $    that      
\begin{align}
\label{5.64}   
\begin{split}
|\hat  \si' (x )  -   \si'_l  ( x' )|  &\leq     |  \hat \si' (x)  -   \hat \si'_{l+1} ( x )|   +  |   \hat \si'_{l+1} ( x )   -   \hat \si'_l  ( x'  ) |  \\ 
&  < 2^{ - (l+1)} +  2^{-l}  +  |a_{l+1}| .
\end{split}
\end{align} 
          From \eqref{5.64}  and  our choice of  $ (a_l) $  we see for  $ \ze $  as in Theorem  \ref{thmC}  and $  (\si_l),  $  as in Lemma \ref{lem5.4} that  
          $ {  \ds \lim_{ x'' \to 0} }  \hat \si' ( x', \ze (x'')  ) $  does not exist for  $ \mathcal{H}^k $  almost  every  $  x'  \in \rn{k}  $   while 
          if $ (\ti \si_l) $  is as in Lemma \ref{lem5.5},  then   $ \hat \si' > 0 $ on $ \rn{n} \sem \rn{k} $   and    $ {\ds \lim_{ x''  \to   0} }  \hat  \si' ( x', \ze (x'' ) ) = 0 $  for  
          $ \mathcal{H}^k $  almost every  $ x' \in \rn{k}. $ 
                             Moreover from  uniform  boundedness of  $ (\si_l), (\ti \si_l),   $  and  the maximum principle for $p$-harmonic  functions we deduce that 
          $  \hat \si' $ is bounded.      
          To complete   the proof of  Theorem \ref{thmC} in case $(i)$,  put  $  \hat \si' =  \hat u $  or   $ \hat  \si' = \hat v $  depending on whether  $ (a_j)$   as in    
          Lemma  \ref{lem5.4} or   Lemma  \ref{lem5.5}, respectively,  was used to construct $ \hat  \si'. $    
          
          The same argument gives Theorem \ref{thmC} in case $(ii). $ This finishes the proof of Theorem \ref{thmC}.   

\begin{remark} 
\label{rmk5.6}
As mentioned in  Remark \ref{rmk1.8},  there  is   no  analogue of  Theorems \ref{thmB} or \ref{thmC} when $ 1 < p \leq n - k.$  
 To  prove  this  assertion  for $ \hat u, $  given $ 0  < r <  1/10,$  let        
\[  
\ph (x) = \max\left[   1 + \frac{\log(|x''|)}{\log (1/r)} \, ,  0\right], \quad \mbox{for} \, \, x \in \rn{n} \sem \rn{k}.  
\] 
If       $  y' \in \rn{k}  $                                    and  $ \he \in C_0^{\infty} (B(y',1)), $ then $ \,  \ph \, \he   $  
can be used as  a test function in the definition of  $p$-harmonicity for  $ \hat u  $ in \eqref{1.3}.  Thus  
\begin{align} 
\label{5.65}  
0 =    \int |\nabla \hat u|^{p-2} \lan \nabla \hat u,  \nabla  \ph \ran   \he  dx   +  
\int |\nabla \hat u|^{p-2} \lan \nabla \hat u,  \nabla \he  \ran  \ph dx  = I_1    +  I_2 . 
\end{align}
 To estimate $ I_1, I_2, $    we  observe from uniform boundedness of $ \hat u $  that  $ \hat u \, \ph\, \he $  can also  be used as  a  test function  in  \eqref{1.3}.   Doing this and  using  H\"{o}lder's inequality in a  Caccioppoli type argument, we obtain 
\begin{align}
\label{5.66}   
\int  |\nabla \hat u |^p \,  (\ph \he )^p \, dx  \,  \leq \,  c(p,n,k)  \| \hat u \|^p_{\infty}\, \int  |\nabla (\ph \he) |^p dx . 
\end{align}
Clearly,   $ 0 \leq \ph  \leq 1 $   in $ B (y',1) $ and   $ \ph ( \cdot, r )  \to  1 $  uniformly as $ r \to  0 $  on compact subsets of $ B (y',1) \sem \{0\}. $  Moreover   
\begin{align} 
\label{5.67}  | \nabla \ph (x)  |   \leq   |x''|^{-1} [ \log (1/r)]^{-1}  \quad \mbox{for}\, \,   x \in B (y',1) \cap \{x:|x''| \geq r \}. 
\end{align}
Using these inequalities in  \eqref{5.66} and the Lebesgue  monotone convergence theorem  we  get  for $ 1 < p \leq n - k,$  
\begin{align} 
\label{5.68}
\begin{split}
    \int  |\nabla \hat u |^p \,  \he^p \, dx  \,  &\leq \,  c'(p,n,k,\he)   \| \hat u \|_{\infty}^p 
  [1 +  \limsup_{r \to  0} \, \left\{  [\log(1/r)]^{- p}    \int_{r}^1   \rho^{ n - k - p - 1 }  
d \rho \right\} ]   \\
&\leq  c''(p,n,k,\he)  \| \hat u \|_{\infty}^p 
    [1  +   {\ds \lim_{r\to  0}}  (\log(1/r))^{1-p}\, ]     \\
    &= c''(p,n,k,\he)  \| \hat u \|_{\infty}^p 
\end{split}
\end{align}
 Armed with  \eqref{5.67}, \eqref{5.68}, and H\"{o}lder's  inequality we can now  estimate  $ I_1$ and $I_2 $  in \eqref{5.65}.   We find that 
 \[ 
 I_2  \to    \int |\nabla \hat{u}|^{p-2} \lan \nabla \hat{u},  \nabla \he \ran dx  \quad \mbox{and}\quad   I_1 \to  0 \quad \mbox{as}\,\, r \to  0. 
 \]      
  We conclude from this inequality, \eqref{5.65},  and  \eqref{5.68} that $\hat u$ extends to  a uniformly bounded  $p$-harmonic function on  $ \rn{n}. $      Using  Liouville's theorem for  bounded entire $p$-harmonic functions (an easy consequence of  \eqref{2.2} $(b)$) we conclude that $ \hat u = $ constant.   Thus  Theorem \ref{thmC} does not have an analogue when $1< p \leq n - k$ for $ \hat u. $  A similar argument yields that   Theorem \ref{thmB} does not have an analogue  and that    Theorem \ref{thmC} does not have an analogue  for $ \hat v. $  
     \end{remark} 
     \section{$\mathcal{A}$-harmonic Functions}  
     \label{sec6} 
\setcounter{equation}{0} 
 \setcounter{theorem}{0}
\subsection{Definition and Basic  Properties of  $\mathcal{A}$-harmonic Functions}  
In this subsection we introduce $ \mathcal{A}$-harmonic functions  and  discuss their basic properties. 
\begin{definition}  
\label{def6.1}	
	Let $p,  \al \in (1,\infty) $ and 
	\[
	\mathcal{A}=(\mathcal{A}_1, \ldots, \mathcal{A}_n) \, : \, \rn{n}\sem \{0\}  \to \rn{n},
\]
   be such that $  
\mathcal{A}= \mathcal{A}(\eta)$  has  continuous  partial derivatives in 
$ \eta_k$ for $k=1,2,\ldots, n$  on $\rn{n}\setminus\{0\}.$  
We say that the function $ \mathcal{A}$ belongs to the class
  $ M_p(\alpha)$ if the following conditions are satisfied whenever  $\xi\in\mathbb{R}^n$ and
$\eta\in\mathbb R^n\setminus\{0\}$: 	
\begin{align*}
		(i)&\, \, \mbox{Ellipticity:}\, \,  \alpha^{-1}|\eta|^{p-2}|\xi|^2\leq \sum_{i,j=1}^n \frac{\partial  \mathcal{A}
_{i}}{\partial\eta_j}(\eta )\xi_i\xi_j   \quad \mbox{and}\quad   \sum_{i=1}^n  \left| \nabla  
 \mathcal{A}_i  (\eta)  \right|   \leq\alpha \, |\eta|^{p-2},\\
(ii)&\, \, \mbox{Homogeneity:}\,\, \mathcal{A} (\eta)=|\eta|^{p-1}  \mathcal{A}(\eta/|\eta|).
	\end{align*}
	\end{definition}   

We  put  $ \mathcal{A}(0) = 0 $  and note that  Definition \ref{def6.1}  $(i)$ and $(ii) $ implies  that    
\begin{align}
 \label{6.1}
 \begin{split}
& (a') \hs{.2in} 
  (|\eta | + |\eta'|)^{p-2} \,   |\eta -\eta'|^2   \leq  c(p,n,k,\al)  \lan   \mathcal{A}(\eta) - 
\mathcal{A}(\eta'), \eta - \eta' \ran,  \,  \\ 
& (b') \hs{.2in}   | \mathcal{A}(\eta) - 
\mathcal{A}(\eta')| \leq c (p,n,k,\al)  ( |\eta| + |\eta'| )^{p-2} | \eta - \eta' |, 
\end{split}
\end{align}
whenever $\eta, \eta'   \in  \rn{n} \sem \{0\}.$  
\begin{definition}
\label{def6.2}                       
	Let $p\in (1,\infty)$ and let $ \mathcal{A}\in M_p(\alpha) $ for some $\alpha\in (1,\infty)$. Given an  open set 
 $ O  $ we say that $ u $ is $  \mathcal{A}$-harmonic in $ O $ provided $ u \in W^ {1,p} ( G ) $ for each open $ G $ with  $ \bar G \subset O $ and
	\begin{align}
	\label{6.2}
		\int \lan    \mathcal{A}
(\nabla u(y)), \nabla \he ( y ) \ran \, dy = 0 \quad \mbox{whenever} \, \,\he \in W^{1, p}_0 ( G ).
			\end{align}
	 We say that $  u  $ is an  $\mathcal{A}$-subsolution ($\mathcal{A}$-supersolution) in $O$ 
	 if   $  u \in W^{1,p} (G) $ whenever $ G $ is as above and  \eqref{6.2} holds with
$=$ replaced by $\leq$ ($\geq$) whenever $  \theta  \in W^{1,p}_{0} (G )$ with $\theta \geq 0$.  
As a short notation for \eqref{6.2} we write $\nabla \cdot \mathcal{A}(\nabla u)=0$ in $O$.  
\end{definition}  

\begin{remark}    
\label{rmk6.3}
We remark  for  $O, \mathcal{A}, p,  u,$  as in  Definition \ref{def6.2}  that if   $F:\mathbb R^n\to \mathbb R^n$  is  the composition of
a translation and a dilation
 then 
 \[
 \hat u(z)=u(F(z))\, \,  \mbox{is}\, \,  \mathcal{A}\mbox{-harmonic in} \, \,  F^{-1}(O).
 \]
 Moreover,  if   $ \ti  F:\mathbb R^n\to \mathbb R^n$  is  the  composition of
a translation,  a dilation, and a rotation 
   then 
   \[
   \ti  u (z) = u (\ti F(z) )\, \, \mbox{is}\, \,  \ti {\mathcal{A}}\mbox{-harmonic in}\, \,  \ti F^{-1}(O)\, \,  \mbox{for some}
  \, \, \ti{\mathcal{A}}\in M_p(\alpha).
  \]  
\end{remark}  

We note that $\mathcal{A}$-harmonic PDEs have been studied in \cite{HKM}. Also   dimensional properties of the
 Radon measure associated with  a positive $\mathcal{A}$-harmonic function $u,$  vanishing  on a portion of the boundary of  $ O, $   
 have been studied in \cite{Akman,ALV,ALV1,AGHLV},  see also \cite{LLN,LN}.  As mentioned in the introduction  
 our goal in this  section is  to  discuss  validity of  Theorems \ref{thmB} and \ref{thmC},  for  $ \mathcal{A}$-harmonic functions. 
 To this end we state   
 
 \begin{lemma} 
 \label{lem6.4}  
 Lemmas \ref{lem2.2}-\ref{lem2.5} are valid  with  $p$-harmonic  replaced by $ \mathcal{A}$-harmonic.   However constants may also depend on  $ \al$. 
 \end{lemma}  
 \begin{proof}  
 References for proofs    of  Lemmas \ref{lem2.2}-\ref{lem2.5} were purposely chosen  to  be  references for  proofs  in the   $ \mathcal{A}$-harmonic setting. 
 \end{proof} 
 
 Next   given $  p,  1 < p < \infty, $    suppose  $ f : \rn{n} \sem \{0\}  \to   (0, \infty) $  satisfies:
\begin{align}
\label{6.3}
\begin{split}
(a)&\, \,   f  (  t   \eta )  =  t^p   f ( \eta ) \quad \mbox{when}\, \,   t  > 0 \, \, \mbox{and}\, \,  \eta \in  \rn{n}.
\\
(b)& \, \,   \mbox{There exists $  \hat \al    \geq 1 $  such that if  $ \eta, \xi   \in  \rn{n} \sem \{0\},$  then } \\
&  \hs{.44in}      \hat \al^{-1}  \,   | \xi  |^2     | \eta |^{p-2}  \,  \leq   \,  
\sum_{i, j = 1}^n  \frac{ \ar^2  f }{ \ar \eta_i \ar \eta_j}  ( \eta )   \,  \xi_i   \,  \xi_j     \leq             \hat \al  \,  | \xi  |^2     | \eta |^{p-2}.   \\
(c)&\,\,   \mbox{There exists $  \al'  \geq 1 $  such that  for  $ \mathcal{H}^{n}$-almost every   $ \eta   \in  B ( 0, 2 )  \sem  B( 0, 1/2),$  }  \\
& \hs{1.47in}   
     {\ds    
\sum_{i, j,k = 1}^n }  \left|  \frac{ \ar^3  f }{ \ar \eta_i \ar \eta_j \ar \eta_k}  ( \eta )   \right|    \leq             \al' .  
\end{split}
\end{align} 
 We note that  $ \mathcal{A} = \nabla f  \in  M_p (\al) $ for some $ \al \in (1,\infty). $   
 \begin{lemma}  
 \label{lem6.5} 
 Lemmas \ref{lem2.6} and \ref{lem2.7},   are  valid when $   k = n - 1$ and  $p > 2,   
 $  or  $ k =  1$ and $p > n-1$ with  $p$-harmonic replaced  by $ \mathcal{A}$-harmonic whenever  $ A \in M_p (\al)$.  Constants may also depend on $\al.$ 
  If   $ 1< k  \leq n - 2$ and $n \geq  3$ with $p > n - k,$  Lemmas \ref{lem2.6} and \ref{lem2.7}  are  valid  when $ A =  \nabla f$  and  \eqref{6.3} holds. 
  Constants may also depend on $ \hat \al,  \al'. $  
 \end{lemma} 
    \begin{proof} 
    References given for Lemmas \ref{lem2.6} and \ref{lem2.7}  provide proofs  for Lemma  \ref{lem6.5}.  
 \end{proof}
  An  $ \mathcal{A}$-harmonic Martin function relative to $ z  \in   \rn{k}  $  is  defined as in Definition \ref{def1.1}  with   
  $p$-harmonic replaced by $\mathcal{A}$-harmonic.  Using     Lemmas  \ref{lem6.4} and   \ref{lem6.5}  
  one can now  argue  as at the beginning of  section \ref{sec4}  to show  the existence of an $ \mathcal{A}$-harmonic  Martin function satisfying  \eqref{1.4} $(a), (b), $  with $p$-harmonic replaced by  $ \mathcal{A} \in M_p (\al)$-harmonic   when   $  k = n - 1$ with $p >2$  and  $ k =  1$ with $p > n-1,$  or    under the  additional assumption that  $  \mathcal{A}  =  \nabla f ,   f  $  as in \eqref{6.3},  when $ 1< k \leq n - 2$ with $p  > n - k. $

 Next for  fixed $ p, n, k,  \al,   \mathcal{A} =  \nabla  f $   as in \eqref{6.3},   and either $  1 \leq k \leq n - 2$ and $p > n - k$  or $ k = n - 1$ and $p  > 2,   $  we claim for given   
  $ F  \in  R^{1, p} (S (\tau)), $  that there exists a  unique 
$\mathcal{A} = \nabla f$-harmonic function  $ \hat  v $   on  $ \rn{n} \sem \rn{k}  $  when $  1  \leq k \leq n - 2$ and $p > n - k$   and on 
$ \rn{n}_+ $  when $ k  =  n - 1$ and $p > 2, $  with   
\beq \label{6.4}  
\hat  v ( z +  \tau e_i  ) = \hat  v ( z ),   1 \leq i \leq k,  \mbox{  for    $z \in  
\rn{n}\sem\rn{k}  $ or $\rn{n}_+,$  satisfying  $ \hat v   - F  \in R_0^{1,p} (S(\tau)). $ } 
\eeq   

In fact the usual minimization argument 
yields that if $ h $  in  $ R^{1,p} ( S ( \tau ) ) $ with $ h - F 
\in R_0^{1,p} (S (\eta)), $   then  
\begin{align}
\label{6.5}   
 \int_{S(\tau)}   f (\nabla \hat v  ) dx  \leq     \int_{S(\tau)}   f (\nabla  F  ) dx.  
 \end{align}
 
 Also the same argument as in Lemma \ref{lem2.8} yields
 \begin{lemma}   
\label{lem6.6} 
Let  $ p, n, k, \tau, F, f,  \hat v,     $  be as above.     Given  $ t > 0, $  let   
$ Z (t) =  \{ ( x', x'') \in \rn{k} \times \rn{n-k} : |x''| = t  \}  $ when  $ 1 \leq k \leq  n - 2, $  and  $ Z (t)  =  \{ x \in \rn{n}_+ : x_n = t \} $ when $ k = n - 1.$   There exists $  \de \in (0, 1),  c \geq 1, $  depending on $p, n, k, \hat \al , \al'  \in (0, 1)  $ and $   \xi   \in \re $  such that  
\begin{align}
\label{6.6}   
| \hat v  ( z ) -  \xi |  \leq   \liminf_{t \to 0} \left(  \max_{Z(t)}  \hat v  -   \min_{Z(t)}   \hat  v  \right)   \,  \left( \frac{\tau}{|z'''|} \right)^{\de}.   
\end{align}
Here  $ z''' = z'' $  when $ z = (z', z'')   \in  \rn{n}\sem \rn{k} $  and  $ z''' = z_n $ when $ z \in  \rn{n-1}_+. $    
 \end{lemma}
 
  \begin{proof}  From \eqref{6.6}  we observe as in Lemma \ref{lem2.8} that  $  \max_{Z(t)}  \hat v, -  \min_{Z(t)}   \hat  v,$  are nonincreasing  functions of $t$ on $ (0, \infty).$   This fact  and Harnack's inequality  for   $ \mathcal{A}$-harmonic functions   imply   Lemma \ref{lem6.6}. 
  \end{proof}  
  \subsection{Theorems \ref{thmB} and \ref{thmC}  for $ \mathcal{A} = \nabla f$-Harmonic Functions} 
Finally we state several modest propositions:   
\begin{proposition}  
\label{prop6.7} 
Fix $ p, n, k $  with either $ 1 \leq k \leq n - 2$ and $p > n - k, $ or $ k = n - 1$ and $p > 2.$   Let $ u $ be the 
$ \mathcal{A} = \nabla f$-harmonic  Martin function in  \eqref{1.4} relative to 0, where $ f $ is as in \eqref{6.3}.   
If  $ \si  <  k, $  then  Theorem \ref{thmB}, \ref{thmC}  are valid with $p$-harmonic replaced by $ \mathcal{A}$-harmonic.  
\end{proposition}  

\begin{proof}  
Using Lemmas \ref{lem6.4}-\ref{lem6.6}  we can give a  proof of  Theorem \ref{thmB}  in  Proposition  \ref{prop6.7} 
by essentially copying  the proof  of  Theorem \ref{thmB} for $p$-harmonic functions.    
To get  Theorem \ref{thmC}  we note that inequality  \eqref{6.1}  can be used in place of    \eqref{5.2},  \eqref{5.3} to obtain an analogue of  Lemma \ref{lem5.1}    for $ \mathcal{A} = \nabla f$-harmonic functions.   Using  this analogue,   Theorem \ref{thmA} for $ \mathcal{A}$-harmonic functions,   as well as  Lemmas \ref{lem6.4} - \ref{lem6.6},   one 
now obtains  an  analogue  of  Lemma  \ref{lem5.2} in the $\mathcal{A}= \nabla f$-harmonic setting.  The rest of the proof  of  Theorem \ref{thmB} follows from this analogue and  lemmas on gap series in  subsection 5.1. 
\end{proof} 
\begin{proposition}  
\label {prop6.8}  
Let $ p, n,k, f,  u, \si, $  be  as in Proposition  \ref{prop6.7}.  There exists $ \ep > 0 $ 
depending only on  $ p, n, k, \hat \al, \al' $ such that  if  
\begin{align} 
\label{6.7}  
\sum_{i,j=1}^n  \left|  \frac{\ar^2 f  }{ \ar \eta_i \ar \eta_j } -    |\eta |^{p-4} (  (p-2) |  \eta_i \eta_j  + \de_{i,j} |\eta|^2)  \right|   <  \ep  
\end{align}
 whenever  $ |\eta|= 1, $  then  $ \si < k.$ 
 \end{proposition}  
\begin{proof}  
If    Proposition  \ref{prop6.8}   is false there exists  a  sequence $ (u_j) $ of $ \mathcal{A}_j =  \nabla{f_j}$-harmonic Martin functions 
 with  $ u_j (e_n) = 1,  $    that are $-\si_j$  homogeneous on $ \rn{n} \sem \{0\}$ where   $ \si_j \geq k$ for $j =1, 2, \ldots.  $  
Also  $ f_j $   satisfies \eqref{6.3} for a fixed $ \hat \al,  \al' $    and   \eqref{6.7} with  $ \ep $ replaced by 
$ \ep_j, \ep_j \to  0 $ as $ j \to  \infty.$ 
 To get a  contradiction 
  observe  from Lemmas \ref{lem2.2}-\ref{lem2.4} for  $ \mathcal{A}$-harmonic functions,  $ u_j (e_n) = 1, $  and  $ - \si_j $ homogeneity of  
   each $ u_j $  that a subsequence of   $ (u_j)  $ say $ (u_{j_l})$ converges  uniformly on compact subsets of $ \rn{n} \sem \{0\}$   to 
   a  $ - \ti \si $  homogeneous function $ u $   on  $ \rn{n} \sem \{0\} $  with  $ \ti \si \geq k.$  Also $  u  \geq  0,$  $ u ( e_n ) = 1, $  
   and  $ u $ is  H\"{o}lder  continuous  on $ \rn{n} \sem \{0\} $  with  $ u  = 0 $ on 
   either $ \rn{k}$ or the complement of  $ \rn{n}_+. $  Choosing subsequences of the subsequence if necessary we see from 
   Lemma   \ref{lem2.5}   that  we nay also assume  
   $ \nabla u_{j_l} $ converges uniformly   on compact subsets of either $ \rn{n} \sem \rn{k} $ or $ \rn{n}_+ $ as  $ l \to  \infty$ to $ \nabla u . $ 
   Moreover   from \eqref{6.3} we observe that  each component of  $ (\nabla  f_j) $ converges locally uniformly in    $ C^{1} ( \rn{n} )  $ as $ j \to  \infty $ to a component of  
   $ \nabla ( | \eta |^p /p ) . $   An easy argument using  these facts,  then  gives that  $ u $ is a  $p$-harmonic Martin function and thereupon that  
   \[ 
   - \lan  \nabla u (e_n) , e_n \ran =    \ti \si   = -  \lim_{l\to  \infty} \lan \nabla u_{j_l}, e_n \ran =  \lim_{l\to  \infty} \si_{j_l} \geq  k, 
   \]   
    which   is a contradiction to   $ \ti \si < k $ as shown in   \eqref{4.1b}.  
    \end{proof}  

\subsection{$\mathcal{A} =  \nabla f $-subsolutions in  $\rn{n} \sem \rn{k}$   when  $f(\eta)= p^{-1} ( |\eta| + \lan a,  \eta \ran )^p$}

For  fixed $ p,n \geq  2,  a \in  \rn{n}  $ with $ |a| < 1, $ let      $     q ( \eta )   =    | \eta |  +  \lan a, \eta \ran$ for $\eta  \in \rn{n}. $  
In this  subsection  we   study the if  part of Proposition \ref{prop6.7} and  dependence on $ \ep $ in Proposition  \ref{prop6.8}  when  
$  \mathcal {A} = \nabla f $   and   $  f ( \eta )  =   p^{-1} q^p (\eta)$ for $\eta \in  \rn{n}. $       
To avoid confusion in calculations  we write $ D q,  D^2 q   $  for $ \nabla q $  and the $ n $ by $n$ matrix of  second derivatives of $q$ with respect to  $ \eta.$   
We  note that  
\begin{align*}
    q(\eta) & = |\eta| + \lan a, \eta \ran,  \\
    Dq (\eta) & = \frac{\eta}{|\eta|} + a, \\
    D^2 q (\eta) & =   \frac{1}{|\eta|}\left( I - \frac{\eta}{|\eta|}\otimes\frac{\eta}{|\eta|}\right).
\end{align*}
Thus 
\begin{equation*}
   Dq (\nabla  u) D^2  u (x)  \tilde{D} q (\nabla  u ) = \frac{\nabla  u}{|\nabla  u|} D^2  u \frac{\tilde{\nabla u}}{|\nabla u|} + 2 \frac{\nabla u}{|\nabla u|}D^2u \ti a + a D^2 u \ti a
\end{equation*}
where $ \tilde{D} q,  \ti a,   \ti {\nabla u},  $  denotes the  $  n \times 1 $  transpose of  $ Dq, a, \nabla u,  $ considered respectively as    $ n \times 1 $  row matrices.  Also  
\begin{equation*}
 q \,  \mbox{trace}\left( D^2 q (\nabla u )   D^2u \right)  =  (1+\frac{\lan a,  \nabla u \ran }{|\nabla u|} )
  \left( \Delta u - \frac{\nabla u }{|\nabla u|} D^2u \frac{\tilde{\nabla u}}{|\nabla u|} \right).
\end{equation*}
  Finally  we arrive at    
\begin{align} 
\label{eqn}  
\begin{split}
q&^{2-p}(\nabla u)  \nabla \cdot ( D f (\nabla u ) ) = 
    q^{2-p} (\nabla u) \,  \nabla \cdot ( D(\frac{1}{p} q^p)(\nabla u)) \\
    &= (p-1)  Dq (\nabla  u) D^2  u (x)  \tilde{D} q (\nabla  u )  +  q ( \nabla u) \,  \mbox{trace}\left( D^2 q (\nabla u )   D^2u \right) \\ 
    & =(p-2-\frac{\lan a, \nabla u \ran }{|\nabla u|}) \frac{\nabla u}{|\nabla u|} D^2u \frac{\tilde{\nabla u}}{|\nabla u|}   
   +(p-1)( 2 \frac{\nabla u}{|\nabla u|}\,D^2u\, \ti a + a D^2u   \ti a)
    +(1+\frac{ \lan a, \nabla u \ran }{|\nabla u| } )\Delta u.
\end{split}
\end{align}
As in section \ref{sec4}  we rewrite  \eqref{eqn}   when $ u = u ( x',  x'' ),    x=(x', x'') \in \rn{n}$,  with 
$x'\in \mathbb{R}^k$ and $x'' \in \mathbb{R}^{n-k}.$  Also $t=|x'|$, $s=|x''|$,  $r^2=t^2+s^2,$ and  $  a = (a' ,  a'' ). $      
Similarly, $\nabla = (\nabla', \nabla'')$ but we will often write $\nabla' u$ or $\nabla'' u$,  in which we regard each vector as $ n  \times 1 $ row vectors.  For example  $ \nabla' u (x )  =  ( \frac{\ar u}{\ar x'_1},  \ldots, \frac{\ar u}{ \ar x'_k}, 0, \ldots,0). $   
Likewise   $ \nabla' \times \nabla'' , $ $ \nabla'' \times \nabla',$ $ \nabla' \times \nabla'$ , $ \nabla'' \times \nabla'$,  are $ n \times n  $  matrix operators.  So    $ \nabla ' \times \nabla'' $ is the operator matrix whose 
$ i $ th row and $j$ th column is   $  \frac{\ar^2}{\ar x_i' \ar x''_j} $ when $ 1 \leq i \leq k,  k+1 \leq j  \leq n. $  All other entries are zero.  Next   the  $ k \times k $  and  $ n - k \times  n - k $  identity matrices, denoted $ I', I'' $  are regarded as $ n \times n $  matrices. So  $ I' = ( \de'_{ij} ),   I'' =  (\de''_{ij}),  $ where  $ \de'_{ii} = 1$  for $  1 \leq i \leq k $ and  $ \de''_{ii} = 1 $ for  $  k + 1 \leq i \leq n. $   Other  entries in these matrices are zero.      
Finally  $ x' \otimes x', x'' \otimes x', x'' \otimes x'',  x' \otimes x'', $  are  considered as $ n \times n $  matrices    Using this notation it follows from the chain rule as in  \eqref{gradu} - \eqref{steqn} that 
\begin{align*}
\nabla'' u & = \frac{u_s}{s} x''  \quad \mbox{and}\quad    \lan \nabla'' u, a  \ran = \frac{u_s}{s}  \lan  x'', a''' \ran,   \\
\nabla' u & = \frac{u_t}{t}x',  \mbox{ and }  \lan a, \nabla' u  \ran =  \frac{u_t}{t}  \lan a', x' \ran,  \\
\nabla' \nabla' u & = \frac{u_t}{t} I' +(\frac{u_{tt}}{t^2} - \frac{u_t}{t^3}) x' \otimes x', \\
\nabla'' \nabla' u & = \frac{u_{st}}{st} x'' \otimes x',      \\
\nabla' \nabla'' u & = \frac{u_{st}}{st} x' \otimes x'', \\
\nabla'' \nabla'' u & = \frac{u_s}{s} I'' +(\frac{u_{ss}}{s^2} - \frac{u_s}{s^3}) x'' \otimes x'' . 
\end{align*}
Then  
\[ \nabla' u \nabla'\nabla' u  \tilde{\nabla' u}  =
 \frac{u_t^2}{t^2} (\frac{u_t}{t}t^2 -\frac{u_t}{t^3}t^4 + \frac{u_{tt}}{t^2} t^4) = u_t^2 u_{tt}.
 \]
 Similarly  $ \nabla'' u \nabla''  \nabla'' u {\tilde\nabla'' u} = u_s^2 u_{ss}$  and  
 $\nabla' u \nabla'' \nabla' u \tilde{\nabla'' u} = \frac{u_t}{t} \frac{u_s}{s} \frac{u_{st}}{st} t^2 s^2= u_s u_t u_{st}$.
 This gives 
 \begin{align} 
 \label{6.19}
 \nabla u D^2 u \tilde{\nabla u} = u_t^2 u_{tt} +2 u_t u_s u_{st}+ u_s^2 u_{ss}.
 \end{align}
 Observe that 
 \begin{align} 
 \label{6.20} 
 \begin{split}
\nabla u D^2 u \ti a & = \nabla' u \nabla' \nabla' u \tilde{a'} + \nabla'' u \nabla'' \nabla' u \tilde{a'}  +      \nabla' u \nabla' \nabla''  u \tilde{a''} +  \nabla'' u \nabla'' \nabla'' u  \tilde{a''}    \\
&= (\frac{u_s u_{st}}{t}   +  \frac{u_{tt}u_t}{t} ) \lan x',  a' \ran +   (\frac{u_t u_{st}}{s}   +  \frac{u_{ss}u_s}{s} ) \lan x'' ,   a'' \ran.
\end{split}
\end{align} 
The next term is 
\begin{align}
\label{6.21}      
\begin{split}
a D^2 u  \ti a  &=   a'  \nabla' \nabla' u \tilde{a'} +  2  a''  \nabla'' \nabla' u \tilde{a'}  +      a'' \nabla''  u \tilde{a''}  \\
&=   \frac{u_t}{t} |a'|^2 + (\frac{u_{tt}}{t^2} - \frac{u_t}{t^3}) (x'\cdot a')^2   \\
& \quad +  2 \lan x',  a' \ran \lan x'', a'' \ran   \frac{u_{st}}{st}      +       \frac{u_s}{s} |a''|^2 + (\frac{u_{ss}}{s^2} - \frac{u_s}{s^3})  \lan x'' ,  a'' \ran^2.                     
\end{split}
\end{align}
Finally we calculate $\Delta u = \text{trace} (D^2u) = \text{trace}(\nabla'\nabla' u) + \text{trace}(\nabla'' \nabla'' u)$: 
\begin{equation} \label{6.22}
\Delta u = \frac{u_t}{t} (k-1) + u_{tt} + \frac{u_s}{s} (n-k-1 ) + u_{ss}. 
\end{equation}
 Substituting into \eqref{eqn} we get (where all derivatives on the right hand side are with respect to $(s,t)$)
\begin{align}
\label{steqn1}  
\begin{split}
&q^{2-p} ( \nabla u ) \,  \nabla \cdot ( D(\frac{1}{p} q^p)(\nabla u)) \\ 
      & =\left( {\ds   p-2-    \frac{\lan x',  a'  \ran  }{  |\nabla u| t} u_t  -  \frac{\lan x'', a''  \ran }{ | \nabla u | s } u_s } \right)  {\ds  \frac{(u_t^2 u_{tt} +2 u_t u_s u_{st}+ u_s^2 u_{ss})}{ |\nabla u |^2 } }  \\ 
&  + {\ds 2 (p-1)  \left[  \left(\frac{u_s u_{st}}{|\nabla u| t}   +  \frac{u_{tt}u_t}{|\nabla u| t} \right) \lan x',  a' \ran +   \left(\frac{u_t u_{st}}{|\nabla u| s}   +  \frac{u_{ss}u_s}{|\nabla u | s} \right) \lan x'' ,   a'' \ran \right] }  \\ 
& + (p-1) \left[ {\ds   \frac{u_t}{t} |a'|^2 + 
  (\frac{u_{tt}}{t^2} - \frac{u_t}{t^3})  \lan x' ,  a' \ran^2    +    { 2 \lan x',  a' \ran \lan x'', a'' \ran   \frac{u_{st}}{st}      +       \frac{u_s}{s} |a''|^2 + (\frac{u_{ss}}{s^2} - \frac{u_s}{s^3})  \lan x'' ,  a'' \ran^2 }    } \right]                        
     \\  
     &  + {\ds  \left(  1   +   \frac{ \lan x', a' \ran u_t }{ |\nabla u| t } + \frac{ \lan x'', a''  \ran u_s }{  |\nabla u| s} \right) \left( \frac{u_t}{t} (k-1) + u_{tt} + \frac{u_s}{s} (n-k-1 ) + u_{ss} \right) }.
\end{split}
\end{align}

We use  subsolutions of \eqref{steqn}  when  $  1 \leq k  \leq n - 2$ and $p >  n - k  $ to study  \eqref{steqn1}. 
  For this  purpose  recall the notation in section 3  and let $  u  =   s^{\ti \be} / r^{( \ti \la + \ti \be) } $  where  
\[  
\ti \be = (1+\de) \be, \quad \ti \la = ( 1 + \de ) \la, \quad  \mbox{and} \quad  \be = \frac{p-n + k}{p-1} 
\]    
with  $  \de \geq 0 $ and    $ \la   \geq  \chi,   (\chi  $   as in Theorem  \ref{thmA}).  From    \eqref{3.2},    \eqref{steqn},   \eqref{3.7},   \eqref{abceqn},        
with  $  \be,  \la,   $  replaced  by  $ \ti \be,  \ti \la,   $   we deduce that  
\begin{align}
  \label{6.24} 
  | \nabla  u ) |^{2-p}  \,  \nabla \cdot ( |\nabla  u |^{p-2} \nabla  u ) =     \frac{ u  \, ( \ti A  \ti \lambda^2 s^4 +\ti  B \ti \beta s^2 t^2 + \ti C \ti \beta^3 t^4) }{s^2 r^2 ( \ti \la^2  s^2 +  \ti \be^2 t^2)} 
\end{align}  
where $ \ti A, \ti B, \ti C, $ are defined as in   \eqref{3.13} with $ \la, \be $ replaced by  $ \ti \la, \ti \be. $    
Using  \eqref{6.24} we rewrite  \eqref{steqn1} with  $  u  =  s^{\ti \be}/ r^{\ti \la  + \ti \be} $   as 
\begin{align}
\label{6.25} q^{2-p} ( \nabla  u ) \,  \nabla \cdot ( D({\ts \frac{1}{p}} q^p)(\nabla  u))   =    
{\ds  \frac{ u ( \ti A \ti \lambda^2 s^4 + \ti B \ti \beta s^2 t^2 + \ti C \ti \beta^3 t^4)  }{s^2 r^2 (  \ti \la^2 s^2 + \ti \be^2 t^2 ) }}  +  E_1 +  E_2 + E_3 + E_4 
\end{align}
where
\begin{align}
 \label{6.26}       
 E_1 =  -  \left( \frac{\lan x', a'  \ran  }{  |\nabla  u| t} \,
 u_t  +  \frac{\lan x'',  a'' \ran }{ | \nabla  u | s }\,   u_s  \right)  {\ds  \frac{( u_t^2  u_{tt} +2  u_t  u_s u_{st}+  u_s^2 u_{ss})}{ |\nabla  u |^2 }}, 
 \end{align}
\begin{align}
 \label{6.27}  
 E_2  =   {\ds 2 (p-1)  \left[  \left(\frac{ u_s  u_{st}}{|\nabla  u| t}   +  \frac{ u_{tt}  u_t}{|\nabla  u| t} \right) \lan x',  a' \ran +   \left(\frac{ u_t  u_{st}}{|\nabla  u| s}   +  \frac{ u_{ss}  u_s}{|\nabla  u | s} \right) \lan x'' ,   a'' \ran \right] }, 
\end{align} 
\begin{align}
 \label{6.28} 
 E_3 = (p-1) \left[ {\ds   \frac{ u_t}{t} |a'|^2 + 
  (\frac{ u_{tt}}{t^2} - \frac{ u_t}{t^3})  \lan x' ,  a' \ran^2    +    { 2 \lan x',  a' \ran \lan x'', a'' \ran   \frac{ u_{st}}{st}      +       \frac{ u_s}{s} |a''|^2 + (\frac{ u_{ss}}{s^2} - \frac{ u_s}{s^3})  \lan x'' ,  a'' \ran^2 }    } \right],
  \end{align}
 and
\begin{align} 
\label{6.29}  
E_4 =   {\ds  \left(   \frac{ \lan x', a' \ran  u_t }{ |\nabla  u| t} + \frac{ \lan x'', a'' \ran  u_s }{  |\nabla  u| s} \right) \left( \frac{ u_t}{t} (k-1) +  u_{tt} + \frac{ u_s}{s} (n-k-1 ) +  u_{ss} \right)   } 
\end{align}
 whenever  $ t = |x|'$ and $s =   |x''|. $  We use \eqref{6.24}-\eqref{6.29}  to  estimate $ |a|$ in terms of  $ \de $ 
 so that $  u $ is  an $  \mathcal{A}  = \nabla f $-subsolution in $ \rn{n} \sem \rn{k}$. 
 From homogeneity of  $f$ it suffices to make this estimate when $ |x|^2 = s^2 + t^2 = 1. $ 
    We  first  show  that  $ \de > 0 $  is  a necessary assumption in   order  for $  u $  to be a $ \mathcal{A} = \nabla f $-subsolution or  supersolution when $ a \neq 0.$   
   Indeed   if  $ \de = 0,  $     then  $ \ti C = 0 $    and  using  $ 0 <  \be  < 1,   x''  = s \,  \om'' ,   s > 0, $  we  find that for fixed $ p, n, k,  $  as $ s \to  0^+,   $ 
\begin{align}  
\label{6.30}  
q^{2-p} ( \nabla  u ) \,  \nabla \cdot ( D({\ts \frac{1}{p}} q^p)(\nabla  u))   =   o ( s^{\be-2} )  +     E_1 +  E_2 + E_3 + E_4. 
\end{align}
    The  $ E $  terms are   also  $ o ( s^{\be-2}) $  as $  s \to  0, $  except for those containing  $ u_{ss}$ or  $u_s/s. $  
    Using this observation and  $ u_s/|\nabla u| \to  1 $ as $ s \to  0, $  we can continue the estimate in  \eqref{6.30}  to  obtain 
\begin{align}  
\label{6.31} 
\begin{split}
   s^{2-\be} &(  E_1 + E_2 + E_3 + E_4 ) \\  
   &=o(1) +    \be (p-1)  |a'' |^2 +  [   2 (p-1)  \be ( \be - 1)      +  \be (n - k - 1) ]   \lan \om'', a'' \ran  \\ 
   &\quad   +  (p-1) ( \be ( \be - 1) - \be ) \lan  \om'' ,  a'' \ran^2     \\ 
   &=  o(1) +   \be   (  (p-1) |a''|^2   + (k + 1 - n )  \lan \om'', a'' \ran  +  (p-1) ( \be - 2 )   \lan \om'', a'' \ran^2). 
\end{split}
\end{align}   
We first choose $ \om'' $ so that $ \langle \om'', a'' \ran = |a''|. $   Then for this value of  $ x'' $ we see for  
     $ s $ small enough from   \eqref{6.30}, \eqref{6.31}   that  $ \nabla \cdot ( D(\frac{1}{p} q^p)(\nabla  u))  < 0 $  
     (since $ \be < 1 $ and $ k + 1 - n  \leq 0) . $   On the other hand choosing $ \om'' $ so that  $ \lan \om'', a'' \ran = 0 $  
     we have     $ \nabla \cdot ( D(\frac{1}{p} q^p)(\nabla  u))  > 0 $  for this value of   $ x'' $  and $ s $  small enough.  
     Thus $  u $   can  never  be  a $ \mathcal{A} = \nabla f$-subsolution or supersolution when $ \de = 0 $ and $ a'' \neq 0. $  
 
 If $ \de =  0$ and  $a'' = 0,  $   the  $ E $  terms are  $ o ( s^{ \be - 1} ) $  as $  s \to  0, $ except for  terms in $ E_1, E_4 $   
 containing    $  u_{ss},  u_s/s,   $  and a term in    $ E_2 $  containing  $  u_{st}. $   
Using  \eqref{6.24}-\eqref{6.29}  it follows  that   
\begin{align}
 \label{6.32}  
\begin{split} 
 s^{1-\be} &(  E_1 + E_2 + E_3 + E_4 )  \\
 & =  o(1)      +  \lan \om', a'  \ran [  ( \la + \be )  ( \be - 1)   -  2 (p-1)  ( \la   + \be)  \be 
  -   ( \la  +  \be) ( n - k + \be - 2 ) ]  \\
  &=   o( 1)   -    \lan \om', a' \ran    ( \la + \be) [  2 (p  - 1) \be   + n  -   k - 1   ].
\end{split}
\end{align}
   The last term in brackets of  \eqref{6.32}  is always positive so choosing $ \om' $ with  $ \lan \om', a' \ran = \pm   \, | a' | $  and $ s > 0 $ small enough  we conclude that 
$  u $ cannot be either an  $  \mathcal{A} = \nabla f$-subsolution or supersolution when $ \de = 0$ and $a \neq 0. $   

  Now suppose that   $ \de > 0 ,  \la \geq    \chi (p, n, k), $ and recall that $ r^2 = s^2 + t^2 = 1$,  $ s = |x''|,  t = |x'|.$     
  Then   from   \eqref{3.13}  we note that if $  \chi =  \chi (p, n, k )$ and $ \be = \be (p, n, k ) $  are as in the proof of  Theorem \ref{thmA} then $ \ti B \geq 0 $  and   
  \begin{align} 
  \label{6.33} 
  \begin{split}
\ti A & =  (p-1)( \ti \la + \ti \be ) (  \ti \la   -  {\ts \frac{k}{p-1}} )  \geq (p-1) ( \ti \la + \ti \be)  \de \chi  , \\  
\ti B &  =	(2 \ti \beta(p-1)+n-k-2)  (  \ti \la  +  \ti \be ) ( \ti \la  -   { \ts \frac {\ti \beta ( p- 2 + k ) }{2p- n +k - 2 }} ), \\ 
\ti C &   = (p-1) ( \ti \beta - \be )  =  (p - 1) \de \be .
 \end{split}
\end{align} 
To get a ballpark estimate on $ |a| $ from above we  use \eqref{6.33}  and either $ s^2 $ or $ t^2  \geq 1/2 $  to  first get  
\begin{align}   
\label{6.34}   
\frac{ u \,  ( \ti A \ti \lambda^2 s^4 + \ti B \ti \beta s^2 t^2 + \ti C \ti \beta^3 t^4)      }{s^2 r^2 (  \ti \la^2 s^2 + \ti \be^2 t^2 ) }  \geq  (p-1)     s^{ \ti \be - 2}\, \, 
 \frac{ (1/4) \min \{ \ti \la^2   ( \ti \la + \ti \be )   \de \chi ), \ti \be^3   \de \be   \} } { \ti \la^2   +
  \ti \be^2   } .
  \end{align}
 
Next from   \eqref{3.5},  \eqref{3.7},  \eqref{p1},  \eqref{6.26},   we get   for $ x  \in  \ar B (0, 1 ) \sem  \rn{k}$       
\begin{align}
\label{6.35}
\begin{split}
 | E_1 |  +  |E_2 |      &\leq    2 (  2 p - 1) |a|  ( |u_{tt} |   +  2 | u_{st} |  +   | u_{ss} | )  \\  
 &   \leq       20  ( p - 1)  \,   |a| \,  s^{\ti \be - 2}  ( \ti \la + \ti \be  + 1)^2.
 \end{split}
\end{align} 
  Similarly   
\begin{align}
\label{6.36}   
|E_3|  & \leq 10  ( p - 1 )  |a| s^{ \ti \be - 2 }   ( \ti \la + \ti \be + 1 )^2,
\end{align}
and
\begin{align} 
\label {6.37} 
| E_4|   \leq 10  |a|   s^{\be - 2}  ( \ti \la + \ti \be + 1 )^2  ( n   +  k ).
\end{align}
From  \eqref{6.34}-\eqref{6.37} we conclude that if   
\begin{align}   
\label{6.38}  | a |  <    \, \, 
 \frac{ (1/4) (p-1) \min \{ \ti \la^2   ( \ti \la + \ti \be )   \de \chi ), \ti \be^3   \de \be   \} } { 
   (\ti \la^2   +
  \ti \be^2 ) [ (30  (p-1)  +  10 (n  + k) )  ( \ti \la + \ti \be + 1 )^2  ] }  
 \end{align}    
then $ u $ is an $ \mathcal{A}  =  \nabla f$-subsolution on $ \rn{n} \sem \rn{k}$.   

 \begin{lemma} 
 \label{lem6.11}  
 Let     $  1 \leq k \leq n - 2,  p > n - k,  $ and  $  f ( \eta ) =p^{-1}  (| \eta | +  \lan \eta , a \ran)^p$ for $\eta  \in \rn{n}. $  
 Let $ v $ be the  $ \mathcal{A} = \nabla f$- harmonic Martin function  for 
  $  \rn{n} \sem \rn{k} $  with $ v (e_n ) = 1. $  If  $ v $  has homogeneity  $ -  \si,  $  and    $ u $  as in  \eqref{6.25} is  an  
  $ \mathcal{A}$-subsolution with $ \la = \chi  (p,n,k) , $   then   $   \si \leq   (1 +  \de ) \,  \chi ( p, n, k  ). $    
\end{lemma} 
  \begin{proof}    
  Let  $  \ti q ( \eta''  ) =   | \eta'' |   +  \lan a,  \eta'' \ran$ for $\eta'' \in  \rn{n-k}, $   and put
  \[   
  \ti h ( x'' )   =   \sup \{ \lan x''  , \eta'' \ran : \ti q ( \eta'' ) \leq 1  \},  x'' \in \rn{n-k}. 
  \]      
   $\ti  h $ is homogeneous  1  on  $ \rn{n - k }$   and in  \cite{LN}    (see also 
       \cite{AGHLV,ALSV}  it is shown   that  if 
\[ 
\ti f ( \eta'') = p^{-1} \ti q (\eta'')^p\quad \mbox{and}\quad \be =  \frac{p + k - n}{p-1}
\]  
then       
\begin{align}  
\label{6.39}  
\ti h ( x'' ) \approx |x'' | \, \,  \mbox{on}\, \, \rn{n-k} \quad \mbox{and}\quad    \ti h^{\be}\, \, \mbox{is}\, \,  \mathcal {\ti A} = \nabla \ti f\mbox{-harmonic on}\, \, \rn{n-k} \sem \{0\}. 
\end{align}                  
                  
      Constants in  \eqref{6.39}  depend only on   $ p, n - k, \al' ,  \hat \al. $ Extend  $ \ti h $ to $ \rn{n} $  by setting $ h ( x  ) = \ti h ( x'' ),  x = \lan x', x'' \ran  \in \rn{n}, $ 
      and  observe that $ h^{\be} $ is  continuous on $ \rn{n}$  as well as   $ \mathcal{A} = \nabla f$-harmonic on $ \rn{n} \sem \rn{k} $   with $ h^{\be} \equiv 0 $  on $ \rn{k}.$ 
      Also from Lemma \ref{lem2.6} for   
      $ \mathcal{A} = \nabla f$-harmonic functions  we deduce that  
      $ v/h^{\be}   \approx 1$   on $ \ar B (0, 1 ) \sem \rn{k} $ with constants having the same dependence as  
      those in \eqref{6.39}.   Using  this   deduction,  \eqref{6.39},  $ \de \geq 0,  $ and  the definition  of 
      $  u, $  we  find   that
     \begin{align} 
\label{6.40}    
u  \leq  c \, v  \quad \mbox{on}\,\,  \ar B (0, 1 ) \sem \rn{k}   \quad \mbox{where }  \, \,      c = c ( \la, \de, p, n, k, \al' ,  \hat \al ). 
\end{align}
     From \eqref{6.40},   $  u + v \to  0 $ uniformly as $ x  \to  \infty, $   and the boundary maximum principle in  Lemma \ref{lem2.2} for 
      $ \mathcal{A} = \nabla f$-harmonic functions we conclude that   $  \si \leq \ti \la. $ Taking 
      $ \la =  \chi $ we obtain  Lemma \ref{lem6.11}.   
      \end{proof}    
      
      \begin{remark} 
      \label{rmk6.12}  
      From   Lemma \ref{lem6.11},  Proposition  \ref{prop6.7},  and  assumption  \eqref{6.38},  we  conclude Theorems \ref{thmB} and \ref{thmC} in the 
        $ \mathcal{A} = \nabla f $  setting  when $ f ( \eta )   =  p^{-1} (  |\eta | + \lan \eta, a \ran )^p $  and $ ( 1 + \de ) \chi (p, n, k  )  < k.$   
       \end{remark}    
     Since \eqref{6.38} is   a rather awkward assumption for applications, we prove:  
      \begin{lemma} 
      \label{lem6.13}  
  Assume  $ p  >  n -  1, n \geq 3 ,   k = 1,   $   and choose  $ \de $  so  that    $ ( 1 + \de ) \chi (p, n, 1)   =   1  - \frac{ (p-2)}{4 (p-1)} . $  If  $ p >  n  -  1, $ and   
\begin{align} 
\label{6.40a}     
|a|   \leq     \frac{ (p-2)  \min \{    ( p + 1 - n )^4   ,  1  \} }{ 100000 (p-1)} 
\end{align}
        then   \eqref{6.38} implies that $  u $ is an  $ \mathcal{A} $  = $ \nabla f $-subsolution  in $ \rn{n} \sem \re. $  
        \end{lemma}  
        
 \begin{proof}  
We note  that    if  $ p  \geq n \geq 3, k = 1, $ then      
\begin{align}  
\label{6.41} 
1  \geq    \be =     \frac{p + 1 - n}{p - 1 }   \geq     \chi = \be   \frac{p -  1  }{ 2  p - n  - 1  }   \geq    \be/2   
\end{align}
and  $  \chi $  is increasing on  $ (n, \infty) $ as a function of $p.$ Thus $ \chi \leq 1/2 $  so  
\begin{align} 
\label{6.42}    
\de \be  \geq  \chi \de  \geq  1/4 \quad \mbox{and}\quad  ( 1 + \de ) \be  \leq 2. 
\end{align}
Using    \eqref{6.42} and   \eqref{6.41}  in   \eqref{6.38} and another ballpark estimate    we get  \eqref{6.40a} and thereupon Lemma   \ref{lem6.13} 
     when $ p \geq n. $ If     $ n - 1 < p < n, $ we note that   $ \chi =  (p- 1)^{-1},  $   so   
\begin{align}   
\label{6.42a}    
\de  =     (3/4) (p-2)  \quad \mbox{and} \quad  ( 1  + \de ) \be =  ( 1 +  \de )  \chi   (p + 1 - n). 
\end{align}
Using this information in  \eqref{6.38}  we get \eqref{6.40a} and Lemma \ref{lem6.13} for $ n - 1  < p < n. $          
\end{proof} 
 \begin{corollary} 
 \label{cor6.14}  
 If  $ a $ satisfies    \eqref{6.40a},   then  Theorems \ref{thmB} and \ref{thmC} are valid for $  1 \leq k \leq n - 2$ and $p   >  n - k , $ in   $  \rn{n} \sem \rn{k}  $  and the  $ \mathcal{ A } = \nabla f$-harmonic setting when $  f ( \eta )   = p^{-1} ( | \eta | +  \lan a, \eta ))^p$ for  $\eta \in \rn{n}.  $    
\end{corollary}  
     \begin{proof}     
     We first observe that Corollary   \ref{cor6.14}  is implied by   Lemma \ref{lem6.13},  Lemma \ref{lem6.11}, and Proposition \ref{prop6.7} when $ k = 1$.  As in the  $p$-harmonic setting,  Theorems  \ref{thmB},  \ref{thmC},  for    other values of $k$ follow from the $ k = 1 $ case by adding dummy variables (see Remark \ref{rmk1.8}).     
 \end{proof}

 \subsection{$ \mathcal{A} =  \nabla f$-subsolutions in  $\rn{n}_+ $   when  $f(\eta)= p^{-1} ( |\eta| + \lan a,  \eta \ran )^p$}
  In this  subsection we  continue  our investigation of  the if part of  
  Proposition \ref{prop6.7} and  $ \ep $ in  Proposition \ref{prop6.8} in   $ \rn{n}_+ $ when   $ k = n - 1$ and $p > 2.$   
  We begin with   
\begin{lemma}  
\label{lem6.9}  
Let  $ u $  be  an  $\mathcal{A } = \nabla f$-harmonic Martin function  in  $ \rn{n}_+ $   when  $  p \geq  2 $ is fixed and $ f $ is as in \eqref{6.3}.      If $ p_1 > p, $ then  
  $ u $  is   an $ \mathcal{A}^{(1)} =  \nabla ( f^{p_1/p})$-subsolution in $ \rn{n}_+$.      
\end{lemma}  
\begin{proof}  Given $ m > n + 2 $  let     $ V_m \geq 0  $  be   the  $ \mathcal{A} = \nabla f$-harmonic  function  in  
\[ 
\Om_m  = [ B (0, m)  \sem \bar B( {\ts  \frac{2 e_n}{m}, \frac{1}{m}} )  ]  \cap  \rn{n}_+ 
\]  
with $ V_m (e_n) = 1, $ and  continuous boundary values  $ V_m  \equiv b_m$ on  $  \ar B (\frac{2 e_m}{m}, \frac{1}{m}) $  
while  $ V_m \equiv 0 $  on      $ \mathbb{R}^n\setminus (B (0, m) \cap \rn{n}_+)$.  Here $b_m>0$ is a constant with $b_m\to \infty$ as $m\to\infty$.  Once again using Lemmas \ref{lem2.3}-\ref{lem2.5} we find  that    
$ (V_m ) $  is  uniformly  bounded and locally H\"{o}lder continuous  in   $ B (0, m)  \sem B(0, 1 ) $  with H\"{o}lder  exponent    and bounds   that are independent of $ m. $      Also there exists 
   $ \tau \in (0, 1) $  with   
\begin{align} 
\label{6.8}    
\max_{B(0,s) \cap \Om_m} V_m   \leq   c(p,n, \al', \hat \al)  (1/s)^{\tau}   \quad  \mbox{whenever}\, \,     m >  s  \geq   2.  
\end{align}
  To briefly outline the proof of  \eqref{6.8}, it follows from Harnack's inequality and the above lemmas applied to  
    ${\ds  \max_{\ar B(0,1)} V_m} -  V_m $  that   for some $ \he \in (0,1) $ (independent of $ m $), 
\begin{align}  
\label{6.9}    
\max_{\ar B(0,2)} V_m  \leq   \he \max_{\ar B(0,1)} V_m. 
\end{align}
  Iterating   \eqref{6.9} we get \eqref{6.8}. 
 
 Next  using \eqref{6.3}  and arguing as  in  Lemma 4.4   of  \cite{AGHLV}   when $  2 \leq  p  < n $  and as in  \cite{ALV}   when $ p \geq n ,$   it follows that  
\begin{align} 
\label{6.10}  
\mbox{for each $ t \in (0, b_m)$, the set $\{ x : V_m ( x ) > t \} $  is  a convex open set. }
\end{align}
    Using \eqref{6.8}, \eqref{6.10}, the above lemmas  and Ascoli's theorem  it follows that  a subsequence of $ (V_m)$  converges 
    uniformly to $ u $   on  compact subsets of $ \rn{n}_+  $  and   \eqref{6.10} is valid with $ V_m $ replaced by $u$ 
     whenever $ t \in (0, \infty). $    We   deduce first  from   homogeneity of $u $   that  
    $ \nabla u  \not = 0 $  in  $ \rn{n}_+ $  and thereupon from  \eqref{6.3},  Lemma \ref{2.5},  and a 
    Schauder type argument that  $ f \in C^2 (\rn{n}) $ and  that  $ u $  has  locally   H\"{o}lder   continuous second partial  derivatives  in $ \rn{n}_+ $  
    with exponent depending only on  $ p, n, \al', \hat \al.$     Let  $ q = p^{-1/p}  f^{1/p}.$  
    Given $ t, 0 < t < 1, $ let $T$ denote the tangent plane to   $ y  \in \{x : u(x) = t \} . $   
    Since $ u $ has continuous second  partials  and $ \{x: u (x) > t \}$  is convex    we note 
    from the maximum principle for $ \mathcal{A}$-harmonic functions  that  $ u|_{T \cap \rn{n}_+}$  has a relative maximum  at  $ y.$    
    From the second derivative test for maxima we conclude that  if  $ z \in T, z \not =  y, $   and      $ \xi  =  (z- y)/|z-  y|, $  then  $ u_{\xi} ( y )  = 0,  u_{\xi \xi}(y) \leq 0.  $ 
      Next we choose an orthonormal  basis,   $ \{  \xi^{(1)}, \xi^{(2)}, \ldots, \xi^{(n)}\} $  for $ \rn{n}$    
        so that  
      $ \xi^{1} = \nabla u ( y)/|\nabla u ( y) | $ and  with  $ \xi^{(i)}, 2 \leq i \leq n,  $ joining $  y $ to points in 
      $ T.$    Thus  
\begin{align}
 \label{6.11}  
 u_{\xi^{(j)} \xi^{(j)}} (  y ) \leq 0  \quad \mbox{for}\, \,  2 \leq j \leq n. 
\end{align} 
      Now  each component of  $ \nabla  q (\eta ) = ( q_{\eta_1}, \ldots, q_{\eta_n} )  $ is homogeneous of degree  0 on $ \rn{n} $ so    for $ 2 \leq i \leq n, $     
\begin{align} 
\label{6.12}  
\sum_{j=1}^n   \xi^{(1)}_j q_{\eta_i \eta_j} ( \xi^{(1)} )  = 0. 
\end{align}
Also from \eqref{6.3} $(a), (b), $  and  
\begin{align}
\label{6.13} 
\begin{split} 
p^{-1} \sum_{i,j=1}^n     & f_{\eta_i \eta_j} ( \nabla u (y))  \\
&=q^{p-2} ( \nabla u )   \sum_{i,j=1}^n  [ (p-1)  ( q_{\eta_i} (\nabla u (y))
       q_{\eta_j} (\nabla u(y)) + q ( \nabla u (y))  q_{\eta_i \eta_j} (\nabla u(y)) ]  
\end{split}
\end{align}       
we deduce first that if  $ (w_1, w_2, \ldots, w_n) $ 
 is orthogonal to   $ \nabla q ( \nabla u (y)), $  then    
\begin{align} 
\label{6.14} 
c(p,n,\al',\hat \al)  \sum_{i, j = 1}^n  q_{\eta_i \eta_j} (\nabla u(y)) w_i w_j  \geq |w|^2/|\nabla u (y)|. 
\end{align}
The subspace, say  $ \Gamma, $ generated by all such $ w $ has dimension $ n - 1.$  Also $ \nabla u (y) $ is  not in this subspace 
since   $ \lan \nabla q (\nabla u(y)), \nabla u(y) \ran = q ( \nabla u(y)). $   We conclude  from  \eqref{6.14} and  \eqref{6.12} that  the 
 $ n \times n $ matrix 
 $ ( q_{\eta_i \eta_j} ( \nabla u (y)) ) $  is positive semi definite  and 0 is an eigenvalue of this matrix 
 while $ \nabla u (y) $ is an eigenvector corresponding to 0.  Next we note from \eqref{6.13} and  $ \mathcal{A} =  \nabla f$-harmonicity of  $ u $      
 that 
 \begin{align} 
  \label{6.15}  
\begin{split}  
  0   =     \sum_{i,j=1}^n  [ (p-1)  ( q_{\eta_i} (\nabla u (y))
       q_{\eta_j} (\nabla u(y)) + q ( \nabla u (y)) 
               q_{\eta_i \eta_j} (\nabla u(y)) ]  u_{x_i x_j} (y)
 \end{split}
 \end{align}              
                 whenever $y \in \rn{n}_+$ so to prove Lemma    \ref{lem6.9}  it  suffices to show that   
\begin{align} 
\label{6.16} 
\mbox{ trace}\left( ( q_{\eta_i  \eta_j} (\nabla u (y)  ) ( u_{x_i x_j} (y)) \right) =   \sum_{i,j=1}^n  q_{\eta_i \eta_j} (\nabla u(y)) \,   u_{x_i x_j} (y) \leq  0  
\end{align}
whenever $y \in \rn{n}_+$.   Since  the  trace of the product of two symmetric  matrices is unchanged under an orthogonal  transformation 
we may assume that  $ \nabla u(y)  = | \nabla u (y) |  (1, 0, \ldots, 0).  $   Then  from  \eqref{6.11}  we see that  
$ ( u_{x_i x_j} ), 2 \leq i, j \leq n, $  is a negative semi-definite matrix and from \eqref{6.12}, \eqref{6.14}  that   
$ (q_{\eta_i \eta_j}(\nabla u (y)),  2  \leq i, j   \leq n,  $  is  positive semi definite.  
Using this fact  and  the observation that the trace of the product of  two positive semi definite matrices is positive semidefinite,  
we  get  (after possibly another rotation) that 
 \eqref{6.16}  and thereupon Lemma \ref{lem6.9}  is true.   
 \end{proof} 
 

 \begin{remark} 
 \label{rmk6.14}  
 Lemma \ref{lem6.9} implies  that if  $ - \la (p'), p' \geq p, $ denotes the  Martin exponent for a   $ \mathcal{A}' =  \nabla ( f^{p'/p})$-harmonic function,  
 then $ \la (p') $ is  a  nonincreasing function       
 in  $ [p, \infty). $   Indeed  from the boundary Harnack inequality in    Lemma \ref{lem2.6},  
 it is easily seen that if  $ u' $  denotes the  
 $ \mathcal{A}' $ Martin function with $ u' (e_n) = 1, $  and  $ x  \in \ar B (0, 1 ) \cap \rn{n}_+, $ 
 then  
 \begin{align}  
 \label{6.17}  
 u (x) / x_n \approx  u' (x )/x_n
 \end{align}
 where the proportionality constants depend on $ p,p',n, \al', \hat \al$.  Using this fact, homogeneity of  $ u, u' ,$  and Lemma \ref{lem2.2}, we get 
 $ \la (p') \leq \la (p). $    We  do not know if  a  similar inequality  holds  when $ k $ is fixed,   
  $  1 \leq   k \leq n - 2$ and $ p  >  n - k,  $ in $ \rn{n} \sem 
 \rn{k},  $ even for $p$-harmonic Martin functions,   although it  is  clear  from drawing levels that the above proof fails. 
 \end{remark}

Next we consider subsolutions of  \eqref{steqn1} in $ \rn{n}_+$ when $p  > 2.$  We  begin by  mimicking  the argument when $ p > n - k$ and $1 \leq k  \leq n - 2, $  with  $ \be = 1. $    
       Let    
\begin{align}       
       \label{6.43}   
       u ( x', x'' )  =  u  ( t, s )  =   s^{1+ \de}   r^{ - (1 + \de )  ( \la + 1)  }  =  x_n^{1+\de}  |x|^{- (1+ \de) ( \la + 1 ) } 
\end{align}      
where
\[      
x' = ( x_1', \ldots, x_{n-1}' ), \quad    |x'|  = t,   \quad x'' = x_n  = s \geq  0,  \quad r = \sqrt{ |x'|^2 + x_n^2 } =  ( t^2 + s^2 )^{1/2}. 
\]    
      Also $ \de \geq 0 $  and  $ n  \geq ( 1 + \de )  \la \geq  \chi ( p, n, n - 1)   $  ($ \chi $ as in \eqref{1.6} of Theorem \ref{thmA} with $ k = n - 1). $    
       With this understanding  one  can start by writing down the new version of   \eqref{steqn1}  and  then continue  the argument to  get  
       \eqref{6.24} -  \eqref{6.29}   with $ \ti \la = ( 1 + \de ) \la$ and $\ti \be = 1 + \de. $ 
       Also  $ r = 1 $  and  $ \ti A, \ti B, \ti C $ are defined in the same way as  $ A, B, C $ are defined  in  \eqref{3.13}  
       (see  also  \eqref{6.33})   only    with   $ \ti \la, \ti \be, k $ replaced by  $ ( 1 + \de ) \la,  1 + \de, n - 1,$  respectively.   
       Next we  investigate as in subsection 6.3  whether $ u $  can be  an   $ \mathcal{A} =  \nabla f $-subsolution  when 
       $  \de  = 0 $   and  $  f ( \eta )  =  p^{-1} ( \eta +   \lan a, \eta \ran )^p$ for $\eta \in \rn{n}. $ 
   Indeed, if $ \de = 0, $ then  from the new version of     \eqref{6.30} we have  for fixed $ p, n, $  and uniformly for  $ s \in (0, 1]$ that  
\begin{align}  
\label{6.44}  
q^{2-p} ( \nabla  u ) \,  \nabla \cdot ( D ( {\ts \frac{1}{p}} q^p)(\nabla  u))   =   O (s)   +     E_1 +  E_2 + E_3 + E_4   .
\end{align}
  To estimate the $ E's$  we  observe   for $s \in  (0,1] $ that 
     \[   
     |u_t|  +  |u_{tt} |  +  |u_{ss} |  =  O(s)  
     \]   
     while  
     \[  
     u_s  =  |\nabla u|  + O(s) \quad \mbox{and} \quad    u_{st}  =  - ( 1 + \la )  + O (s). 
     \]  
     Using these equalities and   $  x''  = s e_n$ and $x'  =  t \om' ,    $  we  find  for $ x \in \ar B (0, 1) \cap \rn{n}_+,   $   and $s \in (0, 1]$ that 
\begin{align}
  \label{6.47}
 \begin{split}  
   |E_1|  + |E_4|  &=  O(s), \\
 E_2  &=  -  2(p-1) ( 1 + \la)  \lan \om',  a' \ran  + O(s), \\
E_3   &= - 2 (p-1)  ( 1 +  \la)   ( \om', a'   \ran  \lan e_n, a'' \ran     +  O(s). 
   \end{split}
\end{align} 
  Combining  \eqref{6.44} - \eqref{6.47}  and letting  $ s \to  0 $ we arrive at    
\begin{align} 
\label{6.48} 
\lim_{ s \to  0}     q^{2-p} ( \nabla  u ) \,  \nabla \cdot ( D({\ts\frac{1}{p}} q^p)(\nabla  u))  
      =   -  2(p-1) ( 1 + \la) \lan \om', a'  \ran     ( 1 +\lan a'', e_n \ran ) . 
\end{align}      
   Since  $ |a''| < 1 $ and we can choose $ \om' $ so that $  \lan \om' , a' \ran = \pm |a'|, $ 
   we conclude that  $u$ can be neither an   $  \mathcal{A} = \nabla f$-subsolution or supersolution when $ a' \neq 0. $  If   $ a '  = 0 $ 
   and $ a'' \neq 0 $ we  need to make more  detailed calculations. 
   For this purpose  we temporarily   allow $ p = 2 $ in  our calculations  and  
   note from  \eqref{6.33}, \eqref{6.25}  that if $  |x'| = t, s = x'' = x_n > 0, $ then  at
    $  x = (x', x'')   \in  \ar B(0,1) \cap \rn{n}_+, $   
\begin{align}  
\label{6.49a}   
s^{-1}  ( 1 + \la)^{-1}   q^{2-p} ( \nabla  u ) \,  \nabla \cdot ( D({\ts \frac{1}{p}} q^p)(\nabla  u))   =    \ti G    +    s^{-1} ( \la + 1 )^{ - 1} (  E_1 +  E_2 + E_3 + E_4)     
\end{align}
    where   
\begin{align}  
   \label{6.49}  
\begin{split}   
      \ti G    &= 
        {\ds  \frac{ (p-1) ( \la  - { \ts \frac{n-1}{p-1}} ) \la^2 s^2     +   (2p - 3)  
    ( \la -  \frac{(p+n - 3) }{ 2 p - 3 })  t^2 }{\la^2  s^2 + t^2 } }       \\
   &= {\ds \frac{ (p-1) ( \la  - { \ts \frac{n-1}{p-1}} ) \la^2 w     +   (2p - 3)  
    ( \la -  \frac{(p+n - 3) }{ 2 p - 3 }) (1 - w)  }{ (\la^2 - 1) w + 1 } } =: G(w)  
\end{split}
\end{align}    
where  we have put  $s^2 = w  =     1 - t^2 $  in the last equality.    
   Also  from \eqref{6.26} and \eqref{3.5}, \eqref{3.7}, \eqref{p1} with $ \be = \ti \be = 1, $ 
   $ \ti \la = \la  $  and   $ \lan a'', e_n  \ran = b $    we find that   
\begin{align} 
\label{6.50}  
\begin{split}
s^{-1} ( \la + 1 )^{-1}  E_1  &=  - { \ds \frac{ b ( - \la s^2 + t^2) }{\sqrt{ \la^2 s^2 + t^2}} 
       \frac{  (\la^3 s^2  +  (2 \la - 1 ) t^2 ) }{ \la^2 s^2  + t^2} }   \\ 
&  = {\ds  \frac{b  [  ( \la  + 1 )  w   -   1] [ ( \la^3  -  2 \la + 1 ) w   + 2 \la - 1]}{   [ (\la^2 - 1) w + 1]^{3/2} } }=: F_1(w).  
\end{split}
\end{align}

From  \eqref{6.27} and \eqref{3.8} we calculate 
\begin{align}
 \label{6.51} 
\begin{split} 
 s^{-1}&( \la + 1 )^{-1}  E_2  \\  
 &= {\ds 2 (p-1)  b   \frac{  (1+ \la) t^2  -  (1+ \la ) (3 + \la) s^2 t^2   +         (  (3+\la) s^2 -  3) ( 1 - (1+ \la) s^2 )}{\sqrt{\la^2 s^2  + t^2} } }    \\
& =  {\ds 2 (p-1) b  \frac{ [ \la - 2 -w(\lambda+2)(\lambda-1) ]}{\sqrt{  1+w(\lambda^2-1)}}  } =: F_2 (w).
\end{split}
\end{align}
From \eqref{6.28} we have 
\begin{align}   
\label{6.52}  
s^{-1}  ( 1 + \la)^{-1}  E_3 =  (p-1)   b^2  [ - 3 + s^2 (3+ \la) ]  = (p-1)  b^2 [ - 3 + w ( 3 + \la) ]   =:  F_3(w) . 
\end{align} 
Finally  from \eqref{6.29}, \eqref{3.5}, \eqref{3.8}  we arrive at 
\begin{align}
\label{6.53}
\begin{split}
  s^{-1}  ( 1 + \la )^{-1}  E_4 & =  \frac{b ( 1 - ( \lambda + 1 )  s^2 ) }{\sqrt{  \lambda^2  s^2  +   t^2}} (  (\lambda + 2)t^2 -s^2)+(  ( 3  + \lambda) s^2 -3  ) -(n-2) )\\ 
   &= \frac{b ( 1 - ( \la  + 1 )  s^2  )}{\sqrt{  \lambda^2 s^2 + t^2}}  (\lambda+1-n) \\
& = \left(\frac{b (1-(\lambda +1) w )}{\sqrt{ 1+w(\lambda^2-1)}}\right)  (\lambda+1-n) =: F_4 (w)
\end{split}
\end{align}

  Armed  with   \eqref{6.49} - \eqref{6.53}  we  first search  for  $ \mathcal{A} = \nabla f $-subsolutions in   the baseline  $  n = 2,  p = 2, \la = 1  $ case. 
  In this case,    $ G (w) = 0 = F_4 (w)$ for  $w \in [0,1],  $  and 
\begin{align}  
\label{6.53a}   
\begin{split}
F_1 (w )   +  F_2 (w)  +  F_3 (w)    &= b ( 2 w - 1) -  2  b   + ( 4 w - 3 ) b^2  \\ 
& = b [  2 w - 3  + ( 4 w - 3 ) b ]  >  0 
\end{split}
\end{align}  
whenever $-1 <  b <0$  and $ 0\leq w \leq 1,$ while the reverse inequality holds when $ 0 < b  < 1.$  We conclude from  \eqref{6.53a} 
\begin{corollary}  
\label{cor6.15}  
 Theorems  \ref{thmB} and \ref{thmC}  are valid in  $  \rn{n}_+ $   and 
 the   $ \mathcal{A} = \nabla f$-harmonic setting  whenever  $ k  = n - 1, p > 2,  n \geq 2,   b \in ( - 1, 0) $   and     
   $ f ( \eta )  =  |\eta|  +  b \, \eta_n$ for $\eta \in  \rn{n}. $  
\end{corollary}  

    \begin{proof} 
    From   a continuity argument we deduce  for given  $ b   \in  (-1, 0) $  the existence of   a  positive  $ \la = \la (b)  < 1 $  for which
       \eqref{6.53a} remains  positive for $ w \in [0,1]$.   From this observation, Remark  \ref{rmk6.14},  and  Proposition  \ref{prop6.7}   
       we obtain  Theorems \ref{thmB} and \ref{thmC} first  in $ \rn{2}_+ $ and then by adding dummy variables in 
   $ \rn{n}_+, n \geq 3.$    
   \end{proof}   
  Next  we  ask  for what values of  $ p, n,  b,$ (for $b \in (0, 1)$)  is   $u $ in \eqref{6.49a} an $ \mathcal{A} = \nabla f $-subsolution on $ \rn{n}_+ ?$   
    To partially answer this question first  put  $ p = 2,  n \geq 3,  \la = n - 1 $ in  \eqref{6.49}-\eqref{6.54}.   
    Then  again $ G \equiv  F_4  \equiv 0. $   Evaluating $ F_1, F_2, F_3, $  at  $ w = 0 $   we have 
\begin{align}
\label{6.54}   
\begin{split}
s^{-1}  ( 1 + \la)^{-1}   q ( \nabla  u ) \,  \nabla \cdot ( D(\frac{1}{2} q^2)(\nabla  u))  & =   - b ( 2 n - 3)  + 2b (n-3)  -  3 b^2    \\ 
& = - 3b - 3b^2  < 0 
\end{split}
\end{align}  
when $ b \in (0, 1). $        
            From   \eqref{6.54} and  a  continuity argument   we conclude that 
            $ u $ in  \eqref{6.49a}  is not an  $  \mathcal{A} = \nabla f $-subsolution  for   $ p > 2 $  and  $ \la < n - 1$ provided   $ p,  \la  $ 
          are sufficiently near  $2$,  $n-1$ respectively.            On the  other hand,  we  prove   
\begin{lemma} 
\label{lem6.16}  
Given $  b \in ( 0, 1), $  and $ p > 1 $ with $ \frac{p-2}{p-1} > 2b - b^2. $   
There exists   $ n'  = n' ( b ),$   a  positive integer, such that   if   $ n \geq n' (b) , $  then  $ u $  in \eqref{6.49a}   is  a  
           $ \mathcal{A}  =  \nabla f$-subsolution  
          on  $ \rn{n}_+ $  for some  $  \la < n -  1. $       
\end{lemma}    
\begin{proof}     
For fixed $ n \geq 3 $  let  $  \la = n - 1 $  in the definition of  $ G $ and the $  F' $s.    Then    
\begin{align} 
 \label{6.56a}    
 G(w)  &={\ds  \frac{     (p-2) \left( [  (n-1)^3   - 2  n  +  3 ] w +    ( 2n - 3)   \right) }{n (n-2) w + 1} }  
 \end{align}
 and 
 \begin{align} 
 \label{6.56b}    
 \frac{dG}{dw}  =   {\ds  \frac{   (p-2) ( - n^3 + 4 n^2 - 5n +2 ) }{ ( n (n-2) w + 1)^2} } < 0 
  \end{align}
when $w \in [0,1]$. 
Also, 
\begin{align}
\label{6.57a}   
(p-1)^{-1} F_2 ( w ) =   {\ds 2  b  \frac{ [ (n- 3 - w( n +1)( n  - 2) ]}{(n (n-2) w + 1)^{1/2}}  } 
\end{align}
and
  \begin{align}
\label{6.57b}  
\begin{split}              
               (p-1)^{-1} \frac{dF_2}{dw}    &=  {\ds     \frac{  - 2b  (n+1) (n-2)  ( n (n-2) w + 1)   - b n (n-2)   [ n - 3 - w (n+1) (n-2) ]  }{ (n ( n - 2 ) w + 1)^{3/2} }}	 \\
               &=      {\ds     \frac{  - b  (n+1) n  (n-2)^2 w      -  b (n-2) ( n^2 - n + 2)     }{ (n ( n - 2 ) w + 1)^{3/2} }}   < 0 
\end{split}
\end{align}              
for  $ w \in [0,1]$  and $n = 3, 4, \ldots$.  

Finally,  $ (p - 1)^{-1} dF_3/dw  =  (n + 2) b^2 . $  This inequality  and  \eqref{6.57a} and \eqref{6.57b}   
imply  for given $ b \in (0, 1) $   that there exists  a  positive integer   $ n_0  = n_0 (b)  $  such that    
\begin{align}  
\label{6.58} 
 d F_2/dw + d F_3/dw   <  0   \quad \mbox{for} \, \, w \in [0,1] \, \,  \mbox{and} \, \, n \geq n_0. 
\end{align} 
          To prove  this assertion suppose   $ M  \geq 1 $ is a positive number to be defined.  If  \\  $  n ( n - 2 ) w  \leq  M, $  
          then from     \eqref{6.57a} and \eqref{6.57b} we see  for $ n \geq 3 $   at $w \in [0,1] $ that     
\begin{align} 
\label{6.59}  
(p-1)^{-1}     dF_2/dw   \leq            \frac{    -  b (n-2)(n^2 - n + 1)  }{ ( M + 1)^{3/2} } 
\end{align}
  while if   $ n (n-2)  w  \geq M$ and $w \in [0,1],  $ we have 
\begin{align} 
\label{6.60}  
\begin{split}
 (p-1)^{-1} dF_2/dw   &\leq    \frac{  - b  (n+1) n  (n-2)^2 w }{ (n ( n - 2 ) w(1 + 1/M) )^{3/2} } \\
 & \leq   - b  (n+1) n^{-1/2} (n-2)^{1/2} ( 1 + 1/M )^{-3/2}.   
\end{split}
\end{align} 
   Define    $ M $  by       
   \[     
   ( 1 + 1/M )^{-3/2}  =  (3/4) + (1/4) b.  
   \]     
   With $ M $ now defined  we  see  from  \eqref{6.60}  that there exists   $ n_1 = n_1 (b),  $ a positive integer 
   such  that   if  $ n  \geq n_1,  w \in [0,1],  $  and   $ n (n - 2 ) w  \geq M, $   then 
\begin{align} 
\label{6.61}   
\begin{split}
(p-1)^{-1} [ dF_2/dw   +  dF_3/dw ]  & \leq  -  b [(3/4) + (1/4) b]  (n+1) n^{-1/2} (n-2)^{1/2}    +    b^2 (n + 2 ) \\
&  \leq  - (1/2) ( b - b^2 ) (n + 2 ) < 0. 
\end{split}
\end{align}
      Next  we  see from     \eqref{6.59}  that  there exists $ n_2 (b)  \geq n_1 (b) $  for which     
    \eqref{6.61} remains valid when $ n \geq n_2 (b),  w \in [0,1],   $  and   $ 0  \leq  n  (n - 2)  w  \leq M. $    
    
     With assertion  \eqref{6.58}  now  proved,  we note  from \eqref{6.57a} and \eqref{6.57b}  that  for $ w \in [0,1] $ 
\begin{align}
\label{6.62}  
\begin{split}
-  F_1 ( w ) & =  {\ds b  \frac{ ( 1   - n w ) [ ( ( n-1)^3 - 2 n + 3 ) w  + (2n - 3) ]}{ ( n ( n - 2) w + 1 )^{3/2}} } \\  
      &\leq  \frac{b}{ (p-2)}  G (w ) \\
      & <   \frac{ 1}{(p-1)(2 - b) } G(w)  
\end{split}
\end{align}
	               where we have used the fact that $p - 2 > ( p -1)(2  b - b^2 ) $. 
Also,    clearly  $ F_1 (  w )  \geq 0 $ on  $ [1/n,1].$    From  this fact, \eqref{6.56a},  \eqref{6.56b},  and   \eqref{6.58} 
we  obtain  for  $ n \geq n_2(b) $ and $ w \in [1/n,1], $ that 
\begin{align} 
\label{6.62a}  
\begin{split}
G(w) +   F_1 ( w ) +  F_2 (w)  +  F_3 ( w )& \geq   G(1) + F_2 ( 1 ) + F_3 ( 1 )  \\  
&=  [ p - 2   +  (p-1) ( - 2b + b^2 ) ] (n-1)  > 0  
\end{split}
\end{align}
     Next    from  \eqref{6.62}, \eqref{6.58},  and   \eqref{6.56a},  \eqref{6.56b} we have  for $ w \in [0,1/n] $   
\begin{align}
\label{6.64} 
\begin{split}
G(w) +  F_1 (w) + F_2 (w)  + F_3 (w)  & \geq          F_2 ( 1/n) + F_3 (1/n)   +  {\ts \frac{ (p - 1)(2-b) - 1}{(p-1) (2-b)}} G(1)   \\
& \geq     -  4 (p-1)   +  {\ts \frac{(p-2) ( (p - 1)(2-b)  -  1)}{(p-1)(2-b)} } (n - 1)   > 0 
\end{split}
\end{align}
   provided   $  n \geq n_3  $ and $ n_3 = n_3 (b) $ is chosen large enough.    \eqref{6.64},  \eqref{6.62a},  and  a  continuity argument imply Lemma \ref{lem6.16}.  
   \end{proof}  
    Lemma  \ref{lem6.16}  implies  
	\begin{corollary}  
	\label{cor6.17}   
	Given  $  b  \in  (0,1)$ and $ p $  with  $ \frac{p-2}{p-1} > 2b - b^2. $  
	There exists a positive integer $n' = n' (b) $   such that  Theorems \ref{thmB} and \ref{thmC} are valid in  $ \rn{n}_+$ and 
	the  $ \mathcal{A} =  \nabla f$-harmonic setting  for   $ f ( \eta ) = p^{-1}  ( |\eta|  +  b    \eta_n  )^p$ for $\eta \in \rn{n}, $     when  $ n  >  n' (b).    $  
	\end{corollary}  
	
	    If $ \de >  0 $  in \eqref{6.43}  we  can  easily get  an  analogue of  \eqref{6.38}    
	    when $ k = n - 1,   n  \geq  2, p  > 2,   $ in  $ \rn{n}_+ . $   
	    In fact we can just copy the proof  given for  Lemma  \ref{lem6.13}  with   $ \be = 1 $  and 
	    $ \la $   with   $  \chi (p, n, n-1)    <  \la < n - 1. $       
Using  this notation we get  first  \eqref{6.33}  with  $ \ti A,  \ti B,  \ti C  $  defined as  in   \eqref{3.13} 
and   with $ \la $ replaced by $ \la ( 1 + \de)$ and $ \be $ by  $ 1 + \de. $ 
After that  we simply copy the proof from \eqref{6.33} -  \eqref{6.38} 
(once again with  $ \ti \be, \ti \la $ replaced by $ 1 + \de,  \la (1+ \de )$)  
to conclude that if \eqref{6.38} holds then $ u $ is an $ \mathcal{A} $ subsolution.  
           in $ \rn{n}_+. $       Using    the analogue of \eqref{6.38}  one gets  Lemma \ref{lem6.9} 
            for a  $  \mathcal{A} =   \nabla f $-harmonic Martin function in  $ \rn{n}_+. $    After that  copying the argument in   the $ p > n $ case of   Lemma   
\ref{lem6.13},  we get first  
 \begin{lemma} 
 \label{lem6.18}  
 Assume  $ p  >  2, n  = 2,  $   and choose  $ \de $  so  that    $  (1+ \de ) \chi (p,2,1)    =   1  - \frac{ (p-2)}{4 (p-1)} . $  
 If     
\begin{align} 
\label{6.65}     
|a|   \leq     \frac{ p-2  }{ 100000 (p-1)} 
\end{align}       
then   $  u $ is an  $ \mathcal{A} $  = $ \nabla f $-subsolution  in $ \rn{2}_+ . $  
\end{lemma}    
     
After this lemma  we  obtain  once again             
\begin{corollary} 
\label{cor6.19}  
If  $ a $ satisfies    \eqref{6.65},   then  Theorems \ref{thmB} and \ref{thmC} are valid for   $ p > 2 $   in $ \rn{n}_+ $   
and the  $ \mathcal{A} = \nabla f $  setting for    $  f ( \eta )   = p^{-1} ( | \eta | +  \lan a, \eta ))^p$ for $\eta \in \rn{2}.  $    
\end{corollary}

\begin{proof}  
The proof  follows in $ \rn{2}_+ $  from Lemma \ref{lem6.18} and the  analogue of  Lemma \ref{lem6.11} in $ \rn{2}_+. $      
To get  a  proof in $ \rn{n}_+, n > 2,  $  extend the solution in  $ \rn{2}_+ $  to  $ \rn{n}_+ $ by adding dummy variables.      
 \end{proof}  

\subsection{Final Remarks}  	 
We began  our investigation of  the  exponent for  $ \mathcal{A}   =  \nabla f$-harmonic Martin   functions   in $ \rn{2}_+ $  when    $ p  = 2 $    with  
\[  
 f (\eta) = 2^{-1}  q(\eta)^2   \quad \mbox{for} \, \, \eta \in \rn{2} 
 \]    
where $q (0) = 0$, smooth,   and $1$-homogeneous on $\rn{2} \sem \{0\}$. We  assumed an  $ \mathcal{A}$-harmonic Martin function, $u$,    on $ \rn{2}_+  ,$   to have the form:  
\[  
u ( x_1,  x_2 )       =    \frac{  x_2}{ l (x_1, x_2 )^{1 + \la}}
  \]   
where $ \la > 0$, $  l  > 0 $  is smooth, $1$-homogeneous on   $ \rn{2} \sem \{0\}$  which was the form dictated by  
	 \eqref{1.4} and   the boundary Harnack  inequalities in Lemma \ref{lem2.6}  for  $ \mathcal{A}$-harmonic functions.   
	 Using  $ \mathcal{A}$-harmonicity of $ u $ and   the  homogeneities,  we  wrote down  a   fully nonlinear  
	 second order differential equation for $ l $  involving   $q,  q_{\eta_1},  q_{\eta_2}, q_{\eta_1 \eta_2} $  
	 evaluated at  $ \left( - \la x_2 l_{x_1}  (x_1, x_2),    (l -  \la x_2  l_{x_2}) (x_1, x_2) \right). $ 
	  Again taking  $ q ( \eta ) = | \eta | +  \lan a, \eta \ran$  and  $ x_1^2 + x_2^2 = 1, $  we obtained  upon  letting $ x_2 \to  0^+ $ 
	  the  necessary condition  
\begin{align} 
\label{6.66}   
\lan \nabla q ( 0, 1 ),  \nabla l (x, 0 ) \ran \leq 0,  \quad \mbox{for} \, \, - 1  \leq  x  \leq 1,  
\end{align}
   for $ u $ to be an  $ \mathcal{A}$-subsolution on  $ \rn{2}_+ $   while   the reverse inequality was  necessary  for $ u $ to be an $\mathcal{A}$-supersolution. 
    \eqref{6.66} for example, showed that   if  $  l (x_1,x_2) = \sqrt{x_1^2 + x_2^2} , $  and  $ a = (b, 0 ) , $  then $  u $  could not be  a   
	  $ \mathcal{A}$-subsolution or  supersolution   on $ \rn{2}_+ . $   We also  let  $  x_2 = 1$ in our  differential equation  for $ l $  
	  and obtained  a  rather complicated equation for $ l_{x_1} (0,1), l_{x_2}(0, 1), $  which however greatly simplified if  
	  $ l_{x_1} (0, 1 ) = 0. $    Assuming this equality  we  were able to check  without much difficulty  that   $ u $  was a Martin subsolution
	   for some  $0 <  \la < 2 $  at  (0,1)  if $ l (x_1, x_2 ) =  \sqrt{ x_1^2 + x_2^2 }$  and 
	  $ a  =  ( 0, b ) $ with $ b < 0. $  Next we asked Maple  to   calculate  and plot  the graph of 
	  \[ 
	  x_2 =    \,  \nabla \cdot ( D({\ts \frac{1}{2}} q^2)(\nabla  u)) ( x_1^2 , 1 - x_1^2 )  \quad \mbox{ for }  x_1 \in (0, 1 )  
	  \]
	  when $   \lan  a, e_1 \ran \neq 0  $  and  for several different choices of $ l $  other than $ l  = \sqrt{x_1^2 + x_2^2}.  $     
	   Maple  had a  difficult time   with this.   
	  	  However   when  $  l  =   \sqrt{x_1^2 + x_2^2} $ and  $ a= (0, b), $  Maple plots gave strong indications  that $u$  was a Martin subsolution 
	  	  when $ b < 0 $ for some $0 <  \la < 1 $  and a  Martin supersolution for 
	  some $ \la > 1 $  when $  b > 0. $   This result went against our intuition, as  it   did  not  seem to  depend on uniform ellipticity  constants for $ f $ in  \eqref{6.3} $(b). $   
	  However thanks to  Maple we eventually obtained  \eqref{6.53a} and  Corollary \ref{cor6.15}.
	  
	   Finally the ballpark estimates  given for $ |a| $ in  Lemmas  \ref{lem6.13}  and \ref{lem6.18}   
	   could definitely be  improved on by a more serious estimate of the  $  E $'s.  Also  we note  that  the Martin exponent
	  for  $ p > 1 $ and a $p$-harmonic Martin function on $ \rn{2}_+ $ is  (see \cite{K,A})	  
	  \begin{align*} 
     (1/3) \left( p - 3  - 2  \sqrt{ p^2 - 3p + 3 } \,  \, \,  \, \right)  / (p - 1 ) 
     \end{align*}    	
and the Martin function can be written down more or less explicitly. 
Using this fact and  arguing as in the proof of \eqref{6.38},  \eqref{6.65},  with  $ u $ replaced by the Martin function,      
one   should be able to  get  a  better estimate in terms of 	$ |a| $  for the  exponent of  an $ A = \nabla f $-harmonic Martin  function
  on $ \rn{2}_+ $ when $ p > 2 $  and $ f ( \eta )  =  p^{-1} ( |\eta|| + \lan a, \eta \ran )^p. $

 \nocite{*}
 
 \bibliographystyle{amsbeta}
 \bibliography{myref}	

\newcommand{\etalchar}[1]{$^{#1}$}
\providecommand{\bysame}{\leavevmode\hbox to3em{\hrulefill}\thinspace}
\providecommand{\MR}{\relax\ifhmode\unskip\space\fi MR }
\providecommand{\MRhref}[2]{%
  \href{http://www.ams.org/mathscinet-getitem?mr=#1}{#2}
}
\providecommand{\href}[2]{#2}
\begin{thebibliography}{LMTW19}

\bibitem[AGH{\etalchar{+}}17]{AGHLV}
M.~Akman, J.~Gong, J.~Hineman, J.~Lewis, and A.~Vogel, \emph{The
  {B}runn-{M}inkowski inequality and {A} {M}inkowski problem for nonlinear
  capacity}, To appear in Memoirs of the AMS, arXiv:1709.00447 ((2017)).

\bibitem[Akm14]{Akman}
Murat Akman, \emph{On the dimension of a certain measure in the plane}, Ann.
  Acad. Sci. Fenn. Math. \textbf{39} (2014), no.~1, 187--209. \MR{3186813}

\bibitem[ALSV21]{ALSV}
Murat Akman, John Lewis, Olli Saari, and Andrew Vogel, \emph{The
  {B}runn-{M}inkowski inequality and a {M}inkowski problem for
  {$\mathcal{A}$}-harmonic {G}reen's function}, Adv. Calc. Var. \textbf{14}
  (2021), no.~2, 247--302. \MR{4236203}

\bibitem[ALV17]{ALV1}
Murat Akman, John Lewis, and Andrew Vogel, \emph{{$\sigma$}-finiteness of
  elliptic measures for quasilinear elliptic {PDE} in space}, Adv. Math.
  \textbf{309} (2017), 512--557. \MR{3607285}

\bibitem[ALV20]{ALV}
Murat Akman, John Lewis, and Andrew Vogel, \emph{On a theorem of {W}olff
  revisited}, To appear in Journal d'Analyse Math{\`e}matique (2020).

\bibitem[Aro86]{A}
Gunnar Aronsson, \emph{Construction of singular solutions to the {$p$}-harmonic
  equation and its limit equation for {$p=\infty$}}, Manuscripta Math.
  \textbf{56} (1986), no.~2, 135--158. \MR{850366}

\bibitem[DS18]{DB}
Dante DeBlassie and Robert~G. Smits, \emph{The {$p$}-harmonic measure of small
  axially symmetric sets}, Potential Anal. \textbf{49} (2018), no.~4, 583--608.
  \MR{3859537}

\bibitem[Eva10]{E}
Lawrence~C. Evans, \emph{Partial differential equations}, second ed., Graduate
  Studies in Mathematics, vol.~19, American Mathematical Society, Providence,
  RI, 2010. \MR{2597943}

\bibitem[HKM06]{HKM}
Juha Heinonen, Tero Kilpel{\"a}inen, and Olli Martio, \emph{Nonlinear potential
  theory of degenerate elliptic equations}, Dover Publications Inc., 2006.

\bibitem[Kro73]{K}
I.~N. Krol', \emph{The behavior of the solutions of a certain quasilinear
  equation near zero cusps of the boundary}, Trudy Mat. Inst. Steklov.
  \textbf{125} (1973), 140--146, 233, Boundary value problems of mathematical
  physics, 8. \MR{0344671}

\bibitem[Lew88]{L1}
John~L. Lewis, \emph{Note on a theorem of {W}olff}, Holomorphic functions and
  moduli, {V}ol. {I} ({B}erkeley, {CA}, 1986), Math. Sci. Res. Inst. Publ.,
  vol.~10, Springer, New York, 1988, pp.~93--100. \MR{955811}

\bibitem[LLN08]{LLN}
John~L. Lewis, Niklas Lundstr\"{o}m, and Kaj Nystr\"{o}m, \emph{Boundary
  {H}arnack inequalities for operators of {$p$}-{L}aplace type in {R}eifenberg
  flat domains}, Perspectives in partial differential equations, harmonic
  analysis and applications, Proc. Sympos. Pure Math., vol.~79, Amer. Math.
  Soc., Providence, RI, 2008, pp.~229--266. \MR{2500495}

\bibitem[LMTW19]{LMTW}
Jos\'{e}~G. Llorente, Juan~J. Manfredi, William~C. Troy, and Jang-Mei Wu,
  \emph{On {$p$}-harmonic measures in half-spaces}, Ann. Mat. Pura Appl. (4)
  \textbf{198} (2019), no.~4, 1381--1405. \MR{3987219}

\bibitem[LN18]{LN}
John Lewis and Kaj Nystr\"{o}m, \emph{Quasi-linear {PDE}s and low-dimensional
  sets}, J. Eur. Math. Soc. (JEMS) \textbf{20} (2018), no.~7, 1689--1746.
  \MR{3807311}

\bibitem[Ser64]{S}
James Serrin, \emph{Local behavior of solutions of quasi-linear equations},
  Acta Mathematica \textbf{111} (1964), 247--302, 10.1007/BF02391014.

\bibitem[Tol84]{T}
Peter Tolksdorf, \emph{Regularity for a more general class of quasilinear
  elliptic equations}, J. Differential Equations \textbf{51} (1984), no.~1,
  126--150. \MR{727034}

\bibitem[Var15]{V}
Harri Varpanen, \emph{On a linearized {$p$}-{L}aplace equation with rapidly
  oscillating coefficients}, Illinois J. Math. \textbf{59} (2015), no.~2,
  499--529. \MR{3499522}

\bibitem[Wol07]{W}
Thomas~H. Wolff, \emph{Gap series constructions for the {$p$}-{L}aplacian}, J.
  Anal. Math. \textbf{102} (2007), 371--394, Paper completed by John Garnett
  and Jang-Mei Wu. \MR{2346563}

\end{thebibliography}
 
	             \end{document}